\documentclass[a4paper,12pt]{article} 
\usepackage[utf8]{inputenc}
\usepackage[T1]{fontenc}
\usepackage{amsmath,amssymb,amsthm}
\usepackage{authblk}
\usepackage{mathabx}
\usepackage{bbm}
\usepackage{mathrsfs}
\usepackage[normalem]{ulem}
\usepackage[numeric,initials,nobysame,msc-links,abbrev]{amsrefs}
\renewcommand{\eprint}[1]{\href{https://arxiv.org/abs/#1}{arXiv:#1}}
\newcommand{\pageafter}[1]{#1~pp.}
\BibSpec{article}{%
+{} {\PrintAuthors} {author}
+{,} { \textit} {title}
+{.} { } {part}
+{:} { \textit} {subtitle}
+{,} { \PrintContributions} {contribution}
+{.} { \PrintPartials} {partial}
+{,} { } {journal}
+{} { \textbf} {volume}
+{} { \PrintDatePV} {date}
+{,} { \issuetext} {number}
+{,} { \pageafter} {pages}
+{,} { } {status}
+{,} { \PrintDOI} {doi}
+{,} { available at \eprint} {eprint}
+{} { \parenthesize} {language}
+{} { \PrintTranslation} {translation}
+{;} { \PrintReprint} {reprint}
+{.} { } {note}
+{.} {} {transition}
+{} {\SentenceSpace \PrintReviews} {review}
}
\BibSpec{collection.article}{%
+{} {\PrintAuthors} {author}
+{,} { \textit} {title}
+{.} { } {part}
+{:} { \textit} {subtitle}
+{,} { \PrintContributions} {contribution}
+{,} { \PrintConference} {conference}
+{} {\PrintBook} {book}
+{,} { } {booktitle}
+{,} { \PrintDateB} {date}
+{,} { \pageafter} {pages}
+{,} { } {status}
+{,} { \PrintDOI} {doi}
+{,} { available at \eprint} {eprint}
+{} { \parenthesize} {language}
+{} { \PrintTranslation} {translation}
+{;} { \PrintReprint} {reprint}
+{.} { } {note}
+{.} {} {transition}
+{} {\SentenceSpace \PrintReviews} {review}
}
\usepackage{verbatim}
\usepackage{mathtools}
\usepackage{bbm}
\usepackage{enumitem}
\usepackage{fullpage}
\usepackage{subcaption}

\usepackage{xcolor}
\hypersetup{
    colorlinks,
    linkcolor={red!50!black},
    citecolor={blue!90!black},
    urlcolor={blue!80!black}
}

\usepackage{tikz} 
\usetikzlibrary {arrows.meta}
\usetikzlibrary{shapes.geometric}
\usetikzlibrary{bending}

\usepackage[capitalise]{cleveref}
\crefformat{equation}{(#2#1#3)}
\newcommand{\crefdefpart}[2]{%
  \hyperref[#2]{\namecref{#1}~\labelcref*{#1}~\ref*{#2}}%
}
\setlist[itemize]{leftmargin=*}
\setlist[enumerate]{leftmargin=*,label=(\roman*),ref=(\roman*)}

\newtheorem{theorem}{Theorem}
\crefname{theorem}{Theorem}{Theorems}
\newtheorem{corollary}[theorem]{Corollary}
\crefname{corollary}{Corollary}{Corollaries}
\newtheorem{lemma}[theorem]{Lemma}
\crefname{lemma}{Lemma}{Lemmas}
\newtheorem{proposition}[theorem]{Proposition}
\crefname{proposition}{Proposition}{Propositions}

\crefname{conjecture}{Conjecture}{Conjectures}
\newtheorem{question}{Question}
\crefname{question}{Question}{Questions}
\theoremstyle{definition}
\newtheorem{definition}[theorem]{Definition}
\crefname{definition}{Definition}{Definitions}
\newtheorem{remark}[theorem]{Remark}
\crefname{remark}{Remark}{Remarks}

\crefname{example}{Example}{Examples}
\newtheorem{observation}[theorem]{Observation}
\crefname{observation}{Observation}{Observations}

\crefname{claim}{Claim}{Claims}

\crefname{assumption}{Assumption}{Assumptions}

\numberwithin{theorem}{section}

\newcommand{\cC}{\ensuremath{\mathcal C}}
\newcommand{\cD}{\ensuremath{\mathcal D}}
\newcommand{\cE}{\ensuremath{\mathcal E}}
\newcommand{\cF}{\ensuremath{\mathcal F}}
\newcommand{\cG}{\ensuremath{\mathcal G}}

\newcommand{\cI}{\ensuremath{\mathcal I}}

\newcommand{\cM}{\ensuremath{\mathcal M}}
\newcommand{\cN}{\ensuremath{\mathcal N}}
\newcommand{\cO}{\ensuremath{\mathcal O}}

\newcommand{\cR}{\ensuremath{\mathcal R}}

\newcommand{\cT}{\ensuremath{\mathcal T}}

\def\clap#1{\hbox to 0pt{\hss#1\hss}}

\def\mathclap{\mathpalette\mathclapinternal}

\def\mathclapinternal#1#2{\clap{$\mathsurround=0pt#1{#2}$}}

\newcommand{\bbC}{{\ensuremath{\mathbb C}} }

\newcommand{\bbM}{{\ensuremath{\mathbb M}} }
\newcommand{\bbN}{{\ensuremath{\mathbb N}} }

\newcommand{\bbP}{{\ensuremath{\mathbb P}} }

\newcommand{\bbR}{{\ensuremath{\mathbb R}} }

\newcommand{\bbZ}{{\ensuremath{\mathbb Z}} }

\newcommand{\1}{{\ensuremath{\mathbbm{1}}} }

\newcommand{\md}{\mathrm{d}}
\DeclareMathOperator{\sh}{short}
\DeclareMathOperator{\lng}{long}
\newcommand{\tf}{\tau^{\mathrm F}}

\renewcommand{\leq}{\leqslant}
\renewcommand{\geq}{\geqslant}
\renewcommand{\le}{\leqslant}
\renewcommand{\ge}{\geqslant}

\title{Bootstrap percolation is local}
\author[1]{Ivailo Hartarsky}
\author[2]{Augusto Teixeira}
\affil[1]{\small Technische Universit\"at Wien, Institut für Stochastik und Wirtschaftsmathematik, Wiedner Hauptstra\ss e 8-10, A-1040, Vienna, Austria
\texttt{ivailo.hartarsky@tuwien.ac.at}}
\affil[2]{\small IMPA, Estrada Dona Castorina, 110 - Rio de Janeiro, Brazil and IST, Av. Rovisco Pais, Nº 1, Lisbon, Portugal \texttt{augusto@impa.br}}
\date{\today}

\begin{document}
\maketitle

\begin{abstract}
Metastability thresholds lie at the heart of bootstrap percolation theory. Yet proving precise lower bounds is notoriously hard. We show that for two of the most classical models, two-neighbour and Frob\"ose, upper bounds are sharp to essentially arbitrary precision, by linking them to their local counterparts.

In Frob\"ose bootstrap percolation, iteratively, any vertex of the square lattice that is the only healthy vertex of a $1\times1$ square becomes infected and infections never heal. We prove that if vertices are initially infected independently with probability $p\to0$, then with high probability the origin becomes infected after
\[\exp\left(\frac{\pi^2}{6p}-\frac{\pi\sqrt{2+\sqrt2}}{\sqrt p}+\frac{O(\log^2(1/p))}{\sqrt[3]p}\right)\]
time steps. We achieve this by proposing a new paradigmatic view on bootstrap percolation based on locality. Namely, we show that studying the Frob\"ose model is \emph{equivalent} in an extremely strong sense to studying its local version. As a result, we completely bypass Holroyd's classical but technical hierarchy method, yielding the first term above and systematically used throughout bootstrap percolation for the last two decades. Instead, the proof features novel links to large deviation theory, eigenvalue perturbations and others.

We also use the locality viewpoint to resolve the so-called bootstrap percolation paradox. Indeed, we propose and implement an exact (deterministic) algorithm which exponentially outperforms previous Monte Carlo approaches. This allows us to clearly showcase and quantify the slow convergence we prove rigorously.

The same approach applies, with more extensive computations, to the two-neighbour model, in which vertices are infected when they have at least two infected neighbours and do not recover. We expect it to be applicable to a wider range of models and correspondingly conclude with a number of open problems.
\end{abstract}


	\tableofcontents

\section{Introduction}
\label{sec:intro}
\subsection{Background}
Bootstrap percolation is a statistical mechanics model branded in 1979 by Chalupa, Leath and Reich \cite{Chalupa79} (also see \cites{Pollak75,Kogut81}). It has since become classical and widespread not only in mathematics and statistical physics, but also in computer and social sciences. The model is so simple and natural to define that it has been reinvented and studied under a number of names, including freezing majority rule \cite{Maldonado18}, $k$-core model \cite{Guggiola15}, jamming percolation \cite{Toninelli07}, diffusion percolation \cite{Adler88}, dynamic monopoly \cite{Zehmakan19}, contagious sets \cite{Guggiola15}, target set selection \cite{Banerjee20} and many more. Bootstrap percolation admits fruitful links with low-temperature stochastic Ising model \cite{Cerf13}, kinetically constrained models of the liquid-glass transition \cite{Hartarsky23FA}, weak saturation in graph theory \cite{Balogh12a}, stability of perturbations of cellular automata \cite{Hartarsky22sharpness} and others. We direct the interested reader to \cites{Morris17,Morris17a,Hartarsky22phd,Adler91,DeGregorio09} for background and more references on the above.

The $r$-neighbour bootstrap percolation model on a graph $G$ is defined as follows (see \cref{subsec:models} for a more formal definition). Initially some of the vertices of $G$ are declared infected. Then, at each discrete time step, one additionally infects each vertex having at least $r$ infected neighbours, while infections never heal. This cellular automaton has been studied on a variety of graphs (e.g.\ high-dimensional hypercubes \cites{Balogh06,Balogh09,Balogh10}, trees \cites{Balogh06a,Biskup09,Bollobas14,Bradonjic15,Fontes08,Gunderson14,Shapira19a}, random regular graphs \cites{Janson09,Balogh07}, Erd\H os--R\'enyi graphs \cites{Angel21,Janson12,Kang16,Torrisi19}, hyperbolic lattices \cite{Sausset10}, Hamming tori \cites{Gravner15,Slivken14}, as well as graphs with ``real world'' features such as sparsity, small diameter, heavy-tailed degree distributions, community structure, etc.\ \cites{Abdullah18,Alves22,Amini10,Amini14,Bradonjic14,Candellero16,Falgas-Ravry23,Fountoulakis18,Gao15,Torrisi23,Whittemore21}) and under a variety of initial conditions. Yet, the most natural and classical setting from the statistical mechanics viewpoint is the following. We take the graph $G$ to be the lattice $\mathbb Z^d$ with edges between points at Euclidean distance 1 and initial condition given by infecting each vertex at random independently with probability $p\in(0,1)$. Thus, the model is completely defined by its \emph{dimension} $d$, \emph{threshold} $r$ and \emph{parameter} $p$. One of the most natural observables is the \emph{infection time}: the random variable $\tau$ given by the (possibly infinite) first time when the origin becomes infected.\footnote{It is also common to consider the finite-volume critical probability. That is, the smallest value of $p$ such that the probability that all vertices the $d$-dimensional torus $G=(\mathbb Z/n\mathbb Z)^d$ become infected eventually. All results we discuss admit essentially equivalent formulations in terms of this critical probability along the lines of \cref{lem:tau}.}

Many variants of the $r$-neighbour model have been introduced, some of which are very similarly behaved. Of particular relevance to us is Frob\"ose bootstrap percolation (on $\mathbb Z^2$) introduced by Frob\"ose in 1989 \cite{Frobose89}. In this model each vertex becomes infected if it completes a $1\times 1$ square of four infections and infections never heal, see \cref{eq:def:FBP} for a more formal definition. As we will see, this model is essentially the same as the two-neighbour one on $\mathbb Z^2$, but technically a bit simpler. We denote the corresponding infection time by~$\tf$. Yet another popular variant is \emph{modified two-neighbour bootstrap percolation}, where two opposite corners of a $1\times1$ square are sufficient to infect the other two corners.

In the initial work \cite{Chalupa79}, $r$-neighbour bootstrap percolation was studied on a regular tree and observed to exhibit a non-trivial phase transition in the sense that $\tau<\infty$ almost surely if and only if $p$ is larger than some critical value bounded away from 0 and 1 (see \cite{Adler88} and the references therein, also \cite{Frobose89} for Frob\"ose bootstrap percolation). Based on simulations, it was initially believed that for certain values of $d$ and $r$ this would also be the case for $r$-neighbour bootstrap percolation on $\mathbb Z^d$. However, this was disproved by van Enter \cite{VanEnter87} for $d=2$ and Schonmann \cite{Schonmann92} for any $d$, showing that the critical parameter is~$0$ for $r\le d$ and $1$ for $r>d$. Consequently, the same holds for the Frob\"ose and modified models. This mismatch between what simulations indicated and the true behaviour of the model is what became known as \emph{the bootstrap percolation paradox}, which has manifested in other erroneous predictions as explained below.

The trivial phase transition was first quantified in the seminal work of Aizenman and Lebowitz \cite{Aizenman88}, who established that for $d=r=2$ there exists $C>0$ such that
\begin{equation}
\label{eq:AL:result}
\lim_{p\to 0}\mathbb P_p\left(\exp\left(\frac{1}{Cp}\right)<\tau<\exp\left(\frac{C}{p}\right)\right)=1,
\end{equation}
the same proof applying to $\tf$ and modified bootstrap percolation. Indeed, the cause of the discrepancy between the numerical and rigorous results were strong finite-size effects. They also pointed out the metastable nature of bootstrap percolation---infection is triggered by the presence of extremely rare ``critical droplets'' which grow easily to invade all space--a phenomenon known to govern the nucleation of a crystal within a metastable liquid phase. Subsequently, suitable analogues of \cref{eq:AL:result} were established for all $r\le d$ by Cerf, Cirillo and Manzo \cites{Cerf99,Cerf02}. More recently, analogues of \cref{eq:AL:result} were proved for a vast class of cellular automata by Bollob\'as, Duminil-Copin, Morris and Smith \cite{Bollobas23} in two dimensions and somewhat weaker versions thereof by Balister, Bollob\'as, Morris and Smith \cites{Balister22,BalisterNaNb} in any dimension.

A breakthrough was made by Holroyd \cite{Holroyd03}, who improved \cref{eq:AL:result} to a sharp threshold:
\begin{equation}
    \label{eq:Holroyd}
\lim_{p\to 0}\mathbb P_p\left(\exp\left(\frac{\lambda_1-\varepsilon}{p}\right)<\tau<\exp\left(\frac{\lambda_1+\varepsilon}{p}\right)\right)=1
\end{equation}
for $\lambda_1=\pi^2/18$ and any $\varepsilon>0$ in the case $r=d=2$. For $\tf$, \cref{eq:Holroyd} holds for $\lambda_1^{\mathrm F}=\pi^2/6$ (the proof of \cite{Holroyd03}*{Theorem 4} applies without change) and similarly for modified bootstrap percolation with $\lambda_1^{\mathrm F}$.\footnote{See \cites{Bringmann12,Holroyd03a} for links between the constants appearing here and integer partitions in combinatorial number theory.} Beyond the importance of the result itself, \cite{Holroyd03} introduced, among others, the fundamental technique of ``hierarchies''. It immediately became the method of choice in bootstrap percolation and has been used in virtually every paper on bootstrap percolation (within the general ``critical'' class studied in \cite{Bollobas23}) with random initial condition in the last two decades. Similarly to \cref{eq:AL:result}, generalisations of \cref{eq:Holroyd} were also intensively sought after. Consequently, several other sharp thresholds have been proved \cites{Duminil-Copin13,Bollobas17,Holroyd03,Holroyd03a}, including a rather general class of models (albeit not as wide as the ones treated in \cite{Bollobas23} at the level of precision of \cref{eq:AL:result}) by Duminil-Copin and the first author \cite{Duminil-Copin23}, as well as the $r$-neighbour model on $\mathbb Z^d$ for all $r\le d$ by Balogh, Bollob\'as, Duminil-Copin and Morris \cites{Balogh12}.

Another reason for the importance of \cref{eq:Holroyd} is its relation with the aforementioned bootstrap percolation paradox. As already indicated, it is very natural to run Monte Carlo simulations to make quantitative predictions for the behaviour of the infection time $\tau$. This was done in order to determine whether the phase transition is trivial, what the correct scaling in \cref{eq:AL:result} is, what the correct value of $\lambda_1$ in \cref{eq:Holroyd} is, what the correct scaling of the error term in \cref{eq:Holroyd} is, etc. The paradox lies in the fact that these predictions have so far systematically been wrong, regardless how detailed rigorous results are taken into account to refine them. An early account of the paradox can be found in \cite{VanEnter90}, while subsequent reassessments include \cites{Gravner08,DeGregorio04,DeGregorio09}.

The large discrepancy between the numerical estimates of $\lambda_1$ \cites{Lenormand84,Nakanishi86,Adler91,Adler89} (also see \cite{Teomy14} for more contemporary simulations) and its actual value \cite{Holroyd03} motivated the quantification of the error term in \cref{eq:Holroyd}. The best upper bound was obtained by Gravner and Holroyd \cite{Gravner08} showing that for some $C>0$
\begin{equation}
\label{eq:GH}\lim_{p\to 0}\mathbb P_p\left(\tau<\exp\left(\frac{\lambda_1}{p}-\frac{1}{C\sqrt p}\right)\right)=1\end{equation}
and a similar argument applies to $\tf$ and $\lambda_1^{\mathrm{F}}$ and to modified bootstrap percolation (see \cite{Uzzell19} for an analogue of \cref{eq:GH} for any $r\le d$). 

Lower bounds were substantially harder to come by. First Gravner and Holroyd \cite{Gravner09} focused on a simplified version of the model called \emph{local two-neighbour bootstrap percolation}, which we discuss next, establishing that for some $C>0$
\begin{equation}
\label{eq:GH:loc}
\lim_{p\to 0}\mathbb P_p\left(\tau_{\mathrm{loc}}>\exp\left(\frac{\lambda_1}{p}-
\frac{\log^{C}(1/p)}{\sqrt p}\right)\right)=1,\end{equation}
where $\tau_{\mathrm{loc}}$ is the infection time for the local model, and similarly for Frob\"ose and modified bootstrap percolation.

In local two-neighbour bootstrap percolation (see \cref{subsec:models} for precise definitions and \cites{Gravner09,DeGregorio05,DeGregorio06,Bringmann12} for works on it) one of the initially infected sites (which are still chosen independently at random with probability $p$) is declared germed; germs are transmitted at each step to all infected neighbours; healthy sites with at least two infected neighbours, at least one of which is germed, become infected. In other words, infection spreads exactly as in ordinary ($r=2$) bootstrap percolation, but only starting from one particular spot. It is clear that (ordinary) bootstrap percolation infects the origin at least as fast as its local counterpart and for this reason all upper bounds in bootstrap percolation (again, more generally, within the critical universality class studied in \cite{Bollobas23}) actually (implicitly) work with a local model instead. It turns out that having to control only one growing patch of infection is significantly simpler than also accounting for the possibility that several such patches of various sizes and relative positions conspire in order to grow. Thanks to this, the exponent $C$ in \cref{eq:GH:loc} was slightly improved by Bringmann and Mahlburg \cite{Bringmann12}, where a few closely related models, including local Frob\"ose and modified bootstrap percolation, were also treated.

Equation \cref{eq:GH:loc} was proved for $\tau$ instead of $\tau_{\mathrm{loc}}$ by Gravner, Holroyd and Morris \cite{Gravner12} by combining \cite{Gravner09} with a refinement of the hierarchy method of \cite{Holroyd03}. Significantly developing this line of work, Morris and the first author \cite{Hartarsky19} matched \cref{eq:GH} by a lower bound to obtain that for $r=d=2$ there exists $C>0$ such that 
\begin{equation}
\label{eq:second}
\lim_{p\to 0}\mathbb P_p\left(\exp\left(\frac{\lambda_1}{p}-\frac{C}{\sqrt p}\right)<\tau<\exp\left(\frac{\lambda_1}{p}-\frac{1}{C\sqrt p}\right)\right)=1.
\end{equation}
Like previous works, this discarded conjectures based on numerical simulations \cite{Teomy14} (see \cites{Gravner12} for more). Analogues of the lower bound in \cref{eq:second} have not yet been obtained for any other model, perhaps owing to the level of technicality of the proof. Finally, the first author \cite{Hartarsky23mod-2n} recently noted that for the modified model \cref{eq:second} actually fails, due to the presence of an additional logarithmic factor in the second order term in the upper bound. This marked the first divergence between the modified and original models (see \cref{subsec:modified} for more on this matter), while the state of Frob\"ose bootstrap percolation remained unsettled.

\subsection{Results}
\subsubsection{Locality}
\label{subsubsec:locality}
The most important novelty of our work is to propose a new viewpoint on bootstrap percolation models based on \emph{locality}. The idea is to establish that the infection time in bootstrap percolation is not only upper bounded by its local counterpart, but that studying the local model is actually \emph{equivalent} to treating the non-local one. Namely, we prove the following (see \cref{cor:reduction} for the technical version we actually use, while the aesthetic one stated here follows directly from \cref{prop:reduction,lem:tau}).
\begin{theorem}[Locality]
\label{th:locality}
For some absolute constant $C>0$ we have
\begin{equation}
\label{eq:locality}\lim_{p\to 0}\mathbb P_p\left(1\le \frac{\tau_{\mathrm{loc}}}{\tau}\le \exp\left(\log^{C}(1/p)\right)\right)=1,
\end{equation}
where $\tau$ (resp.\ $\tau_{\mathrm{loc}}$) is the infection time in two-neighbour bootstrap percolation (resp.\ local) on $\mathbb Z^2$ defined in \cref{eq:def:tau} (resp.\ \cref{eq:def:tauloc}). The same holds for $\tf_{\mathrm{loc}}/\tf$ corresponding to (local) Frob\"ose bootstrap percolation.
\end{theorem}

It is important to note that, while it follows e.g.\ from \cref{eq:GH:loc,eq:second} that $\tau$ and $\tau_{\mathrm{loc}}$ are close (up to a factor of $\exp(C/\sqrt{p})$), our approach is radically different. Namely, we do not recover \cref{th:locality} as a corollary of precise asymptotics on the two quantities involved, but rather prove it \emph{a priori}, only relying on rather crude bounds corresponding to \cref{eq:Holroyd}, thus completely bypassing \cref{eq:GH,eq:GH:loc,eq:second}. This results in the bound in \cref{th:locality} being much stronger than the precision of \cref{eq:second} and the more precise results we show below. Thus, \cref{th:locality} allows one to completely restrict one's attention to the local model also in potential future works essentially all the way down to the critical window (see \cite{Hartarsky22phd}*{Proposition 1.4.3}).

Another important consequence of \cref{th:locality} is completely bypassing Holroyd's hierarchy method, on whose refinements previous works (e.g.\ \cites{Gravner12,Hartarsky19}) rely. Indeed, hierarchies do \emph{not} appear directly in the present paper in any form. Instead, the rather short proof of \cref{th:locality} itself involves some novel ingredients such as tools from large deviation theory (see \cref{subsec:a:priori}).

\subsubsection{Sharp asymptotics}
\label{subsubsec:main}
Our remaining results show that the locality viewpoint tremendously simplifies the study of the models of interest, enabling us to go well beyond previous results. The first such application concerns the classical problem of determining the precise asymptotics of the infection time. Namely, we significantly improve on \cref{eq:GH,eq:GH:loc} for the Frob\"ose model, establishing in particular that \cref{eq:second} does hold for it.
\begin{theorem}
\label{th:main}
The infection time $\tf$ of Frob\"ose bootstrap percolation satisfies\footnote{We have not sought to optimise the power of the logarithm in \cref{eq:main}, since we do not expect $\sqrt[3]p$ to be the right order of the next term.}
\begin{equation}
\label{eq:main}
\lim_{p\to 0}\mathbb P_p\left(\exp\left(\frac{\lambda_1^{\mathrm{F}}}{p}-\frac{\lambda_2^{\mathrm{F}}}{\sqrt p}-\frac{\log^{2}(1/p)}{\sqrt[3]p}\right)<\tf<\exp\left(\frac{\lambda_1^{\mathrm{F}}}{p}-\frac{\lambda_2^{\mathrm{F}}}{\sqrt p}+\frac{\log^{2}(1/p)}{\sqrt[3]p}\right)\right)=1\end{equation}
with $\lambda_1^{\mathrm{F}}=\pi^2/6\approx1.6449$ and $\lambda_2^{\mathrm{F}}=\pi\sqrt{2+\sqrt 2}\approx5.8049$.
\end{theorem}
We also establish an analogous result for two-neighbour bootstrap percolation. While its proof is conceptually very similar, it requires significantly more extensive casework. Most notably, \cref{tab:transitions} becomes a 221-row table, an excerpt of which is provided in \cref{tab:transitions:2n}, each row corresponding to an event like the seven given in \cref{eq:def:transitions:2n}. For this reason, we leave the additional technical details in the proof of \cref{th:2n} below to \cref{appendix} and warn the prudent reader that only the upper bound will be proved in detail. 
\begin{theorem}
\label{th:2n}
For two-neighbour bootstrap percolation on $\mathbb Z^2$, we have
\begin{equation}
\label{eq:main:2n}
\lim_{p\to 0}\mathbb P_p\left(\exp\left(\frac{\lambda_1}{p}-\frac{\lambda_2}{\sqrt p}-\frac{\log^{2}(1/p)}{\sqrt[3]p}\right)<\tau<\exp\left(\frac{\lambda_1}{p}-\frac{\lambda_2}{\sqrt p}+\frac{\log^{2}(1/p)}{\sqrt[3]p}\right)\right)=1\end{equation}
with $\lambda_1=\pi^2/18\approx 0.54831$ and an explicit $\lambda_2\in(0,\infty)$ given by $\lambda_2=\int_0^\infty h_2\approx7.0545$ with $h_2$ defined in \cref{eq:def:h2}.
\end{theorem}
\Cref{th:2n} strengthens \cite{Hartarsky19}*{Conjecture 7.1}, postulating the existence of a $\lambda_2$, without identifying its value, based on the fact that the size of the critical window is known to be small \cite{Hartarsky22phd}*{Proposition 1.4.3}. Moreover, thanks to \cite{Hartarsky23FA}*{Section~2}, our result implies an analogous lower bound on the expected infection time of the Fredrickson--Andersen 2-spin facilitated and Frob\"ose kinetically constrained models in two dimensions with the constants $\lambda_1$, $\lambda_2$, $\lambda_1^{\mathrm{F}}$, $\lambda_2^{\mathrm{F}}$ doubled.

Let us highlight a few steps in the proof of \cref{th:main}. Owing to \cref{th:locality}, we may focus on the local Frob\"ose model, starting with the upper bound in \cref{eq:main}. 
Namely, we consider a growing rectangle (the infected patch discussed above) and want to estimate the probability of each possible sequence of sizes recording its growth. Then we need to sum over all such sequences. In order to ensure that the corresponding events are disjoint, we need to keep track of a certain frame around the rectangle known to be free of infections (see \cref{subsec:frames}). The 6 possible states of this frame naturally give rise to a $6\times 6$ matrix whose entries are the probabilities of going from one state to another (see \cref{fig:augmented:cycle}). Subtracting from this matrix the terms yielding the main contribution $\lambda_1^{\mathrm F}/p$ in \cref{eq:main}, we obtain a new matrix. The constant $\lambda_2^{\mathrm F}$ appearing in \cref{th:main} is the integral of the function $h(z)=\sqrt{(2+\sqrt 2)/(e^z-1)}$ arising as the Perron--Frobenius eigenvalue of this matrix. Another way to view $h$ is as the entropy of the fluctuations of growth sequences around the solution of the optimisation problem for a certain differential form, which was identified by Holroyd \cite{Holroyd03} (see \cref{subsec:variations}). Thus, a careful treatment of the deviations of the differential form from its value along this optimal path are needed (see \cref{lem:W:to:g,lem:Hp}).

Thanks to \cref{th:locality} and in stark contrast to previous works, proving the lower bound in \cref{th:main} is essentially analogous to proving the upper one. The only noteworthy additional argument we employ in the lower bound are quantitative results on the sensitivity of the Perron--Frobenius eigenvalue of a positive matrix subjected to a perturbation (see the proof of \cref{prop:main:lower} in \cref{subsec:lower:step}).

\begin{figure}
    \centering
    \begin{tikzpicture}[x=1.2cm,y=4cm]
\draw[->,color=black] (1,0) -- (12.5,0) node[above]{$\log\frac1p$};
\foreach \x in {1,2,3,4,5,6,7,8,9,10,11,12}
\draw[shift={(\x,0)}] (0pt,2pt) -- (0pt,-2pt) node[below] {\footnotesize $\x$};
\draw[->,color=black] (1,0) -- (1,1.75) node[right]{$p\log\Pi(p)$};
\foreach \y in {0,0.25,0.5,0.75,1,1.25,1.5}
\draw[shift={(1,\y)}] (2pt,0pt) -- (-2pt,0pt) node[left] {\footnotesize $\y$};
\fill (1.38629436111989076, 0.4557867386130542) circle (2pt);
\fill (2.0794415416798357, 0.5928339472416244) circle (2pt);
\fill (2.7725887222397811, 0.7745581895375311) circle (2pt);
\fill (3.4657359027997265, 0.9546038641712591) circle (2pt);
\fill (4.1588830833596715, 1.1129858612276886) circle (2pt);
\fill (4.8520302639196169, 1.2430073481077526) circle (2pt);
\fill (5.5451774444795623, 1.345804448976588) circle (2pt);
\fill (6.2383246250395077, 1.424783934531369) circle (2pt);
\fill (6.9314718055994531, 1.4842946522264553) circle (2pt);
\fill (7.6246189861593985, 1.5284964038060636) circle (2pt);
\fill (8.317766166719343, 1.5609736388102249) circle (2pt);
\fill (9.0109133472792884, 1.5846389054119723) circle (2pt);
\fill (9.7040605278392338, 1.6017726607118656) circle (2pt);
\fill (10.397207708399179, 1.6141156659249392) circle (2pt);
\fill (11.090354888959125, 1.6229727530039586) circle (2pt);
\fill (11.78350206951907, 1.6293089682631352) circle (2pt);
\draw (1,1.644934066848226436472415166646) node[left]{$\pi^2/6$}--(12,1.644934066848226436472415166646);
\draw[domain=3:12,smooth] plot (\x,{1.644934066848226436472415166646-5.8049063042788620281153757483082*exp(-\x/2)});
\draw (4.5,1) node[right]{$\pi^2/6-\pi\sqrt{2+\sqrt2}p^{1/2}\pm p^{2/3}$};
\draw[domain=3:12,smooth,dashed] plot (\x,{1.644934066848226436472415166646-5.8049063042788620281153757483082*exp(-\x/2)-exp(-2*\x/3)});
\draw[domain=3:12,smooth,dashed] plot (\x,{1.644934066848226436472415166646-5.8049063042788620281153757483082*exp(-\x/2)+
exp(-2*\x/3)});
\draw[very thick] (1,0)--(4.2,0)--(4.2,1.13)--(1,1.13)--cycle;
\end{tikzpicture}
    \caption{Plot of the numerical estimate of $\Pi(p)$ defined in \cref{eq:def:Pip} for Frob\"ose bootstrap percolation. We also show the asymptotics proved in \cref{eq:Holroyd,eq:main} for comparison. The error term in \cref{eq:main} is simplified to $p^{-1/3}$, since the polylogarithmic factor is not optimised in the proof. The range of parameters previously accessible is given by the thick box, clearly showing why asymptotics could not have been deduced from such data.}
    \label{fig:frobose}
\end{figure}
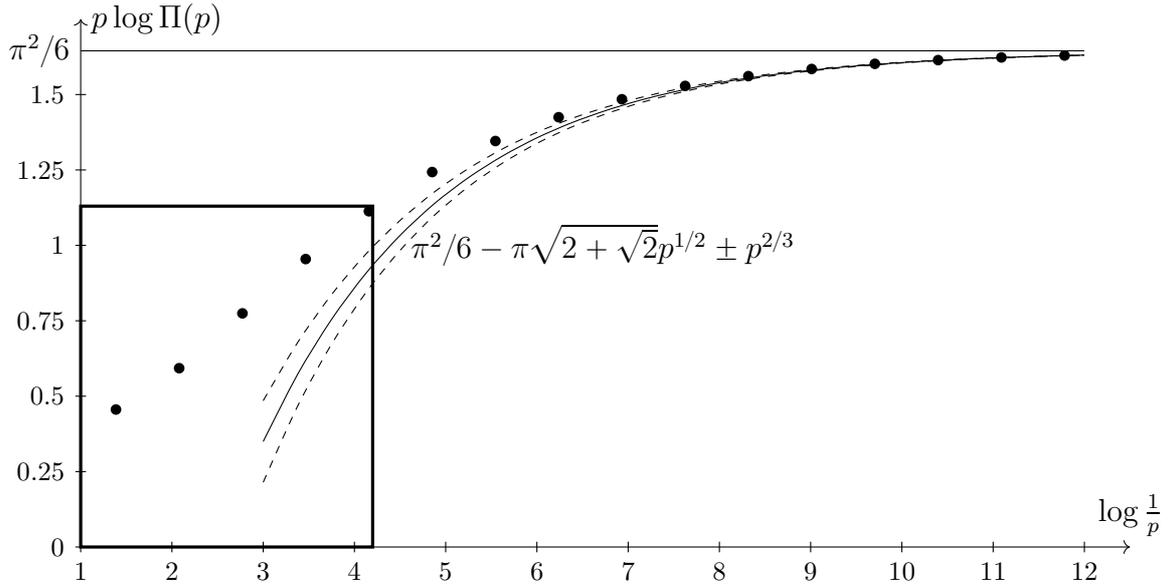

\subsubsection{Bootstrap percolation paradox}

Finally, we explore the implications of the locality \cref{th:locality} for the bootstrap percolation paradox. The present discussion applies equally well to both two-neighbour and Frob\"ose bootstrap percolation, so we abusively identify these models, as well as their local versions, in view of \cref{th:locality}, or even (local) modified two-neighbour bootstrap percolation.

An early stage of the proof of \cref{th:main} discussed above establishes that, up to a very small and well-controlled error, determining the infection time $\tau$ is equivalent to finding the probability that a certain Markov chain on rectangles with their frame state reaches perimeter $L=4\log(1/p)/p$. The trajectory of this chain encodes the growth sequence discussed above (see \cref{subsec:Markov}). There are only 7 possible frame states and the perimeter of the rectangle grows by at most 4 at each step. Therefore, we are able to apply a dynamic programming algorithm (see \cref{sec:paradox} for more details) to compute the probability of reaching the desired perimeter.\footnote{It should be noted that a strategy along these lines was suggested for local bootstrap percolation in \cite{Gravner09}*{p. 387}, but was never carried out.} Note that this deterministic algorithm has time complexity of order $L^2$ and space complexity of order $L$, both of which are polynomial in $1/p$.

This approach compares very favourably to the straightforward Monte Carlo simulation of the initial state and running the dynamics used e.g.\ in \cites{Adler88,Adler89,Adler91,Nakanishi86,Adler03,Lenormand84,Frobose89,Jackle91a,VanEnter90}. Indeed, Monte Carlo simulations of this kind require sampling roughly $\exp(1/p)$ Bernoulli variables in order to obtain any non-trivial result. In addition, they are subject to statistical errors. Thus, our algorithm brings the complexity from exponential down to quadratic, removes all statistical errors (the output of our algorithm is not random) and even allows one to rigorously quantify errors, if desired. We should note that \cites{Teomy14,DeGregorio05,DeGregorio06} attempted numerical computations somewhat related to ours, as opposed to Monte Carlo simulations. However, they still led to incorrect predictions: ``The striking conclusion is that not even the rigorous correction term can be captured reliably by the numerical exact solution yet'' \cite{DeGregorio09}.

The local approach allows us to reach significantly beyond the regimes previously probed (see \cref{fig:frobose}). Indeed, even relatively recent massive simulations \cite{Teomy14} fail to go beyond $p\approx0.016$. Using modest computational resources, we have been able to reach values of $p$ down to $2^{-17}\approx 0.0000076$. Moreover, our algorithm is not informed of any asymptotic results such as \cref{eq:AL:result,eq:Holroyd,eq:second,eq:main} and does not apply any heuristic simplifications to the problem. As it is clear from \cref{fig:frobose}, the agreement with rigorous results is remarkable. We could therefore use our numerical estimates in order to infer the various constants and exponents appearing in these results. Indeed, we successfully predict all these values within an error of about $0.6\%$ and $2\%$ for the first and second order terms respectively (see \cref{fig:frobose:2}) and even reliably estimate all four values simultaneously (see \cref{eq:4fit}). In comparison, Monte Carlo estimates of $\lambda_1$ are typically off by more than $50\%$.

\subsection{Organisation}
The remainder of the paper is structured as follows. In \cref{sec:preliminaries} we gather some definitions and preliminaries standard in the field, along with a few ingredients from \cite{Holroyd03}. The locality \cref{th:locality} is proved in \cref{sec:locality}.\Cref{sec:sequences} sets the stage for the following sections by introducing frames and growth sequences. In \cref{sec:upper} we prove the upper bound in \cref{eq:main}, while \cref{sec:lower} completes the proof of \cref{th:main}, by proving the lower bound, along similar but a bit more technical lines. In \cref{sec:paradox} we discuss numerics for the bootstrap percolation paradox in light of \cref{th:main,th:locality}. We conclude by suggesting a number of future directions of research opened up by the present work in \cref{sec:outlook}. Finally, \cref{appendix} provides the additional technical details needed for the proof of \cref{th:2n} with respect to \cref{th:main}.

\section{Preliminaries}
\label{sec:preliminaries}

In this section we gather various standard preliminaries, which will be used throughout the rest of the paper.

\subsection{Models}
\label{subsec:models}
\subsubsection{Two-neighbour bootstrap percolation in two dimensions}
Let us start by formally defining $2$-neighbour bootstrap percolation on $\bbZ^2$, which we refer to simply as \emph{bootstrap percolation}. Given a set $A\subset \bbZ^2$ of initially \emph{infected} sites, at each time step we further infect sites with at least $2$ infected neighbours. In other words, the set of vertices infected at time $t\ge 1$ is
\begin{equation}
\label{eq:def:2n}
A_t = A_{t-1} \cup \left\{x\in\bbZ^2:\left|A_{t-1}\cap N(x)\right|\ge 2\right\}\end{equation}
with $A_0=A$ and $N(x)$ denotes the set of neighbours of $x$ in the usual, nearest neighbour graph structure on $\bbZ^2$. Given $A$, we denote by $[A]=\bigcup_{t\ge 0} A_t$ its \emph{closure}. We say that $A$ is \emph{stable}, if $[A]=A$. The observable we are interested in is the \emph{infection time} of the origin 
\begin{equation}
\label{eq:def:tau}
\tau=\inf\left\{t\ge 0:0\in A_t\right\}\end{equation}
with $\inf\varnothing=\infty$. Throughout the paper we consider $A$ to be taken at random with distribution $\bbP_p$ such that $(\1_{x\in A})_{x\in\bbZ^2}$ are independent Bernoulli random variables where $p\in(0,1)$ is the \emph{parameter} of the model. Thus, $\tau$ becomes a random variable whose distribution is the object of our study.

\subsubsection{Frob\"ose bootstrap percolation}
\emph{Frob\"ose bootstrap percolation} is defined identically to two-neighnour bootstrap percolation on $\mathbb Z^2$ but with \cref{eq:def:2n} replaced by 
\begin{equation}
\label{eq:def:FBP}
A_t = A_{t-1}\cup\left\{x\in\bbZ^2:\exists \{a,b,c\}\subset A_{t-1},\{a,b\}=N(x)\cap N(c),a\neq b\neq c\neq a\right\}
\end{equation}
Its closure and infection time are defined analogously and denoted by $[\cdot]^{\mathrm{F}}$ and $\tf$ respectively.

\subsubsection{Local two-neighbour bootstrap percolation}
We next turn to \emph{local two-neighbour bootstrap percolation} on $\mathbb Z^2$. As for its non-local version, we need a set $A\subset\bbZ^2$ of \emph{initially infected sites}. However, we further require a \emph{germ} $x\in A$. All germs are considered infected. In the local bootstrap percolation process any site with at least two infected neighbours, at least one of which is a germ, becomes a germ. Moreover, any infected site with a germ neighbour becomes a germ. Formally, this leads to the following definition for $t\ge 1$
\begin{align}
\label{eq:def:loc2n}A_t^{x}&{}= A_{t-1}^x \cup \left\{y\in\bbZ^2:\left|A^x_{t-1}\cap N(y)\right|\ge 2,X^x_{t-1}\cap N(y)\neq\varnothing\right\},\\
\nonumber X^x_t&{} = X_{t-1}^x \cup \left\{y\in A^x_{t}:X^x_{t-1}\cap N(y)\neq\varnothing\right\}
\end{align}
with $A^x_0=A$ and $X^x_0=\{x\}$. Thus, $A_t^x$ and $X_t^x$ are the sets of infected sites and germs with initial germ $x$ at time $t$ respectively. The local closure is defined by $[A]^x=\bigcup_{t\ge 0}X^x_t$ and the local infection time by
\begin{equation}
\label{eq:def:tauloc}
\tau_{\mathrm{loc}}=\inf_{x\in A}\inf\left\{t\ge 0:0\in A_t^x\right\}\ge \tau.
\end{equation}
Note the importance of the infimum over $x \in A$ to avoid introducing a big discrepancy between the local and original models e.g.\ if $x$ is an isolated infection.

\subsubsection{Local Frob\"ose bootstrap percolation}
\label{subsec:local:F}
Local Frob\"ose bootstrap percolation is defined by its set $A\subset\bbZ^2$ of \emph{initially infected sites} and \emph{germ} $x\in A$ by setting $A^x_0=A$, $X_0^x=\{x\}$ and for $t\ge 1$
\begin{align*}
\label{eq:def:locFBP}A_t^{x}&{}=A_{t-1}^x\cup\left\{y\in\bbZ^2:\exists a\in X_{t-1}^x,\exists\{b,c\}\subset A^x_{t-1},\{a,b\}=N(y)\cap N(c),a\neq b\neq c\neq a\right\},\\
X_t^x&{}=X_{t-1}^x\cup\left\{y\in A^x_t:X_{t-1}^x\cap N(y)\neq\varnothing\right\}\end{align*}
Its closure and infection time are denoted by $[A]^x_{\mathrm{F}}=\bigcup_{t\ge 0}X_t^x$ and 
\begin{equation}
\label{eq:def:tfloc}
\tf_{\mathrm{loc}}=\inf_{x\in A}\inf\left\{t\ge 0:0\in A_t^x\right\}\ge \tf.
\end{equation}

\subsection{Correlation inequalities}
We next recall two standard correlation inequalities for the product measure $\bbP_p$. We say that an event $\cE$ (i.e.\ a family of subsets of $\bbZ^2$) is \emph{increasing} if for any $E\in \cE$ and $F\subset\bbZ^2$ such that $F\supset E$ it holds that $F\in\cE$. The first inequality was proved by Harris \cite{Harris60}, but is often referred to as FKG inequality.
\begin{lemma}[Harris inequality]
\label{lem:Harris}For any increasing events $\cE,\cF$ we have
\[\bbP_p(\cE\cap\cF)\ge\bbP_p(\cE)\bbP_p(\cF).\]
\end{lemma}
For the second inequality we need the notion of disjoint occurrence. We say that the increasing events $\cE$ and $\cF$ \emph{occur disjointly} for a realisation $A\subset \bbZ^2$, if there exist disjoint sets $B,C\subset A$ such that $\{E\subset \bbZ^2:B\subset E\}\subset\cE$ and $\{F\subset\bbZ^2:C\subset F\}\subset\cF$. The disjoint occurrence of $\cE$ and $\cF$ is denoted by $\cE\circ\cF$. We can then state the van den Berg--Kesten inequality \cite{BK85}.
\begin{lemma}[BK inequality]
\label{lem:BK} For any increasing events $\cE$ and $\cF$ measurable with respect to $A\cap[-N,N]^2$ for some $N\ge 0$, then
\[\bbP_p(\cE\circ\cF)\le \bbP_p(\cE)\bbP_p(\cF).\]
\end{lemma}

\subsection{Rectangles}

As it will become clear in \cref{subsec:rectangles:process}, rectangles with axis-parallel sides arise naturally in bootstrap percolation. We therefore call a \emph{rectangle} a set of the form
\[R(a,b;c,d)=([a,c)\times[b,d))\cap\bbZ^2\]
for any integers $a<c$ and $b<d$. The \emph{side lengths} of a rectangle $R=R(a,b;c,d)$ are $\sh(R)=\min(c-a,d-b)$ and $\lng(R)=\max(c-a,d-b)$, while the \emph{semi-perimeter} of $R$ is $\phi(R)=c-a+b-d$. We further set
\[R(a,b)=R(0,0;a,b).\]

We next define a few important events for a rectangles, which are central to our work. Recall from \cref{subsec:models} that $A$ denotes the random subset of $\bbZ^2$ of initial infections with law $\bbP_p$, $[\cdot]$ denotes the closure and $[\cdot]^x$ the local closure. We say that a rectangle $R$ is \emph{internally filled} (resp.\ \emph{Frob\"ose internally filled}, \emph{locally internally filled}, \emph{Frob\"ose locally internally filled}), if the event
\begin{align}
\label{eq:def:I}
\cI(R)&{}=\left\{A\subset \bbZ^2:[A\cap R]=R\right\},\\
\label{eq:def:IF}\cI^{\mathrm{F}}(R)&{}=\left\{A\subset \bbZ^2:[A\cap R]^{\mathrm{F}}=R\right\},\\
\label{eq:def:Iloc}\cI_{\mathrm{loc}}(R)&{}=\left\{A\subset \bbZ^2:\exists x\in A,[A\cap R]^x=R\right\}\subset\cI(R),\\
\label{eq:def:IFloc}\cI^{\mathrm{F}}_{\mathrm{loc}}(R)&{}=\left\{A\subset \bbZ^2:\exists x\in A,[A\cap R]^x_{\mathrm{F}}=R\right\}\subset\cI^{\mathrm{F}}(R)
\end{align}
occurs. For two nested rectangles $S\subset R$, we say that there is a \emph{crossing from $S$ to $R$}, if
\begin{align}
\label{eq:def:C}
\cC(S,R)&{}=\left\{A\subset \bbZ^2:\exists s\in S,[S\cup (A\cap R)]^s=R\right\}\\
\label{eq:def:C:F}\cC^{\mathrm{F}}(S,R)&{}=\left\{A\subset \bbZ^2:\exists s\in S,[S\cup (A\cap R)]^s_{\mathrm{F}}=R\right\}
\end{align}
occurs.
The use of crossings comes from the next observation, which follows directly from \cref{eq:def:C,eq:def:I,eq:def:Iloc}.
\begin{observation}[Stacking crossings]
\label{obs:stacking}
For any rectangles $T\subset S\subset R$ we have
\begin{align*}
\cI(S)\cap\cC(S,R)&{}\subset\cI(R),\\
\cI_{\mathrm{loc}}(S)\cap\cC(S,R)&{}\subset \cI_{\mathrm{loc}}(R),\\
\cC(T,S)\cap\cC(S,R)&{}\subset\cC(T,R).\end{align*}
The same holds for the Frob\"ose model.
\end{observation}

Finally, we say that a set (not necessarily a rectangle) $X\subset \bbZ^2$ is \emph{occupied}, if the event
\begin{equation}
    \label{eq:def:O}\cO(X)=\left\{A\subset\bbZ^2:A\cap X\neq\varnothing\right\}
\end{equation}
occurs.

\subsection{Rectangles process}
\label{subsec:rectangles:process}
The first thing to notice about two-neighbour bootstrap percolation (resp.\ Frob\"ose bootstrap percolation) is that the closure of any finite set of infections is the smallest\footnote{A collection $\cC$ is \emph{smaller} than a collection $\cC'$, if $\bigcup_{R\in\cC}R\subset\bigcup_{R\in\cC'}R$.} collection of rectangles at graph distance at least 3 (resp.\ 2) from each other containing all the infections. Thus, the closure of any set can be determined via the following \emph{rectangles process}. We start off with a collection of rectangles consisting of each of the initial infections. At each step we merge two of them at graph distance 2 or less (resp.\ 1 or less), replacing them by the smallest rectangle containing their union. Repeating this until the process becomes stationary yields the collection of rectangles in the closure.

We next derive a few simple but fundamental consequences of the rectangles process. The first one is the following extremal bound, which constitutes a folklore exercise.
\begin{lemma}[Extremal bound]
\label{lem:perimeter}
Let $R=R(a,b)$ be a rectangle. If $\cI(R)$ occurs, then $|R\cap A|\ge \lceil (a+b)/2\rceil$. If $\cI^{\mathrm{F}}(R)$ occurs, then $|R\cap A|\ge a+b-1$.
\end{lemma}
\begin{proof}
Observe that infecting a site with at least 2 infected neighbours cannot increase the edge-boundary of the infected zone. Initially the edge-boundary is at most $4|R\cap A|$, while in the end it is $2\phi(R)=2a+2b$. For Frob\"ose bootstrap percolation see \cite{Gravner12}*{p.~20} (use \cref{lem:decomposition} and induction on $\phi(R)$).\footnote{For Frob\"ose bootstrap percolation another proof can be obtained by considering the bipartite graph with vertices being the rows and columns of $R$ and edges corresponding to $A$ and observing that this graph needs to be connected.}
\end{proof}
The second corollary of the rectangles process we require is the following lemma due to Holroyd~\cite{Holroyd03}*{Proposition 30} (see \cite{Hartarsky19}*{Lemma 5.6} for a strengthening). It follows by considering the last merging step of the rectangles process.
\begin{lemma}[Disjoint occurrence decomposition]
\label{lem:decomposition}
Let $R$ be a rectangle with $|R|>1$. If $\cI(R)$ occurs, then there exist rectangles $S,T\subsetneq R$ such that $\cI(S)\circ \cI(T)$ occurs and $[S\cup T]=R$.
The same holds for Frob\"ose bootstrap percolation.\end{lemma}
Finally, iterating \cref{lem:decomposition} and taking the larger of the two resulting rectangles at each step yields the following fundamental lemma of Aizenman and Lebowitz~\cite{Aizenman88}*{Lemma~1}.
\begin{lemma}[AL lemma]
\label{lem:AL}
Let $R=R(a,b)$ and $1\le l\le \max(a,b)$. If $\cI(R)$ occurs, then there exists a rectangle $S\subset R$ with longest side length $s$ satisfying $l\le s\le 2l$ such that $\cI(S)$ occurs. The same holds for Frob\"ose bootstrap percolation.
\end{lemma}

Turning to local Frob\"ose bootstrap percolation, a similar rectangles process is available. Namely, there is only one ``seed'' rectangle, which grows by merging with a single infection adjacent to it at a time. This immediately entails a local version of the AL lemma.
\begin{lemma}[Local Frob\"ose AL lemma]
\label{lem:AL:loc}
Let $R=R(a,b)$ and $2\le l\le a+b$. If $\cI_{\mathrm{loc}}^{\mathrm{F}}(R)$ occurs, then there exists a rectangle $S\subset R$ with $\phi(S)=l$ such that $\cI_{\mathrm{loc}}^{\mathrm{F}}(S)\cap \cC^{\mathrm{F}}(S,R)$ occurs.
\end{lemma}

\subsection{Constants and asymptotic notation}

Before we proceed with ways of bounding the probability of internal filling, we should say a few words about asymptotic notation. Given a real function $f$ and a positive one $g$, we say that $f(x)=O(g(x))$, if there exists $C>0$ such that $|f(x)|\le C g(x)$ for all $x$ in the domain of $f$. We write $f(x)=O(g(x))$ as $x\to a\in\bbR\cup\{\pm\infty\}$, if there exists $C>0$ such that $\liminf_{x\to a}Cg(x)-|f(x)|\ge 0$. We write $f(x)=o(g(x))$ as $x\to a\in\bbR\cup\{\pm\infty\}$, if $\lim_{x\to a} f(x)/g(x)=0$. Unless otherwise specified, all asymptotic notation holds as $p\to 0$.

We also need several constants
\[1\ll C_0\ll C_1\ll C_2\ll C_3\ll C_4.\]
That is, $C_0$ is chosen large enough, $C_1$ is chosen larger than a sufficiently large function of $C_0$, etc. All of these constants are allowed to depend on implicit constants in asymptotic $O(\cdot)$ notation, but not on $p$. Indeed, we require $p>0$ to be small enough depending on $C_4$ unless otherwise stated (this will only be the case in \cref{prop:Holroyd:upper} and its application in \cref{lem:isotropic:bound}).

\subsection{Traversability}
We next import some notions from Holroyd's work \cite{Holroyd03}, starting with some important functions. 
\begin{align}
\label{eq:def:f}f&{}:(0,\infty)\to(0,\infty):z\mapsto-\log(1-e^{-z}),\\
\label{eq:def:beta}\beta&{}:(0,1)\to(0,1):u\mapsto\frac{u+\sqrt{u(4-3u)}}{2},\\
g&{}:(0,\infty)\to(0,\infty):z\mapsto-\log\beta(1-e^{-z}).\label{eq:def:g}
\end{align}
It is not hard to check that $f$ and $g$ are decreasing convex analytic functions with the following asymptotics\footnote{Corresponding but less precise asymptotics for first and second derivatives of $f$ and $g$ will also be used.}
\begin{align}
\label{eq:f:asymptotics}f(z)&=\begin{cases}-\log z+z/2+O(z^2)&z\to 0,\\
e^{-z}+O(e^{-2z})&z\to\infty,\end{cases}\\
\label{eq:g:asymptotics}g(z)&{}=\begin{cases}-\frac12\left(\log(z)+\sqrt{z}\right)+O(z)&z\to 0,\\
e^{-2z}+O\left(e^{-3z}\right)&z\to\infty.
\end{cases}
\end{align}
In particular, $f$ and $g$ are integrable and one can show that \cite{Holroyd03}*{Proposition~5}
\begin{align}
\label{eq:integral}
\lambda_1^{\mathrm{F}}&{}=\int_0^\infty f=\frac{\pi^2}{6}&\lambda_1&{}=\int_0^\infty g=\frac{\pi^2}{18}\end{align}
(also see \cites{Holroyd03a,Bringmann12} for generalisations of the functions $f,g$ and \cref{eq:integral}).

The relevance of $f$ comes from the following fact. Let \begin{equation}
    \label{eq:def:q}
q=-\log(1-p)=p+O(p^2),
\end{equation}
so that $q\ge p$. Then, recalling \cref{eq:def:O}, for any finite set $X\subset\bbZ^2$, we have 
\[\bbP_p\left(\cO^c(X)\right)=\bbP_p\left(A\cap X\neq\varnothing\right)=1-e^{-|X|q}=e^{-f(|X|q)}.\]
\Cref{lem:traversability} below gives a similar relation for $g$. We say that $R=R(a,b)$ is \emph{East-traversable}, if 
\[\cT_\rightarrow(R)=\cO\left(R(a-1,0;a,b)\right)\cap \bigcap_{i=1}^{a-1}\cO\left(R(i-1,0;i+1,b)\right)\]
occurs. In other words, we require that there is an infection on every two consecutive columns of $R(a,b)$ and that the right-most column does contain an infection. We similarly define traversability for other directions. We denote the corresponding events by $\cT_{\zeta}(R)$ with $\zeta\in\{\uparrow,\leftarrow,\downarrow\}$.
For the Frob\"ose model the corresponding notion is the following. We say that $R=R(a,b)$ has no \emph{horizontal gaps} (resp.\ vertical gaps), if 
\begin{align*}
\cG_-(R)&{}=\bigcap_{i=0}^{b-1}\cO\left(R(0,i;a,i+1)\right),&\cG_|(R)&{}=\bigcap_{i=0}^{a-1}\cO\left(R(i,0;i+1,b)\right)\end{align*}
occurs. The importance of these events comes from the following observation.
\begin{observation}
\label{obs:traversability}
Let $R=R(a,b)$ be a rectangle and $A\subset\bbZ^2$. The following hold.
\begin{enumerate}    \item\label{obs:travers:infect} $\cI(R)\subset\bigcap_{\zeta\in\{\rightarrow,\uparrow,\leftarrow,\downarrow\}}\cT_\zeta(R)$. Similarly, $\cI^{\mathrm{F}}(R)\subset\cG_-(R)\cap\cG_|(R)$.
    \item\label{obs:travers:boundary} If $R(-1,0;0,b)\subset[A]$ and $A\in\cT_\rightarrow(R)$, then $R\subset[A]$. Similarly, if $R(-1,0;0,b)\subset [A]^{\mathrm{F}}$ and $A\in\cG_|(R)$, then $R\subset [A]^{\mathrm{F}}$.
    \item\label{obs:travers:germ_boundary} If $x\in A$, $R(-1,0;0,b)\subset [A]^x$ and $A\in\cT_\rightarrow(R)$, then $R\subset[A]^x$, where we recall that $[\cdot]^x$ is the closure of local bootstrap percolation with germ $x$.
    Similarly, if $x\in A$, $R(-1,0;0,b)\subset [A]^x_{\mathrm{F}}$ and $A\in\cG_|(R)$, then $R\subset[A]^x_{\mathrm{F}}$.
\end{enumerate}
\end{observation}
Clearly, \begin{equation}
\label{eq:gaps}\bbP_p\left(\cG_|(R)\right)=\exp(-af(bq)).\end{equation} A similar link between traversability and the function $g$ is provided by \cite{Holroyd03}*{Lemma 8} as follows.
\begin{lemma}[Traversability probability]
\label{lem:traversability}
Let $R=R(a,b)$. Then
\[
\exp(-ag(bq))\ge\bbP_p\left(\cT_\rightarrow(R)\right)\ge\exp(-(a-1)g(bq)-f(bq))\ge p\exp(-(a-1)g(bq)).\]
\end{lemma}

\subsection{Variational principles}
\label{subsec:variations}
We next recall some variational tools from \cite{Holroyd03}*{Section 6}. For any coordinate-wise non-decreasing piecewise-differentiable path in $[0,\infty)^2$ (simply \emph{path} in the sequel) we set
\begin{align*}
W(\gamma)&{}=\int_\gamma g(x)\md y+g(y)\md x,&W^{\mathrm{F}}(\gamma)&{}=\int_\gamma f(x)\md y+f(y)\md x.
\end{align*}
In other words, if we view $\gamma$ as a piecewise-differentiable function $\gamma:[0,1]\to [0,\infty)^2:t\mapsto(\gamma_1(t),\gamma_2(t))$, we have
\[W(\gamma)=\int_0^1\left(g(\gamma_1(t))\gamma_2'(t)+g(\gamma_2(t))\gamma_1'(t)\right)\md t\]
and similarly for $W^{\mathrm{F}}$. Recalling \cref{eq:def:q}, we further set
\begin{align}
\label{eq:def:Wp}
W_p(\gamma)&{}=\int_\gamma g(qx)\md y+g(qy)\md x,&W^{\mathrm{F}}_p(\gamma)&{}=\int_\gamma f(qx)\md y+f(qy)\md x.
\end{align}
Since paths we need are all piecewise linear, we write them out by specifying the points to be joined by straight line segments, such as $\gamma = (x_1, x_2, x_3)$, for $x_i \in [0, \infty)^2$.

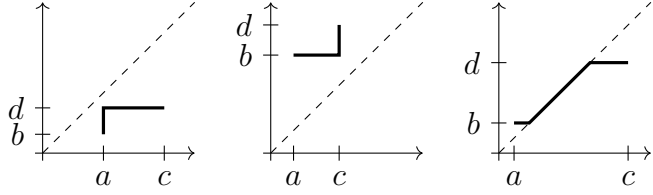
\begin{figure}
    \centering
    \begin{tikzpicture}
        \draw[->] (-.1, 0) -- (2, 0);
        \draw[->] (0, -.1) -- (0, 2);
        \draw[dashed] (0, 0) -- (2, 2);
        \draw[very thick] (0.8, 0.25) -- (0.8, 0.6) -- (1.6, 0.6);
        \draw (0.8,0.1) -- (0.8,-0.1) node[below]{$a$};
        \draw (1.6,0.1) -- (1.6,-0.1) node[below]{$c$};
        \draw (0.1,0.25) -- (-0.1,0.25) node[left]{$b$};
        \draw (0.1,0.6) -- (-0.1,0.6) node[left]{$d$};
        \draw[->] (2.9, 0) -- (5, 0);
        \draw[->] (3, -.1) -- (3, 2);
        \draw[dashed] (3, 0) -- (5, 2);
        \draw[very thick] (3.3, 1.3) -- (3.9, 1.3) -- (3.9, 1.7);
        \draw (3.3,0.1) -- (3.3,-0.1) node[below]{$a$};
        \draw (3.9,0.1) -- (3.9,-0.1) node[below]{$c$};
        \draw (3.1,1.3) -- (2.9,1.3) node[left]{$b$};
        \draw (3.1,1.7) -- (2.9,1.7) node[left]{$d$};
        \draw[->] (5.9, 0) -- (8, 0);
        \draw[->] (6, -.1) -- (6, 2);
        \draw[dashed] (6, 0) -- (8, 2);
        \draw[very thick] (6.2, .4) -- (6.4, .4) -- (7.2, 1.2) -- (7.7, 1.2);
        \draw (6.2,0.1) -- (6.2,-0.1) node[below]{$a$};
        \draw (7.7,0.1) -- (7.7,-0.1) node[below]{$c$};
        \draw (6.1,0.4) -- (5.9,0.4) node[left]{$b$};
        \draw (6.1,1.2) -- (5.9,1.2) node[left]{$d$};
    \end{tikzpicture}
    \caption{The three cases for the shape of $\gamma_{S,T}$ defined in \cref{eq:optimal:path}.}
    \label{fig:optimal_path}
\end{figure}
We next define certain specific paths which turn out to optimise $W$ and $W^{\mathrm F}$. Given $0\le a\le c$, $0\le b\le d$ and rectangles $S,T$ which are translates of $R(a,b)$ and $R(c,d)$ respectively, we define 
\begin{align}
    \label{eq:gammaR}
    \gamma_S&{}=((0,0),(\min(a,b),\min(a,b)),(a,b)),\\
    \label{eq:optimal:path}
    \gamma_{S,T}&{}=\begin{cases}
    ((a,b),(a,d),(c,d))&d<a,\\
    ((a,b),(c,b),(c,d))&c<b,\\
    ((a,b),(\max(a,b),\max(a,b)),(\min(c,d),\min(c,d)),(c,d))&\text{otherwise}\end{cases}
\end{align}
(see \cref{fig:optimal_path}), so that in fact $\gamma_{S}=\gamma_{\varnothing,S}$. In words, these are the paths staying as close to the diagonal as possible.
The following is proved in \cite{Holroyd03}*{Proof of Proposition 14}.
\begin{lemma}[Optimal path]
\label{lem:optimal:path}
Let $0\le a\le c$, $0\le b\le d$ and $S,T$ be translates of $R(a,b)$ and $R(c,d)$ respectively. Then,
    \begin{align*}
    \inf_{\gamma:(a,b)\to(c,d)}W(\gamma)&{}=W(\gamma_{S,T}),&\inf_{\gamma:(a,b)\to(c,d)}W^{\mathrm{F}}(\gamma)&{}=W^{\mathrm{F}}(\gamma_{S,T}),
    \end{align*}
    where the infimum is over all continuous paths $\gamma$ starting at $(a,b)$ and ending at $(c,d)$.
\end{lemma}

\subsection{Holroyd bounds}
The next lemma relates the functional $W$ to the probability of internally filling a rectangle. It is morally due to Gravner and Holroyd~\cite{Gravner08}, up to minor adjustments as detailed below.
\begin{proposition}[A priori lower bound]
\label{prop:Holroyd:lower}
Let $R$ be a rectangle. Recalling \cref{eq:def:Iloc,eq:def:Wp,eq:gammaR}, we have
\begin{align*}
\bbP_p\left(\cI_{\mathrm{loc}}(R)\right)&{}\ge p^3\exp\left(-W_p(\gamma_R)\right),&\bbP_p\left(\cI^{\mathrm{F}}_{\mathrm{loc}}(R)\right)&{}\ge p\exp\left(-W^{\mathrm{F}}_p(\gamma_R)\right).\end{align*}
\end{proposition}
\begin{proof}
Let $R=R(a,b)$ with $a\le b$. By \crefdefpart{obs:traversability}{obs:travers:germ_boundary} and the fact that $g$ is decreasing, we have
\begin{equation}
\label{eq:square:traversability:decomposition}
\bbP_p(\cI_{\mathrm{loc}}(R))\ge \bbP_p(\cI_{\mathrm{loc}}(R(a,a)))\bbP_p(\cT_\uparrow(R(a,b-a))).
\end{equation}
\cite{Gravner08}*{Lemma 12} gives
\begin{equation}
\label{eq:square:lower:bound}\bbP_p(\cI_{\mathrm{loc}}(R(a,a)))\ge \bbP_p(\cI(R(2,2)))\exp\left(-2\sum_{i=1}^{a-2}g(iq)\right)\ge p^2\exp\left(-\frac2q\int_0^{aq}g\right),\end{equation}
if $a>1$, while clearly $\bbP_p(\cI_{\mathrm{loc}}(R(a,a)))=p$ if $a=1$, so \cref{eq:square:lower:bound} holds for any $a\ge 1$. Combining \cref{eq:square:lower:bound,eq:square:traversability:decomposition} with \cref{lem:traversability} yields the desired result. The proof for Frob\"ose bootstrap percolation is analogous.
\end{proof}

Converse bounds are significantly harder to prove. The following proposition is essentially due to Holroyd \cite{Holroyd03}, even though it only appears in a convenient form in \cite{Gravner12}*{Proposition 15} (see also \cite{Gravner12}*{Theorem 20} for the Frob\"ose model).
\begin{proposition}[A priori upper bound on the critical scale]
\label{prop:Holroyd:upper}
Let $p\le 1/C_4$. Let $R=R(a,b)$ with
\[\frac{1}{C_3p}\le a\le b\le \frac{C_3\log(1/p)}{p}.\]
Then, recalling \cref{eq:gammaR}, we have
\begin{align*}
\bbP_p(\cI(R))&{}\le \exp\left(1/(C_3p)-W_p(\gamma_R)\right),&\bbP_p\left(\cI^{\mathrm{F}}(R)\right)&{}\le \exp\left(1/(C_3p)-W_p^{\mathrm{F}}(\gamma_R)\right).\end{align*}
\end{proposition}

Roughly speaking, \cref{prop:Holroyd:lower,prop:Holroyd:upper} are exactly the full content of \cite{Holroyd03}. In particular, together with the classical \cref{lem:tau} below, they establish \cref{eq:Holroyd}. Fortunately, we do not need to look deeper into the proof of \cite{Holroyd03}, but only use the above results as black boxes.

\section{Reduction to local bootstrap percolation}
\label{sec:locality}
In the present section we prove the locality \cref{th:locality}. We start by proving some rough \emph{a priori} bounds on the probability of a rectangle being internally filled in \cref{subsec:a:priori}, which is then used repeatedly in the main argument in \cref{subsec:reduction}.

\subsection{A priori upper bounds on subcritical scales}
\label{subsec:a:priori}
While \cref{prop:Holroyd:upper} is tight to leading order for rectangles on scale $1/p$, we also need (crude) bounds for smaller rectangles. Both \cref{lem:isotropic:bound,lem:anisotropic:bound} below can be viewed as improvements on \cite{Gravner12}*{Lemma 2}, which is insufficient for our purposes. Both lemmas incorporate tilting ideas common in large deviations theory (see e.g.\ \cite{DenHollander08} for background). In \cref{lem:isotropic:bound} we treat rectangles with bounded aspect ratio, by tilting the $\mathbb P_p$ measure to a $\mathbb P_{p_0}$ for $p_0>p$ so that we can apply \cref{prop:Holroyd:upper}.
\begin{lemma}[Bounded aspect ratio]
\label{lem:isotropic:bound}
Let $C_4\le s\le t\le C_1 s$. If $s\le1/(C_2p)$ and $R=R(s,t)$, then
\begin{align*}
\bbP_p(\cI(R))&{}\le \left(spe^{-1/(3C_1)}\right)^{\lceil(s+t)/2\rceil},&\bbP_p\left(\cI^{\mathrm{F}}(R)\right)&{}\le \left(spe^{-1/(3C_1)}\right)^{s+t-1}.\end{align*}
\end{lemma}
\begin{proof}
Let $m=\lceil(s+t)/2\rceil$. By \cref{lem:perimeter}, $\cI(R)$ implies $|R\cap A|\ge m$. Set $p_0= 2m/(C_2|R|)\ge p$. Observe that for any $A_0\subset R$ such that $|A_0|\ge m$ we have 
\begin{align*}
\frac{\bbP_p(A\cap R=A_0)}{\bbP_{p_0}(A\cap R=A_0)}&{}=\frac{p^{|A_0|}(1-p)^{|R|-|A_0|}}{p_0^{|A_0|}(1-p_0)^{|R|-|A_0|}}\le \frac{p^{m}(1-p)^{|R|-m}}{p_0^{m}(1-p_0)^{|R|-m}}\\
&{}\le \frac{p^m}{p_0^m(1-p_0)^{|R|}}\le\left(\frac{e^{4/C_2}p}{p_0}\right)^m,\end{align*}
since $p_0\ge p$. Thus,
\begin{equation}
\label{eq:isotropic:intermediate}
\bbP_p(\cI(R))\le \left(\frac{e^{4/C_2}p}{p_0}\right)^m\bbP_{p_0}(\cI(R)).\end{equation}
Yet, by \cref{prop:Holroyd:upper} (note that it applies to $p_0$ instead of $p$, because $p_0\le 2/(C_2C_4)$ is taken small enough depending on $C_3$),
\begin{align*}
    \bbP_{p_0}(\cI(R))&{}\le \exp\left(\frac{1}{C_3p_0}-W_{p_0}(\gamma_{R})\right)=\exp\left(\frac{1}{C_3p_0}-\frac{2}{q_0}\int_{0}^{sq_0}g-(t-s)g(sq_0)\right)\\
    &{}\le\begin{multlined}[t]\exp\left(\frac{1}{C_3p_0}+s(\log(sp_0)-1)+O\left(sp_0+\sqrt{s^3p_0}\right)\right)\\
    \times\exp\left(-(t-s)\log(sp_0)/2+O(tp_0+t\sqrt{sp_0})\right)\end{multlined}\\
    &{}\le \exp\left(-s\left(1-O\left(C_1/\sqrt{C_2}\right)\right)+\frac{s+t}{2}\log(sp_0)\right)\\
    &{}\le \exp\left(m\left(\log(sp_0)-\frac{1}{2C_1}\right)\right),
    \end{align*}
    using the fact that $q_0=p_0+O(p_0^2)$ and $g(z)=-\log(z)/2+O(\sqrt z)$ as $z\to0$ on the second line, $C_4\le s\le t\le C_1s\le 2C_1/(C_2p_0)$ in the third one and $s\ge t/C_1\ge m/C_1\ge C_4/C_1\gg C_2\ge 1/(sp_0)$ in the last one. Injecting this into \cref{eq:isotropic:intermediate} and recalling that $C_2\gg C_1$, we obtain the desired bound.

    The proof for Frob\"ose bootstrap percolation is the same taking $m=s+t-1$.
\end{proof}
We next turn to rectangles with large aspect ratio, starting with Frob\"obse bootstrap percolation. The idea is to replace the event that a rectangle is internally filled by it having no horizontal gaps and a sufficient number of infections. The latter event is then approximated by a large deviation event for a suitable random walk.
\begin{lemma}[Frob\"ose large aspect ratio]
\label{lem:anisotropic:bound:frobose}
There exists a function $\xi^{\mathrm F}:[1,2]\to[1,2]$ such that $\xi^{\mathrm F}(x)=1+O((x-1)\log(x-1))$ as $x\to1$ and the following holds. Let $s\le t$ with $s\le 1/(C_1p)$ and $t>1$. Set $m=s+t-1$. If $R=R(s,t)$, then
\[\bbP_p\left(\cI^{\mathrm{F}}(R)\right)\le\left(sp\xi^{\mathrm F}\left(\frac{m-1}{t-1}\right)\right)^{m}.\]
\end{lemma}
\begin{proof}
We have that $\cI^{\mathrm{F}}(R)$ implies that $R$ has no horizontal gaps by \crefdefpart{obs:traversability}{obs:travers:infect} and contains at least $m$ infections by \cref{lem:perimeter}. Let us denote 
\[\cI'(R)=\left\{A\subset\bbZ^2:|R\cap A|\ge m\right\}\cap\cG_-(R)\supset\cI^{\mathrm{F}}(R).\]
Observe that from any configuration in $\cI'(R)$ one can extract a set of $m$ infections witnessing $\cI'(R)$. We next seek to upper bound the number of such reduced configurations in $\cI'(R)$. Given a configuration with $m$ infections in $R$ we denote the infections by $(x_i,y_i)_{i=1}^m$ with $y_{i+1}\ge y_{i}$ for all $i\in\{1,\dots,m-1\}$. Clearly, $(x_i)_{i=1}^m\in\{0,\dots,s-1\}^m$, while $\cG_-(R)$ implies that $(y_{i+1}-y_i)_{i=1}^{m-1}\in\{0,1\}^{m-1}$, $y_1=0$ and $y_m=t-1$. Thus, $t-1=y_m-y_1=\sum_{i=1}^{m-1}(y_{i+1}-y_i)$. By the exponential Markov inequality (Cram\'er's theorem)
\begin{align}
\nonumber
\bbP_p\left(\cI^{\mathrm{F}}(R)\right)&{}\le \bbP_p\left(\cI'(R)\right)\le \frac{(2sp)^m}{2} \exp\left(-(m-1)\sup_{u\in\bbR}\left(\frac{u(t-1)}{m-1}-\log\frac{1+e^u}{2}\right)\right)\\
&{}=(sp)^m\left(\xi^{\mathrm F}\left(\frac{m-1}{t-1}\right)\right)^{m-1},\label{eq:anisotropic:intermediate:F}\\
\nonumber\xi^{\mathrm F}(x)&{}=x(x-1)^{(1-x)/x}\in[1,2],\end{align}
which has the right asymptotics.
\end{proof}

For the two-neighbour model the statement and proof of \cref{lem:anisotropic:bound:frobose} are analogous, but a little more technical, as seen in what follows.
\begin{lemma}[Large aspect ratio]
\label{lem:anisotropic:bound}
There exists a function $\xi:(1/2,1]\to(1,3]$ such that $\xi(x)=1+O((1/2-x)\log(x-1/2))$ as $x\to1/2+$ and the following holds. Let $s\le t$ with $s\le 1/(C_1p)$ and $t>1$. Set $m=\lceil(s+t)/2\rceil$ and 
\begin{equation}
\label{eq:def:x}
x(s,t)=\max\left(\frac{m-1}{t-1},\frac12+(6sp)^{1/(2C_0)}\right).\end{equation}
If $R=R(s,t)$, then
\[\bbP_p(\cI(R))\le3(sp\xi(x(s,t)))^{m}.\]
\end{lemma}
\begin{proof}
We show the inequality with 
\begin{align*}
\xi(x)&{}=\frac{1+T+T^2}{T^{1/x}}\in(1,3],&T&{}=\frac{\sqrt{x^2+6x-3}+1-x}{2(2x-1)}\in[1,\infty)\end{align*}
being the positive root of the equation
\[(2x-1)T^2-T(1-x)-1=0.\]
This will conclude the proof, since $T=1/(4(x-1/2))+O(1)$, so that $\xi(x)=1-4(x-1/2)\log(x-1/2)+O(x-1/2)$ as $x\to1/2+$.

We have that $\cI(R)$ implies that $R$ is North- and South-traversable by \crefdefpart{obs:traversability}{obs:travers:infect} and contains at least $m$ infections by \cref{lem:perimeter}. Let us denote 
\[\cI'(R)=\left\{A\subset\bbZ^2:|R\cap A|\ge m\right\}\cap\cT_\uparrow(R)\cap\cT_\downarrow(R)\supset\cI(R).\]
Observe that from any configuration in $\cI'(R)$ one can extract a set of at least $m$ and at most $t$ infections witnessing $\cI'(R)$. We next seek to upper bound the number of such reduced configurations in $\cI'(R)$. Given a configuration with $k$ infections in $R$ we denote the infections by $(x_i,y_i)_{i=1}^k$ with $y_{i+1}\ge y_{i}$ for all $i\in\{1,\dots,k-1\}$. Clearly, $(x_i)_{i=1}^k\in\{0,\dots,s-1\}^k$, while North and South-traversability imply that $(y_{i+1}-y_i)_{i=1}^{k-1}\in\{0,1,2\}^{k-1}$, $y_1=0$ and $y_k=t-1$. Thus, $t-1=y_k-y_1=\sum_{i=1}^{k-1}(y_{i+1}-y_i)$. By the exponential Markov inequality (Cram\'er's theorem)
\begin{align}
\nonumber
\bbP_p(\cI(R))&{}\le \bbP_p(\cI'(R))\le \sum_{k=m}^{t}\frac{(3sp)^k}{3} \exp\left(-(k-1)\sup_{u\in\bbR}\left(\frac{u(t-1)}{k-1}-\log\frac{1+e^u+e^{2u}}{3}\right)\right)\\
&{}=\sum_{k=m}^{t} (sp)^k(\xi(x(k)))^{k-1}\le \sum_{k=m}^{\lceil k_0\rceil}(sp\xi(x(k_0)))^k+\sum_{k=\lceil k_0\rceil +1}^t(sp\xi(x(k)))^k\label{eq:anisotropic:intermediate}
\end{align}
with $x(k)=(k-1)/(t-1)$ and 
\begin{equation}
    \label{eq:def:k0}
k_0=\max\left(m,1+(t-1)\left(1/2+(6sp)^{1/(2C_0)}\right)\right),
\end{equation}
since $\xi$ is increasing (it is the logarithm of the Legendre transform of a convex increasing function).

The first sum in \cref{eq:anisotropic:intermediate} is bounded by $2(sp\xi(x(k_0)))^m$, since $\xi(x(k_0))\le 3\le 1/(2sp)$. We claim that the second one is bounded by $(sp\xi(x(k_0)))^{k_0}$. To see this, it suffices to show that for any $k\in[k_0,t-1]$ it holds that
\[2sp\le\frac{(\xi(x(k)))^k}{(\xi(x(k+1)))^{k+1}}.\]
But for any $k\in[k_0,t-1]$ we have
\begin{align*}
\frac{(\xi(x(k)))^k}{(\xi(x(k+1)))^{k+1}}&{}\ge \frac13 \left(1-\frac{\xi(x(k+1))-\xi(x(k))}{\xi(x(k+1))}\right)^{t-1}\ge \frac13\left(1-\frac{\max_{x\in[x(k_0),1]}\xi'(x)}{t-1}\right)^{t-1}\\
&{}\ge\frac13(x(k_0)-1/2)^{2C_0}\ge 2sp,\end{align*}
since $\xi'(x)\le -C_0\log (x-1/2)$ for all $x\in(1/2,1]$ and using \cref{eq:def:k0}.

Putting this together and recalling \cref{eq:def:x}, we obtain the desired inequality:
\[\bbP_p(\cI(R))\le 3(sp\xi(x(k_0)))^m+(sp\xi(x(k_0)))^{k_0}\le 3(sp\xi(x(k_0)))^m=3(sp\xi(x(s,t)))^m.\qedhere\]
\end{proof}

\subsection{Reduction}
\label{subsec:reduction}
We fix the following length scale, sizes above which we view as ``very supercritical''
\begin{equation}
\label{eq:def:Lambda}
\Lambda=\frac{C_1\log(1/p)}{p}\le q^{-5/4}.
\end{equation}
For a rectangle $R=R(a,b)$ we denote by 
\[\bar R=R(3a,3b)-(a,b),\]
its three-fold expansion. We set
\begin{align*}\bar\cI_{\mathrm{loc}}(R)&{}=\bigcup_{R\subset R'\subset\bar R}\cI_{\mathrm{loc}}(R'),&
\Phi(R)&{}=\Phi(a,b)=\frac{\bbP_p(\cI(R))}{\bbP_p(\bar\cI_{\mathrm{loc}}(R))},\\
\bar\cI^{\mathrm{F}}_{\mathrm{loc}}(R)&{}=\bigcup_{R\subset R'\subset\bar R}\cI_{\mathrm{loc}}^{\mathrm{F}}(R'),&
\Phi^{\mathrm{F}}(R)&{}=\Phi^{\mathrm{F}}(a,b)=\frac{\bbP_p(\cI^{\mathrm{F}}(R))}{\bbP_p(\bar\cI^{\mathrm{F}}_{\mathrm{loc}}(R))}.\end{align*}
We are now ready to state the main technical step in the reduction to local bootstrap percolation, which is the heart of the present work.
\begin{proposition}[Critical internal filling is local]
\label{prop:reduction}
For any rectangle $R=R(a,b)$ with $a,b\le \Lambda$ we have
\begin{align*}
\Phi(R)&{}\le e^{\log^{19}(1/p)},&\Phi^F(R)&{}\le e^{\log^{19}(1/p)}.\end{align*}
\end{proposition}
\begin{proof}
We begin with the harder case of two-neighbour bootstrap percolation. We may assume that $\min(a,b)\ge 3$, as otherwise $\cI(R)=\cI_{\mathrm{loc}}(R)\subset\bar\cI_{\mathrm{loc}}(R)$ deterministically, so there is nothing to prove. The proof proceeds recursively on the size of $R$. Let $\cN=\cI(R)\setminus\bar\cI_{\mathrm{loc}}(R)$, so that our main goal is to upper bound $\bbP_p(\cN)$.

For rectangles $S,T\subsetneq R$ with $S\not\subset T$ and $T\not\subset S$, we call $(S,T)$ a \emph{decomposition} of $R$ if $d(S,T)\le 2$. Given a decomposition $(S,T)$ we let $c,d,s,t$ be such that $S$ is a translate of $R(c,d)$ and $T$ is a translate of $R(s,t)$. We say that a decomposition $(S,T)$ is \emph{local} if $\min(c,d,s,t)=1$, $\max(c,s)=a$ or $\max(d,t)=b$.

Let $\cD$ be the set of non-local decompositions of $R$. Recalling \cref{eq:def:C}, we next claim that 
\begin{equation}
\label{eq:I:decomposition}
\cN\subset \cI(R)\setminus\cI_{\mathrm{loc}}(R)\subset \bigcup_{(S,T)\in\cD}\cI(S)\circ\cI(T)\circ\cC([S\cup T],R).\end{equation}
To prove \cref{eq:I:decomposition}, we apply \cref{lem:decomposition} to $R$. Assume the resulting decomposition $S,T$ is local. If $c=a$, then simply notice that $\cI(T)$ implies $\cT_\uparrow(R\setminus S)\cap\cT_\downarrow(R\setminus S)$, so that $\cI(T)\subset \cC(S,R)$.\footnote{Here we observe that $R\setminus S$ is either one or two rectangles of width $a$ and in the latter case mean that each of them is North- and South-traversable.} The cases $s=a$, and $\max(d,t)=b$ are treated analogously. If $s=1$ and $c<a$, then we can also check that $\cI(T)\subset\cC(S,R)$. The remaining cases are treated analogously. Thus, in total, if the decomposition given by \cref{lem:decomposition} is local, we can find a rectangle $S\subsetneq R$ such that $\cI(S)\cap\cC(S,R)$ occurs. We may then repeat the same reasoning for $S$ instead of $R$ until we reach a non-local decomposition or a decomposition consisting of two rectangles consisting of a single site each. In the latter case \cref{obs:stacking} gives that there exists $x\in A\cap R$ such that $\cC(\{x\},R)$ occurs, which is exactly $\cI_{\mathrm{loc}}(R)$ (recall \cref{eq:def:Iloc,eq:def:C}). This completes the proof of \cref{eq:I:decomposition}. 

By \cref{lem:BK,eq:I:decomposition}, we seek to upper bound 
\begin{equation}
\label{eq:sum:decompositions}
\frac{\bbP_p(\cN)}{\bbP_p(\bar\cI_{\mathrm{loc}}(R))}\le \sum_{(S,T)\in\cD}\frac{\bbP_p(\cI(S))\bbP_p(\cI(T))\bbP_p(\cC([S\cup T],R)}{\bbP_p(\bar\cI_{\mathrm{loc}}(R))}.\end{equation}
Let us fix $(S,T)\in\cD$ and without loss of generality assume that $s\le\min(c,d,t)$. We next consider several cases for the values of $c,d,s,t$, as illustrated in \cref{fig:cases}. It is convenient to fix a translate $T'$ of $T$ such that $S+T'=\{x+y:x\in S,y\in T'\}\supset[S\cup T]$.

\begin{figure}
    \centering
\subcaptionbox{\label{fig:1}Cases 1, 2.1, 3}[.19\textwidth]{
\centering  
\begin{tikzpicture}
        \draw (0,0) rectangle (1,1.3);
        \draw (0.5,0.65) node{$S$};
        \draw (0.5,0) node[below]{$c$};
        \draw (0,0.65) node[left]{$d$};
        \draw (0.9,1.2) rectangle (1.8,2.4);
        \draw (1.35,1.8) node{$T$};
        \draw (1.35,2.4) node[above]{$s$};
        \draw (1.8,1.8) node[right]{$t$};
    \end{tikzpicture}
    }  
\subcaptionbox{\label{fig:2.2}Case 2.2}[.19\textwidth]{
\centering
    \begin{tikzpicture}
        \draw (0,0) rectangle (1,1.3);
        \draw (0.5,0.65) node{$S$};
        \draw (0.5,0) node[below]{$c$};
        \draw (0,0.65) node[left]{$d$};
        \draw (0.95,1.2) rectangle (1.05,2.4);
        \draw (0.95,1.8) node[left]{$T$};
        \draw (1,2.4) node[above]{$s$};
        \draw (1.05,1.8) node[right]{$t$};
    \end{tikzpicture}}
\subcaptionbox{\label{fig:4.1}Case 4.1}[.13\textwidth]{
\centering
    \begin{tikzpicture}
        \draw (0,0) rectangle (0.4,2.3);
        \draw (0.2,1.15) node{$S$};
        \draw (0.2,0) node[below]{$c$};
        \draw (0,1.15) node[left]{$d$};
        \draw (0.35,2.2) rectangle (0.5,2.4);
        \draw (0.475,2.2)--(0.7,1.4) node[below]{$T$};
        \draw (0.425,2.4) node[above]{$s$};
        \draw (0.5,2.3) node[right]{$t$};
    \end{tikzpicture}}
\subcaptionbox{\label{fig:4.2}Case 4.2}[.22\textwidth]{
\centering
    \begin{tikzpicture}
        \draw (0,0) rectangle (2.3,0.4);
        \draw (1.15,0.2) node{$S$};
        \draw (0,0.2) node[left]{$d$};
        \draw (1.15,0) node[below]{$c$};
        \draw (2.2,0.35) rectangle (2.4,0.6);
        \draw (2.2,0.5)--(1.4,0.7) node[left]{$T$};
        \draw (2.4,0.425) node[right]{$t$};
        \draw (2.3,0.6) node[above]{$s$};
    \end{tikzpicture}}
\subcaptionbox{\label{fig:4.3}Case 4.3}[.22\textwidth]{
\centering
     \begin{tikzpicture}
        \draw (0,0) rectangle (2,2.3);
        \draw (1,1.15) node{$S$};
        \draw (1,0) node[below]{$c$};
        \draw (0,1.15) node[left]{$d$};
        \draw (1.95,2.2) rectangle (2.1,2.4);
        \draw (2.075,2.2)--(2.3,1.4) node[below]{$T$};
        \draw (2.025,2.4) node[above]{$s$};
        \draw (2.1,2.3) node[right]{$t$};
    \end{tikzpicture}}
    \caption{\label{fig:cases}Illustration of the relative sizes and aspect ratios of the rectangles $S$ and $T$ in the various cases of the proof of \cref{prop:reduction}.}   
    \end{figure}
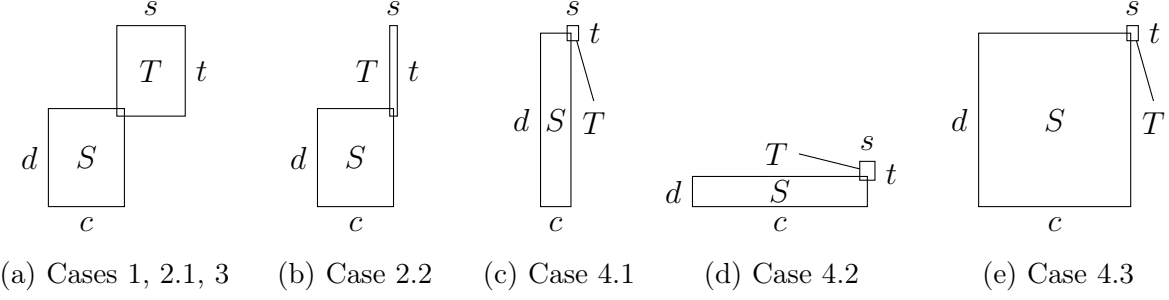

\noindent\textbf{Case 1} (Critical $S$ and $T$). Assume that $s\ge 1/(C_2p)$ (see \cref{fig:1}).
Then \cref{prop:Holroyd:upper,prop:Holroyd:lower} give
\begin{equation}
\label{eq:large:decomposition}\frac{\bbP_p(\cI(S))\bbP_p(\cI(T))}{\bbP_p(\cI_{\mathrm{loc}}(S+T'))}\le \frac{1}{p^{3}}\exp\left(\frac{2}{C_3p}-W_p(\gamma_S)-W_p(\gamma_T)+W_p(\gamma_{S+T'})\right)\end{equation}
However, by \cref{lem:optimal:path}
\[W_p(\gamma_{S+T'})\le W_p(\gamma_T)+W_p((s,t)+\gamma_{S})\]
Moreover, recalling \cref{eq:gammaR}, $s\ge1/(C_2p)\ge 1/(C_2q)$ and the fact that $g$ is convex decreasing, we have
\begin{multline*}
W_p(\gamma_S)-W_p((s,t)+\gamma_S)\\
\begin{aligned}[t]&{}\ge W_p((0,0),(1/(C_2p),1/(C_2p)))-W_p((s,t),(s+1/(C_2p),t+1/(C_2p)))\\
&{}\ge \int_0^{1/C_2} \frac{2g(z)-g(z+sq)-g(z+tq)}{q}\md z\ge 2\frac{g(1/C_2)-g(sq+1/C_2)}{C_2q}\\
&{}\ge 2\frac{g(1/C_2)-g(2/C_2)}{C_2q}\ge\frac{1}{2C_2p},\end{aligned}
\end{multline*}
using \cref{eq:g:asymptotics} in the last inequality. Putting this together and recalling that $C_3\gg C_2$, we get that 
\begin{equation}
\label{eq:reduction:1}\frac{\bbP_p(\cI(S))\bbP_p(\cI(T))\bbP_p(\cC([S\cup T],R))}{\bbP_p(\bar\cI_{\mathrm{loc}}(R))}\le \exp(-1/(3C_2p)),\end{equation}
by observing that $[S\cup T]\subset S+T'\subset \bar R$ and using the Harris inequality \cref{lem:Harris} to deduce that
\begin{equation}
\label{eq:Harris:application}
\bbP_p(\cI_{\mathrm{loc}}(S+T'))\bbP_p(\cC([S\cup T],R))\le \bbP_p(\cI_{\mathrm{loc}}([(S+T')\cup R]))\le \bbP_p(\bar\cI_{\mathrm{loc}}(R)).\end{equation}

\noindent\textbf{Case 2} (Subcritical non-microscopic $T$). Assume that $s\le 1/(C_2p)$ and $t\ge C_2\log(1/p)$. By \cref{lem:traversability}
\begin{align}
\nonumber \bbP_p(\cC(S,S+T'))&{}\ge \bbP_p(\cT_\rightarrow(R(s,d)))\bbP_p(\cT_\uparrow(R(c+s,t)))\\
&{}\ge p^2\exp(-g(dq)(s-1)-g((c+s)q)(t-1))\ge p^2(sp)^{s/2}(2sp)^{t/2},\label{eq:P:bound}\end{align}
since $g(z)<-\frac{1}{2}\log z$ for any $z>0$ small enough and $\min(d,c)\ge s$.

\noindent\textbf{Case 2.1} (Small aspect ratio). Assume that $t\le C_1s$ (see \cref{fig:1}). Then we can apply \cref{lem:isotropic:bound,eq:P:bound} to get
\[\bbP_p(\cI(T))\le (sp)^{\lceil (s+t)/2\rceil}\le \bbP_p(\cC(S,S+T')) e^{-t/C_1}.\]
Note that for any $S\subset R$ we have
\begin{equation}
\label{eq:S-to-R:crossing}\bar\cI_{\mathrm{loc}}(S)\cap\cC(S,S+T')\cap\cC([S\cup T],R)\subset \bar\cI_{\mathrm{loc}}(R).\end{equation}
Then the Harris inequality \cref{lem:Harris} gives
\begin{equation}
\label{eq:reduction:2}
\frac{\bbP_p(\cI(S))\bbP_p(\cI(T))\bbP_p(\cC([S\cup T],R))}{\bbP_p(\bar\cI_{\mathrm{loc}}(R))}\le e^{-t/C_1}\Phi(S).
\end{equation}

\noindent\textbf{Case 2.2} (Large aspect ratio). Assume that $t\ge C_1s$ (see \cref{fig:2.2}). Then we can apply \cref{lem:anisotropic:bound} to get
\[\bbP_p(\cI(T))\le 3(sp\xi(x(s,t)))^{\lceil (s+t)/2\rceil}\le \bbP_p(\cC(S,S+T')) e^{-t/C_1},\]
since $\lim_{x\to1/2}\xi(x)=1<2$. As in Case 2.1. this yields \cref{eq:reduction:2}.

\noindent\textbf{Case 3} (Microscopic $S$ and $T$).
Assume that $\max(c,d,t)\le\log^9(1/p)$ (see \cref{fig:1}). Then 
\begin{align}
\nonumber
\frac{\bbP_p(\cI(S))\bbP_p(\cI(T))\bbP_p(\cC([S\cup T],R))}{\bbP_p(\bar\cI_{\mathrm{loc}}(R))}
&{}\le \frac{\bbP_p(\cC([S\cup T],R))}{\bbP_p(\cI_{\mathrm{loc}}(S+T'))\bbP_p(\cC([S\cup T],R))}\\
&{}\le p^{-|S+T'|}\le e^{\log^{19}(1/p)/3},\label{eq:reduction:4}
\end{align}
using \cref{eq:Harris:application} as above.

\noindent\textbf{Case 4} (Microscopic $T$, non-microscopic $S$). Assume that 
\begin{align*}
t&{}\le C_2\log(1/p),&\max(c,d)&{}\ge\log^9(1/p).\end{align*}

\noindent\textbf{Case 4.1} (Tall $S$). Assume that $c\le \log^4(1/p)$, so that $d\ge \log^9(1/p)$, since $\max(t,c,d)\ge \log^9(1/p)$ (see \cref{fig:4.1}). In this case we proceed as in Case 1, but using \cref{lem:anisotropic:bound} instead of \cref{prop:Holroyd:upper}. Namely, by \cref{lem:anisotropic:bound} we have
\begin{align}
\nonumber
\bbP_p(\cI(T))\bbP_p(\cI(S))&{}\le 3(3sp)^{\lceil (s+t)/2\rceil}\cdot3(cp\xi(x(c,d)))^{\lceil(c+d)/2\rceil}\\
&{}\le 9\left(3C_2p\log(1/p)\right)^{\lceil (s+t)/2\rceil}\left(cp\left(1+\log^{-9/2}(1/p)\right)\right)^{\lceil(c+d)/2\rceil}\label{eq:tall:S:1}
\end{align}
On the other hand, recalling that $f(z)=-\log z+O(z)$ and $q=p+O(p^2)$, we get
\begin{align}\nonumber\bbP_p(\cI_{\mathrm{loc}}(S+T'))&{}\ge p^{c+s}\exp(-\lceil(d+t-c-s)/2\rceil f((c+s)q))\\
\nonumber&{}\ge p^{c+s}((c+s)p)^{\lceil(d+t-c-s)/2\rceil}e^{-O(cdp)}\\
&{}\ge p^{\lceil(s+t)/2\rceil+\lceil(c+d)/2\rceil}\left(e^{-O(p\log^4(1/p))}(c+1)\right)^{d/2-\log^4(1/p)},\label{eq:tall:S:2}\end{align}
asking for the entire diagonal of the top-most translate of $R(c+s,c+s)$ contained in $S+T'$ to be infected and for an infection on every second row in the rest of $S+T'$.

Dividing \cref{eq:tall:S:1} and \cref{eq:tall:S:2}, we see that the terms $p^{\lceil (s+t)/2\rceil}$ and $p^{\lceil (c+d)/2\rceil}$ cancel out, giving
\begin{multline*}
\frac{\bbP_p(\cI(S))\bbP_p(\cI(T))}{\mathbb P_p(\cI_{\mathrm{loc}}(S+T'))}\\
\begin{aligned}[t]&{}\le 9  \left(3C_2\log(1/p)\right)^{2C_2\log(1/p)}\frac{(ce^{2\log^{-9/2}(1/p)})^{\lceil(c+d)/2\rceil}}{(c+1)^{d/2-\log^4(1/p)}}\\
&{}\le e^{\log^2(1/p)}(2c)^{2\log^4(1/p)}\exp\left(\left(d/2-\log^4(1/p)\right)\left(-\frac12\log^{-4}(1/p)+2\log^{-9/2}(1/p)\right)\right)\\
&{}\le\exp\left(\log^{9/2}(1/p)-d/\left(5\log^{4}(1/p)\right)\right)\le \exp\left(-\log^5(1/p)/6\right).\end{aligned}
\end{multline*}
By \cref{eq:Harris:application}, this yields
\begin{equation}
\label{eq:reduction:5}
\frac{\bbP_p(\cI(S))\bbP_p(\cI(T))\bbP_p(\cC([S\cup T],R))}{\bbP_p(\bar\cI_{\mathrm{loc}}(R))}\le \exp\left(-\log^{5}(1/p)/6\right).\end{equation}

\noindent\textbf{Case 4.2} (Wide $S$). Assume that $d\le \log^4(1/p)$, so that $c\ge \log^9(1/p)$, since $\max(t,c,d)\ge \log^9(1/p)$ (see \cref{fig:4.2}). The treatment is essentially identical to Case 4.1., so we omit most of the computations. \Cref{lem:anisotropic:bound} gives
\[\frac{\bbP_p(\cI(S))\bbP_p(\cI(T))}{\bbP_p(\cI_{\mathrm{loc}}(S+T'))}\le \frac{3(3sp)^{\lceil (s+t)/2\rceil}\cdot3(dp\xi(d,c))^{\lceil(c+d)/2\rceil}}{p^{d+t}\exp(-\lceil(c+s-d-t)/2\rceil f((d+t)q))}\le \exp\left(-\log^{5}(1/p)/6\right),\]
as in the previous case. By \cref{eq:Harris:application} this still yields \cref{eq:reduction:5}.

\noindent\textbf{Case 4.3} (Squarish $S$). Assume that $\min(c,d)\ge \log^4(1/p)$ (see \cref{fig:4.3}). We apply \cref{lem:anisotropic:bound} to get
\begin{equation}
\label{eq:PpIT}\bbP_p(\cI(T))\le 3(3sp)^{\lceil(s+t)/2\rceil}.\end{equation}
In order to efficiently compare this quantity to $\bbP_p(\cC(S,S+T'))$, we need to consider the parity of $s$ and $t$, since we cannot afford even a single additional infection for rectangles as small as $T$.

\noindent\textbf{Case 4.3.1} ($T$ with an even side). Assume that $st$ is even, so that $\lceil (s+t)/2\rceil=\lceil s/2\rceil+\lceil t/2\rceil$.  Then
\begin{align*}
\bbP_p(\cC(S,S+T'))&{}\ge \exp(-f(dq)\lceil s/2\rceil-f(cq)\lceil t/2\rceil)\\
&{}\ge \frac{cdp^2}{C^2_2\log^2(1/p)}\left(\frac{p\log^3(1/p)}{C_2}\right)^{\lceil s/2\rceil+\lceil t/2\rceil-2},\end{align*}
where we used that for any $x\le \Lambda$ it holds that $e^{-f(xq)}\ge xp/(C_2\log(1/p))$. Combining this with \cref{eq:PpIT}, we get
\[\frac{\bbP_p(\cI(T))}{\bbP_p(\cC(S,S+T'))}\le \frac{3\log^8(1/p)}{cd}\left(\frac{3C_2^2}{\log^2(1/p)}\right)^{\lceil(s+t)/2\rceil}.\]
Using \cref{eq:S-to-R:crossing} like in Case 2, this gives
\begin{equation}
\label{eq:reduction:6}
\frac{\bbP_p(\cI(S))\bbP_p(\cI(T))\bbP_p(\cC([S\cup T],R))}{\bbP_p(\bar\cI_{\mathrm{loc}}(R))}\le \frac{3\log^8(1/p)}{cd}\left(\frac{3C_2^2}{\log^2(1/p)}\right)^{\lceil(s+t)/2\rceil}\Phi(S).
\end{equation}

\noindent\textbf{Case 4.3.2} ($T$ with odd sides). Assume that $s$ and $t$ are odd and recall that $\min(s,t)\ge 2$, since $(S,T)\in\cD$. Then
\begin{align*}\bbP_p(\cC(S,S+T'))&{}\ge p\exp(-f(dq)(s-1)/2-f(cq)(t-1)/2)\\
&{}\ge \frac{cdp^3}{C_2^2\log^2(1/p)}\left(\frac{p\log^3(1/p)}{C_2}\right)^{s/2+t/2-3},\end{align*}
as in Case 4.3.1. By \cref{eq:PpIT} this entails
\begin{align*}
\nonumber
\frac{\bbP_p(\cI(T))}{\bbP_p(\cC(S,S+T'))}&{}\le \frac{3^4C_2^5\log^5(1/p)}{cd}\left(\frac{3C_2\log(1/p)p}{p\log^3(1/p)/C_2}\right)^{(s+t)/2-3}\\
&{}= \frac{3\log^{11}(1/p)}{C_2cd}\left(\frac{3C_2^2}{\log^2(1/p)}\right)^{(s+t)/2}.
\end{align*}
As in the previous case this gives
\begin{equation}
\label{eq:reduction:7}
\frac{\bbP_p(\cI(S))\bbP_p(\cI(T))\bbP_p(\cC([S\cup T],R))}{\bbP_p(\bar\cI_{\mathrm{loc}}(R))}\le \frac{3\log^{11}(1/p)}{C_2cd}\left(\frac{3C_2^2}{\log^2(1/p)}\right)^{(s+t)/2}\Phi(S),
\end{equation}
concluding the final case.

We are now ready to conclude the proof of \cref{prop:reduction}. Plugging \cref{eq:reduction:1,eq:reduction:2,eq:reduction:4,eq:reduction:5,eq:reduction:6,eq:reduction:7} into \cref{eq:sum:decompositions}, we get
\begin{align*}
\frac{\bbP_p(\cN)}{\bbP_p(\bar\cI_{\mathrm{loc}}(R))}&{}\le\begin{multlined}[t]\sum_{(S,T)\in\cD}\Bigg(e^{-1/(3C_2p)}+e^{\log^{19}(1/p)/3}+e^{-\log^5(1/p)/6}\\
+\Phi(S)\left(\frac{\1_{t\ge C_2\log(1/p)}}{e^{t/C_1}}+\frac{\log^{11}(1/p)}{cd}\left(\frac{3C_2^2}{\log^2(1/p)}\right)^{\lceil(s+t)/2\rceil}\right)\Bigg)
\end{multlined}\\
&{}\le e^{\log^{19}(1/p)/2}+\sum_{c=2}^{a-1}\sum_{d=2}^{b-1}
\Phi(c,d)\frac{\log^{8}(1/p)}{cd},\end{align*}
where in the second inequality we used that there are clearly at most $p^{-5}$ decompositions, given that $\max(a,b)\le \Lambda$ (recall \cref{eq:def:Lambda}), as well as $\min(s,t)\ge 2$, since $(S,T)\in\cD$. Hence,
\begin{align*}\Phi(a,b)&{}\le 1+e^{\log^{19}(1/p)/2}+\sum_{c=2}^{a-1}\sum_{d=2}^{b-1}\frac{\log^{8}(1/p)}{cd}\Phi(c,d)\\
&{}\le \left(1+e^{\log^{19}(1/p)/2}\right)\sum_{k=0}^{\Lambda}\sum_{(c_i)_{i=1}^k,(d_i)_{i=1}^k}\prod_{i=1}^k\frac{\log^{8}(1/p)}{c_id_i}\\
&{}\le \left(1+e^{\log^{19}(1/p)/2}\right)\left(\prod_{c=2}^{\Lambda}\left(1+\frac{\log^4(1/p)}{c}\right)\right)^2\le e^{\log^{19}(1/p)}\end{align*}
where the sum runs over increasing positive integer sequences $(c_i)_{i=1}^k,(d_i)_{i=1}^k$ such that $\max(c_k,d_k)\le \Lambda$. This concludes the proof of \cref{prop:reduction} for the two-neighbour model.

The proof of \cref{prop:reduction} for Frob\"ose bootstrap percolation is almost identical, so we only indicate the changes required. Firstly, \cref{eq:sum:decompositions} remains valid and so do the proofs in Cases 1.\ and 3., up to requiring $d(S,T)\le 1$ for decompositions and replacing $g$ by $f$, $W$ by $W^{\mathrm{F}}$, etc.). In Case 2., $\lceil(s+t)/2\rceil$ becomes $s+t-1$ in view of \cref{lem:isotropic:bound,lem:anisotropic:bound:frobose}, while \cref{eq:P:bound} transforms into
\begin{align*}
\bbP_p\left(\cC^{\mathrm{F}}(S,S+T')\right)&{}\ge \bbP_p\left(\cG_|(R(s,d))\right)\bbP_p\left(\cG_-(R(c+s,t))\right)= \exp(-f(dq)s-f((c+s)q)t)\\
&{}\ge (sp)^{s}(2sp)^{t}e^{-O(spt)}\ge p(sp)^{s+t-1}2^{t(1-1/C_1)},
\end{align*}
taking \cref{eq:f:asymptotics} and $s\le 1/(C_2p)$ into account. Finally, Case 4.\ needs a little more care, so we write it out for the reader's convenience.

Let $(S,T)$ be a non-local (Frob\"ose) decomposition of $R$ with $S$ a translate of $R(c,d)$ and $T$ a translate of $R(s,t)$.  Note that for Frob\"ose bootstrap percolation $d(S,T)\le 1$, so we can find a rectangle $T'$ which is a translate of $R(s',t')$ such that $[S\cup T]^{\mathrm F}\subset S+T'$ and $(s',t')\in\{(s,t-1),(s-1,t)\}$. We further assume that $s\le\min(c,d,t)$, $t\le C_2\log(1/p)$ and $\max(c,d)\ge \log^9(1/p)$ (corresponding to Case 4.).

\noindent\textbf{Case 4.1-F} (Tall $S$). Assume that $c\le \log^4(1/p)$, so that $d\ge \log^9(1/p)$, since we have $\max(t,c,d)\ge \log^9(1/p)$. By \cref{lem:anisotropic:bound:frobose} we have
\begin{align}
\nonumber
\bbP_p\left(\cI^{\mathrm F}(S)\right)\bbP_p\left(\cI^{\mathrm F}(T)\right)&{}\le (2sp)^{s+t-1}\left(cp\xi^{\mathrm F}\left(\frac{c+d-1}{d-1}\right)\right)^{c+d-1}\\
&{}\le (2C_2p\log(1/p))^{s+t-1}\left(cp\left(1+\log^{-9/2}(1/p)\right)\right)^{c+d-1}\label{eq:tall:S:1:F}
\end{align}
On the other hand, by \crefdefpart{obs:traversability}{obs:travers:germ_boundary}, $\cI_{\mathrm{loc}}^{\mathrm F}(S+T')$ is implied by the presence of an L-shaped set of $2(c+\min(s',t'))-1$ infections at the bottom-left corner of $S+T'$ and the top-most translate of $R(c+s',d+t'-c-s')$ contained in $S+T'$ having no horizontal gap. Therefore, recalling that $f(z)=-\log z+O(z)$ and $q=p+O(p^2)$, we get
\begin{align}\nonumber
\bbP_p\left(\cI^{\mathrm F}_{\mathrm{loc}}(S+T')\right)&{}\ge p^{2c+2s'-1}\exp(-(d+t'-c-s')f((c+s')q))\\
\nonumber&{}\ge p^{2c+2s'-1}((c+s')p)^{d+t'-c-s'}e^{-O(cdp)}\\
&{}\ge p^{c+d-1+s+t-1}\left(e^{-O(p\log^4(1/p))}(c+1)\right)^{d-\log^4(1/p)},\label{eq:tall:S:2:F}\end{align}
since $s'+t'=s+t-1$.

Combining \cref{eq:tall:S:1,eq:tall:S:2:F}, we get
\begin{multline*}
\frac{\bbP_p(\cI^{\mathrm F}(S))\bbP_p(\cI^{\mathrm F}(T))}{\mathbb P_p(\cI^{\mathrm F}_{\mathrm{loc}}(S+T'))}\\
\begin{aligned}[t]&{}\le \left(2C_2\log(1/p)\right)^{2C_2\log(1/p)}\frac{(ce^{2\log^{-9/2}(1/p)})^{c+d-1}}{(c+1)^{d-\log^4(1/p)}}\\
&{}\le e^{\log^2(1/p)}(2c)^{2\log^4(1/p)}\exp\left(\left(d-\log^4(1/p)\right)\left(-\frac12\log^{-4}(1/p)+2\log^{-9/2}(1/p)\right)\right)\\
&{}\le\exp\left(\log^{9/2}(1/p)-d/\left(3\log^{4}(1/p)\right)\right)\le \exp\left(-\log^5(1/p)/6\right).\end{aligned}\end{multline*}
By \cref{eq:Harris:application} this yields
\begin{equation}
\label{eq:reduction:5:F}
\frac{\bbP_p(\cI^{\mathrm F}(S))\bbP_p(\cI^{\mathrm F}(T))\bbP_p(\cC^{\mathrm F}([S\cup T],R))}{\bbP_p(\bar\cI^{\mathrm F}_{\mathrm{loc}}(R))}\le \exp\left(-\log^{5}(1/p)/6\right).\end{equation}

The Case 4.2-F. of $d\le \log^4(1/p)$ and $c\ge \log^9(1/p)$ is analogous and therefore omitted.

\noindent\textbf{Case 4.3-F} (Squarish $S$) Assume that $\min(c,d)\ge \log^4(1/p)$. By \cref{lem:anisotropic:bound:frobose}
\[\bbP_p\left(\cI^{\mathrm F}(T)\right)\le(2sp)^{s+t-1}.\]
Moreover, using that for any $x\le \Lambda$, $e^{-f(xq)}\ge xp/(C_2\log(1/p))$, we get
\[\bbP_p\left(\cC^{\mathrm F}(S,S+T')\right)\ge \exp(-f(dq)s'-f(cq)t')\ge \frac{cdp^2}{C_2^2\log^2(1/p)}\left(\frac{p\log^3(1/p)}{C_2}\right)^{s'+t'-2}.\]
Combining these two bound with \cref{eq:S-to-R:crossing} as above, we obtain
\[\frac{\bbP_p(\cI^{\mathrm F}(S))\bbP_p(\cI^{\mathrm F}(T))\bbP_p(\cC^{\mathrm F}([S\cup T],R))}{\bbP_p(\bar\cI^{\mathrm F}_{\mathrm{loc}}(R))}\le \frac{\log^8(1/p)}{cd}\left(\frac{2C_2^2}{\log^2(1/p)}\right)^{s+t-1}\Phi^{\mathrm F}(S).\]
This concludes Case 4.3-F. The rest of the proof of \cref{prop:reduction} for Frob\"ose bootstrap percolation is identical to the one for two-neighbour model.
\end{proof}

The next lemma is fairly standard (morally going back to \cite{Aizenman88}), but we include the proof for completeness and in order to state it in a convenient form.
\begin{lemma}[Critical internal filling determines $\tau$]
\label{lem:tau}
Let $R=R(\Lambda,\Lambda)$. Then
\begin{equation}
\lim_{p\to0}\bbP_p\left(\frac{|\log(\tau\sqrt{\bbP_p(\cI(R))})|}{\log (1/p)}\le 7\right)=1,
\label{eq:tau:Fp}
\end{equation}
The same holds for $\tau_{\mathrm{loc}}$ with $\cI$ replaced by $\cI_{\mathrm{loc}}$ or $\bar\cI_{\mathrm{loc}}$. The same holds for Frob\"ose bootstrap percolation.
\end{lemma}
\begin{proof}
Let us first show that 
\begin{equation}
\label{eq:rectangle:to:square}
\bbP_p(\cI(R))=(1-o(1))\max \bbP_p(\cI(R(a,b))),
\end{equation}
where the maximum runs over $a,b\ge 1$ such that $\max(a,b)/\Lambda\in[1/2,1]$. Indeed, $\bbP_p(\cI(R))=\bbP_p(\cI(R(\Lambda,\Lambda)))\le \max\bbP_p(\cI(R(a,b)))$ is clear. To prove the converse first note that by the Harris inequality \cref{lem:Harris}
\[\bbP_p(\cI(R))\ge \bbP_p(\cI(R(a,b)))\bbP_p(\cE_1)\]
for any $a,b\ge 1$ with $\max(a,b)/\Lambda\in[1/2,1]$, where $\cE_1$ is the event that every row or column of length $\Lambda/2$ in $R(\Lambda, \Lambda)$ contains an infection. This entails \cref{eq:rectangle:to:square}, since
\[\bbP_p(\cE_1)\ge 1-\Lambda^2(1-p)^{\Lambda/2}=1-o(1).\]

By the classical results of Aizenman--Lebowitz \cite{Aizenman88} we have that $\bbP_p(\tau\le e^{-1/(C_1p)})\to0$. For any $T\ge 1$, define the event $\cT(T)=\{e^{1/(C_1p)}<\tau<T\}$. If $\cT(T)$ occurs, then there exists a set $A_0\subset R(-T,-T;T,T)$ of infections such that $0\in[A_0]$. Therefore, by the rectangles process (recall \cref{subsec:rectangles:process}), there exists a rectangle $R'\subset R(-T,-T;T,T)$ such that $0\in R'$, $|R'|\ge e^{1/(C_1p)}$ and $\cI(R')$ occurs. However, by \cref{lem:AL}, this implies the existence of a rectangle $R''$ with longest side of length in $[\Lambda/2,\Lambda]$ contained in $R'$. By the union bound and \cref{eq:rectangle:to:square} we have that for any $T\ge 1$
\begin{equation}
\label{eq:tau:1}\bbP_p(\cT(T))\le 9T^2\Lambda^2\max \bbP_p(\cI(R(a,b)))\le T^2p^{-3}\bbP_p(\cI(R))/2.\end{equation}
On the other hand, the Harris inequality \cref{lem:Harris} gives that for any $T>e^{1/(C_1p)}$,
\begin{equation}
\label{eq:tau:2}\bbP_p(\cT(T))\ge \left(1-\left(1-\bbP_p(\cE_2)\bbP_p(\cI(R))\right)^{(Tp^4/\Lambda)^2}\right)\bbP_p(\cE_3),\end{equation}
where $\cE_2$ is the event that every row or column of length $\Lambda$ in $R(1/p^3,1/p^3)$ contains an infection and $\cE_3$ is the event that every row or column of length $1/p^3$ in $\bar R_T$ contains an infection, where $R_T=R(-p^4T,-p^4T;p^4T,p^4T)$. Indeed, if $\cE_3$ occurs, there exists $x\in R_T$ such that $\cI(x+R)$ occurs and $\cE_2$ translated by $x$ occurs, then $\tau\le \Lambda^2+\Lambda/p^3+Tp^4/p^3<T$. By a union bound we obtain
\begin{align}
\label{eq:tau:3}\bbP_p(\cE_2)&{}\ge 1-p^{-6}(1-p)^{\Lambda}=1-o(1),\\
\label{eq:tau:4}
\bbP_p(\cE_3)&{}\ge 1-T^2(1-p)^{1/p^3}=1-o(1),
\end{align}
for any $T\le \exp(p^{-4/3})$. Putting \cref{eq:tau:1,eq:tau:2,eq:tau:3,eq:tau:4} together and taking into account that $\bbP_p(\cI(R))=o(1)$ e.g.\ by \cite{Aizenman88}, we obtain \cref{eq:tau:Fp}, as desired.

The proofs for local and/or Frob\"ose bootstrap percolation are identical.
\end{proof}
Combining \cref{prop:reduction,lem:tau}, we obtain the following result, which allows us to completely restrict our attention to the local model in the remainder of the work.
\begin{corollary}[$\tau$ is local]
\label{cor:reduction}
Let $R=R(\Lambda,\Lambda)$ (recall \cref{eq:def:Lambda}). Then
\[\lim_{p\to 0}\bbP_p\left(e^{-\log^{20}(1/p)}\le \tau\sqrt{\bbP_p(\cI_{\mathrm{loc}}(R))}\le p^{-7}\right)=1.\]
The same holds for Frob\"ose bootstrap percolation.
\end{corollary}
\begin{proof}
Since $\tau\le \tau_{\mathrm{loc}}$, \cref{lem:tau} gives that
\[\tau\sqrt{\bbP_p(\cI_{\mathrm{loc}}(R))}\le \tau_{\mathrm{loc}}\sqrt{\bbP_p(\cI_{\mathrm{loc}}(R))}\le p^{-7}\]
with high probability. Moreover, by \cref{prop:reduction,lem:tau}, 
\[\tau\sqrt{\bbP_p\left(\bar \cI_{\mathrm{loc}}(R)\right)}\ge \tau e^{-\log^{19}(1/p)}\sqrt{\bbP_p(\cI(R))}\ge p^{7}e^{-\log^{19}(1/p)}\]
with high probability. It remains to note, by applying \cref{lem:tau} twice, that
\[\bbP_p(\cI_{\mathrm{loc}}(R))\ge p^{28}\bbP_p\left(\bar\cI_{\mathrm{loc}}(R)\right).\]
The proof for Frob\"ose bootstrap percolation is analogous.
\end{proof}

\section{Growth sequences}
\label{sec:sequences}

In view of \cref{cor:reduction}, we focus on the local Frob\"ose model for the rest of the paper and seek to prove \cref{th:main}.
In this section we analyse the detailed structure of the events $\cC^{\mathrm F}(S,R)$ for rectangles $S\subset R$.

\subsection{Buffers and Frames}
\label{subsec:frames}

As mentioned in \cref{sec:intro}, we will analyse the infection from one rectangle $R$ to a larger $R'$ by relating it to a Markov chain.
This is done using an exploration procedure that selectively reveals whether certain regions around a rectangle are occupied or not.
To make this exploration precise, we first introduce the notion of \emph{buffers} and \emph{frames}.

Given a rectangle $R = R(a, b; c, d)$, we define its right, up, left and bottom buffers respectively as
\begin{equation}
  \label{e:buffers}
  \begin{split}
    & B^r(R) = R(c,b;c+1,d), \qquad B^u(R) = R(a, d; c, d + 1)\\
    & B^l(R) = R(a - 1, b; a, d), \qquad B^d(R) = R(a, b - 1; c, b).
  \end{split}
\end{equation}
Observe that if all these buffers are free of infections, then a germ located in $R$ cannot locally infect anything outside this enclosed region.

As we inspect the buffers around a rectangle, our Markov exploration process will transition between eight classes of states, that we call \emph{frames}.
\begin{definition}[Frame]
\label{def:frame}
A \emph{framed rectangle} is a rectangle $R(a,b;c,d)$ equipped with a label $s \in \{0, 1, 1', 1'', 2, 2', 3, 4\}$. We denote it by $F(a,b;c,d;s)$ and refer to $(c-a,d-b)$ as its \emph{dimensions} and to $s$ as its \emph{frame state}. For a  framed rectangle $F=F(a,b;c,d;s)$ we set $F_\circ=R(a,b;c,d)$ and define \emph{frame} of $F$ by
\[F_\square=\begin{cases}
\varnothing&s=0,\\
B^r(F_\circ)&s=1,\\
B^r(F_\circ) \cup B^u(F_\circ) & s=2,\\
B^r(F_\circ) \cup B^u(F_\circ) \cup B^l(F_\circ) & s=3,\\
\phantom{B^r(F_\circ)\cup{}}B^u(F_\circ) \cup B^l(F_\circ) & s=2',\\
\phantom{B^r(F_\circ)\cup B^u(F_\circ) \cup {}} B^l(F_\circ) & s=1',\\
B^r(F_\circ) \cup B^u(F_\circ) \cup B^l(F_\circ) \cup B^d(F_\circ) & s=4,\\
\phantom{B^r(F_\circ)\cup{}} B^u(F_\circ) & s=1'',
\end{cases}\]
see \cref{fig:augmented:cycle}. We denote $F_\blacksquare=F_\circ\cup F_\square$.
\end{definition}

\subsection{The framed rectangle Markov chain}
\label{subsec:Markov}
We next define a Markov chain on framed rectangles. Each framed rectangle only has transitions to a bounded number of others, so that such transitions are translation invariant.
That is, the probability of the transition from $F(a,b;c,d;s)$ to $F(a',b';c',d';s')$ is the same as the one from $F(0,0;c-a,d-b;s)$ to $F(a'-a,b'-b;c'-a,d'-b;s')$.
The transition probabilities are given in \cref{tab:transitions} and illustrated in \cref{fig:augmented:cycle}.
The first seven are called \emph{buffer creations} (note that there is no such transition from $s=0$ to $s'=1'$), the next eight are called \emph{loops}, the next seven are the \emph{single buffer deletions}, the next five are the \emph{double buffer deletions}, while the last one is the \emph{triple buffer deletion}.

As we will see, only the buffer creations, loops and single buffer deletions contribute to the second term in the asymptotic behaviour described in \cref{th:main}.
Moreover, since the only transition towards frame state $1''$ is a double buffer deletion, this state (and all transitions to and from it) also does not contribute to the final result nor does the absorbing frame state $4$.

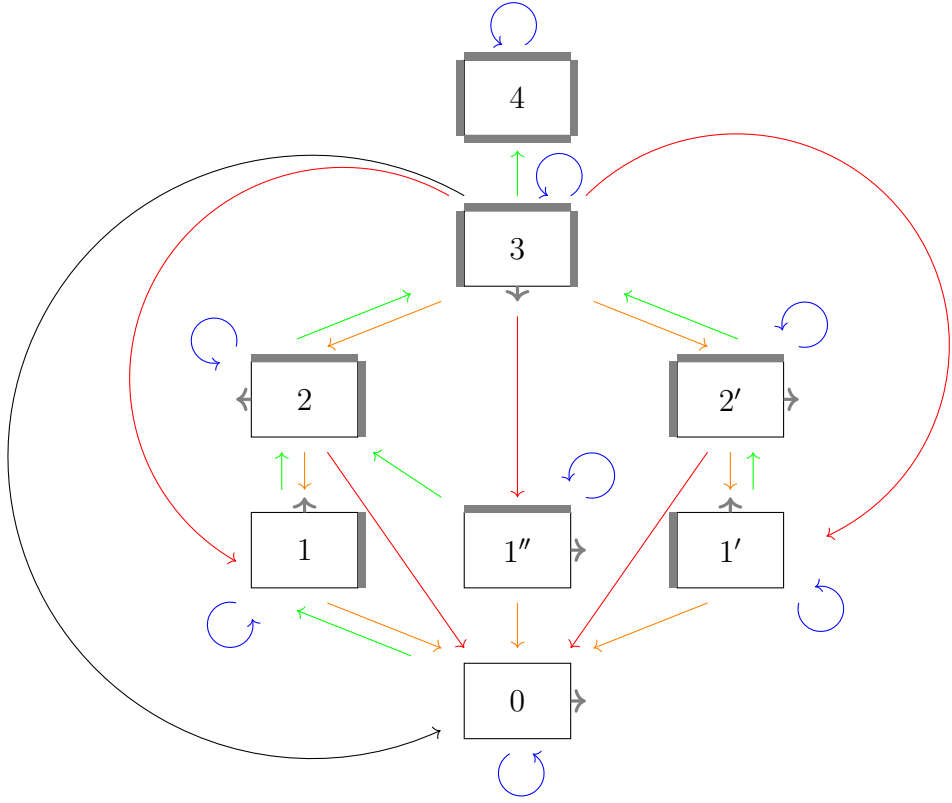
\begin{figure}
    \centering
    \begin{tikzpicture}
        \draw (0, 0) rectangle (1.4, 1) node[midway] {$0$};
        \draw[->, very thick, gray] (1.4, .5) -> ++(.2,0);
        \draw (0, 2) rectangle (1.4, 3) node[midway] {$1''$};
        \draw[->, very thick, gray] (1.4, 2.5) -> ++(.2,0);
        \fill[gray] (0, 3) rectangle (1.4, 3.1);
        \draw (0, 6) rectangle (1.4, 7) node[midway] {$3$};
        \draw[->, very thick, gray] (.7, 6) -> ++(0,-.2);
        \fill[gray] (1.4, 6) rectangle (1.5, 7);
        \fill[gray] (0, 7) rectangle (1.4, 7.1);
        \fill[gray] (-.1, 6) rectangle (0, 7);
        \draw (0, 8) rectangle (1.4, 9) node[midway] {$4$};
        \fill[gray] (1.4, 8) rectangle (1.5, 9);
        \fill[gray] (0, 9) rectangle (1.4, 9.1);
        \fill[gray] (-.1, 8) rectangle (0, 9);
        \fill[gray] (0, 7.9) rectangle (1.4, 8);
        \draw (-2.8, 2) rectangle (-1.4, 3) node[midway] {$1$};
        \draw[->, very thick, gray] (-2.1, 3) -> ++(0,.2);
        \fill[gray] (-1.4, 2) rectangle (-1.3, 3);
        \draw (-2.8, 4) rectangle (-1.4, 5) node[midway] {$2$};
        \draw[->, very thick, gray] (-2.8, 4.5) -> ++(-.2, 0);
        \fill[gray] (-1.4, 4) rectangle (-1.3, 5);
        \fill[gray] (-2.8, 5) rectangle (-1.4, 5.1);
        \draw (2.8, 2) rectangle (4.2, 3) node[midway] {$1'$};
        \draw[->, very thick, gray] (3.5, 3) -> ++(0,.2);
        \fill[gray] (2.8, 2) rectangle (2.7, 3);
        \draw (2.8, 4) rectangle (4.2, 5) node[midway] {$2'$};
        \draw[->, very thick, gray] (4.2, 4.5) -> ++(.2, 0);
        \fill[gray] (2.8, 4) rectangle (2.7, 5);
        \fill[gray] (2.8, 5) rectangle (4.2, 5.1);
        \path[->,orange] (.7, 1.8) edge (.7, 1.2);    
        \path[<-,green] (.7, 7.8) edge (.7, 7.2);    
        \path[->,orange] (-1.8, 1.8) edge (-.3, 1.2); 
        \path[<-,green] (-2.2, 1.7) edge (-.7, 1.1); 
        \path[->,orange] (3.2, 1.8) edge (1.7, 1.2);  
        \path[->,red] (3.2, 3.8) edge (1.4, 1.2);  
        \path[->,red] (-1.8, 3.8) edge (0, 1.2);   
        \path[->,red] (.7, 5.6) edge (.7, 3.2);    
        \path[->,orange] (-.3, 5.8) edge (-1.8, 5.2); 
        \path[<-,green] (-.7, 5.9) edge (-2.2, 5.3); 
        \path[->,orange] (1.7, 5.8) edge (3.2, 5.2);
        \path[<-,green] (2.1, 5.9) edge (3.6, 5.3);
        \path[->,orange] (-2.1, 3.8) edge (-2.1, 3.3); 
        \path[<-,green] (-2.4, 3.8) edge (-2.4, 3.3);
        \path[->,orange] (3.5, 3.8) edge (3.5, 3.3); 
        \path[<-,green] (3.8, 3.8) edge (3.8, 3.3);
        \path[->,green] (-.3, 3.2) edge (-1.2, 3.8);
        \draw[->,blue] (.6, -.2) arc [radius=.3, start angle=120, delta angle=300];
        \draw[->,blue] (-3, 1.8) arc [radius=.3, start angle=75, delta angle=300];
        \draw[->,blue] (-3, 5.2) arc [radius=.3, start angle=-15, delta angle=300];
        \draw[->,blue] (1.4, 7.2) arc [radius=.3, start angle=-60, delta angle=300];
        \draw[->,blue] (.8, 9.2) arc [radius=.3, start angle=-60, delta angle=300];
        \draw[->,blue] (4.4, 5.2) arc [radius=.3, start angle=255, delta angle=300];
        \draw[->,blue] (4.4, 1.8) arc [radius=.3, start angle=165, delta angle=300];
        \draw[->,blue] (1.6, 3.2) arc [radius=.3, start angle=255, delta angle=300];
        \draw[->,red] (1.6, 7.2) arc [radius=2.8, start angle=135, delta angle=-200]; 
        \draw[->,red] (-.2, 7.2) arc [radius=2.8, start angle=60, delta angle=180]; 
        \draw[->,black] (0, 7.2) arc [radius=4, start angle=60, delta angle=235]; 
    \end{tikzpicture}
    \caption{Here we represent all the frame classes as rectangles.
    Their respective buffers are featured in gray, while the transitions are shown as arrows under the following color code: green for buffer creation, blue for loops, orange for single buffer deletions, red for double buffer deletions and black for the transition with a triple buffer deletion. Each rectangle (except for the one corresponding to the absorbing state $4$) is decorated with a gray arrow that indicates the direction that will be explored next.}
    \label{fig:augmented:cycle}
\end{figure}

\begin{table}
    \centering
    \footnotesize
\begin{tabular}{c c c c c c |l||c c c c c c | l}
$s$ & $s'$ & $\alpha$ & $\beta$ & $\gamma$ & $\delta$ & $-\log \pi$ & $s$ & $s'$ & $\alpha$ & $\beta$ & $\gamma$ & $\delta$ & $-\log \pi$\\\hline\hline
0&1&0&0&0&0&$qb$&1&0&0&0&1&1&$\log(1/p)+f(qa)$\\
1&2&0&0&0&0&$qa$&2&1&1&0&0&1&$\log(1/p)+f(qb)+q$\\
2&3&0&0&0&0&$qb$&3&2&1&1&0&0&$\log(1/p)+f(qa)+2q$\\
3&4&0&0&0&0&$qa$&3&2'&0&1&1&0&$\log(1/p)+f(qa)+2q$\\
2'&3&0&0&0&0&$qb$&2'&1'&0&0&1&1&$\log(1/p)+f(qb)+q$\\
1'&2'&0&0&0&0&$qa$&1'&0&1&0&0&1&$\log(1/p)+f(qa)$\\
1''&2&0&0&0&0&$qb$&1''&0&0&0&1&1&$\log(1/p)+f(qb)$\\\hline
0&0&0&0&1&0&$f(qb)$&2&0&1&0&1&1&$2\log(1/p)+f(qb)$\\
1&1&0&0&0&1&$f(qa)+q$&2'&0&1&0&1&1&$2\log(1/p)+f(qb)$\\
2&2&1&0&0&0&$f(qb)+q$&3&1&1&1&0&1&$2\log(1/p)+f(qa)+2q$\\
3&3&0&1&0&0&$f(qa)+2q$&3&1'&0&1&1&1&$2\log(1/p)+f(qa)+2q$\\
2'&2'&0&0&1&0&$f(qb)+q$&3&1''&1&1&1&0&$2\log(1/p)+f(qa)+2q$\\\cline{8-14}
1'&1'&0&0&0&1&$f(qa)+q$&3&0&1&1&1&1&$3\log(1/p)+f(qa)-\log(4-3p)$\\
1''&1''&0&0&1&0&$f(qb)+q$\\
4&4&0&0&0&0&0
\end{tabular}
\caption{Transition probabilities $\pi=\bbP_p(\cT(F(0,0;a,b;s),F(-\alpha,-\beta;a+\gamma,b+\delta;s')))$ of the framed rectangle Markov chain for local Frob\"ose bootstrap percolation.}
\label{tab:transitions}
\end{table}

We are now in position to define the events that correspond to the transitions of the Markov chain (see \cref{fig:example} for an illustration).
Recalling \cref{eq:def:O,eq:def:C:F}, we say that the transition from state (framed rectangle) $F = F(a,b;c,d;s)$ to $F'=F(a',b';c',d';s')$ as in \cref{tab:transitions,fig:augmented:cycle} occurs if
\begin{equation}
  \label{eq:def:transitions}
  \cT(F,F') = \left\{A : (A\setminus F_\blacksquare) \in \left( \cC^{\mathrm F}(F_\circ,F'_\circ)\cap \cO^c\left(F'_\square\right) \right)\right\}\end{equation}
occurs.
Intuitively speaking, we require that the set of infections that have not been explored yet (that is $A \setminus F_\blacksquare$) should be able to transform $F_\circ$ into $F'_\circ$, while also respecting the buffers of $F'$. If $F_\circ=F'_\circ$, the event $\cC^{\mathrm F}(F_\circ,F'_\circ)$ always occurs by definition (\cref{eq:def:C:F}).

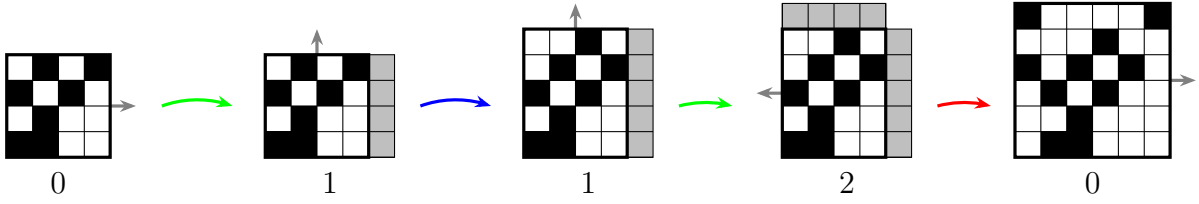
\begin{figure}
\centering
\begin{tikzpicture}[scale=3.4]
  \fill (.1, 0) rectangle ++(.1, .1);
  \fill (0, .2) rectangle ++(.1, .1);
  \fill (.1, .3) rectangle ++(.1, .1);
  \foreach \x in {0,.1,.2,.3} {
    \fill (\x, \x) rectangle ++(.1, .1);
    \draw (\x, 0) -- ++(0, 0.4);
    \draw (0, \x) -- ++(0.4, 0);
  }
  \draw[gray,-{Stealth[length=6]}, very thick] (.4, .2) -- ++(.1, 0);
  \draw[very thick] (0, 0) rectangle ++(0.4, 0.4);
  \node at (.2, -.1) {$0$};

  \draw[-{Stealth[length=6]}, very thick,green] (.6, .2) arc [radius=.4, start angle=110, delta angle=-40];

  \draw[fill=gray!50] (1.4, 0) rectangle ++(0.1, 0.4);
  \fill (1.1, 0) rectangle ++(.1, .1);
  \fill (1, .2) rectangle ++(.1, .1);
  \fill (1.1, .3) rectangle ++(.1, .1);
  \foreach \x in {0,.1,.2,.3} {
    \fill (1 + \x, \x) rectangle ++(.1, .1);
    \draw (1 + \x, 0) -- ++(0, 0.4);
    \draw (1, \x) -- ++(0.5, 0);
  }
  \draw[gray,-{Stealth[length=6]}, very thick] (1.2, .4) -- ++(0, .1);
  \draw[very thick] (1, 0) rectangle ++(0.4, 0.4);
  \node at (1.25, -.1) {$1$};

  \draw[-{Stealth[length=6]}, very thick,blue] (1.6, .2) arc [radius=.4, start angle=110, delta angle=-40];

  \draw[fill=gray!50] (2.4, 0) rectangle ++(0.1, 0.5);
  \fill (2.1, 0) rectangle ++(.1, .1);
  \fill (2, .2) rectangle ++(.1, .1);
  \fill (2.1, .3) rectangle ++(.1, .1);
  \foreach \x in {0,.1,.2,.3} {
    \fill (2 + \x, \x) rectangle ++(.1, .1);
  }
  \fill (2.2, .4) rectangle ++(.1, .1);
  \foreach \x in {0,.1,.2,.3,.4} {
    \draw (2 + \x, 0) -- ++(0, 0.5);
    \draw (2, \x) -- ++(0.5, 0);
  }
  \draw[gray,-{Stealth[length=6]}, very thick] (2.2, .5) -- ++(0, .1);
  \draw[very thick] (2, 0) rectangle ++(0.4, 0.5);
  \node at (2.25, -.1) {$1$};

  \draw[-{Stealth[length=6]}, very thick,green] (2.6, .2) arc [radius=.4, start angle=105, delta angle=-30];

  \draw[fill=gray!50] (3, .5) rectangle ++(0.4, 0.1);
  \draw[fill=gray!50] (3.4, 0) rectangle ++(0.1, 0.5);
  \fill (3.1, 0) rectangle ++(.1, .1);
  \fill (3, .2) rectangle ++(.1, .1);
  \fill (3.1, .3) rectangle ++(.1, .1);
  \foreach \x in {0,.1,.2,.3} {
    \fill (3 + \x, \x) rectangle ++(.1, .1);
  }
  \fill (3.2, .4) rectangle ++(.1, .1);
  \foreach \x in {0,.1,.2,.3,.4} {
    \draw (3 + \x, 0) -- ++(0, 0.6);
    \draw (3, \x) -- ++(0.5, 0);
  }
  \draw[gray,-{Stealth[length=6]}, very thick] (3, .25) -- ++(-.1, 0);
  \draw[very thick] (3, 0) rectangle ++(0.4, 0.5);
  \node at (3.25, -.1) {$2$};

  \draw[-{Stealth[length=6]}, very thick,red] (3.6, .2) arc [radius=.4, start angle=105, delta angle=-30];

  \fill (4.1, 0) rectangle ++(.1, .1);
  \fill (4, .2) rectangle ++(.1, .1);
  \fill (4.1, .3) rectangle ++(.1, .1);
  \fill (3.9, .3) rectangle ++(.1, .1);
  \fill (3.9, .5) rectangle ++(.1, .1);
  \fill (4.4, .5) rectangle ++(.1, .1);
  \foreach \x in {0,.1,.2,.3} {
    \fill (4 + \x, \x) rectangle ++(.1, .1);
  }
  \fill (4.2, .4) rectangle ++(.1, .1);
  \foreach \x in {0,.1,.2,.3,.4,.5} {
    \draw (3.9 + \x, 0) -- ++(0, 0.6);
    \draw (3.9, \x) -- ++(0.6, 0);
  }
  \draw[gray,-{Stealth[length=6]}, very thick] (4.5, .3) -- ++(.1, 0);
  \draw[very thick] (3.9, 0) rectangle ++(0.6, 0.6);
  \node at (4.2, -.1) {$0$};
\end{tikzpicture}
\caption{Illustration of some of the transitions in the Markovian exploration. The frame $F_\square$ is indicated in gray and the frame state is given below the framed rectangle. The sites that have already been revealed are black or white, depending on whether they are infected or not, respectively. The thick rectangle is the interior $F_\circ$ of the framed rectangle.
Finally, a gray arrow represents the direction in which the exploration will continue, while the transition arrows use the color code of \cref{fig:augmented:cycle}.
\label{fig:example}}
\end{figure}

Recalling \cref{eq:def:f,eq:def:q,obs:stacking,subsec:local:F}, one can check the following assertions.
\begin{observation}
\label{obs:transitions}
\Cref{tab:transitions} defines a stochastic matrix denoted by $\cT$. For any framed rectangle $F$, the events $\cT(F,F')$ (see \cref{eq:def:transitions}) corresponding to transitions present in $\cT$ are disjoint for different choices of $F'$. For each transition $\cT(F,F')$ we have $F_\circ\subset F'_\circ$ and $F_\blacksquare\subsetneq F'_\blacksquare$, unless $F$ has buffer state $4$. The transition events $\cT(F,F')$ define a Markov chain $(\cF(t))_{t\in\bbN}$ with transition matrix $\cT$, whose state $\cF(t)$ is measurable with respect to $A\cap \cF(t)_\blacksquare$. 

Given a rectangle $S$, we denote by $(\cF^S(t))_{t\in\bbN}$ the chain with initial state $S$ and frame state 0. If $\cI_{\mathrm{loc}}^{\mathrm F}(S)$ occurs, then for all $t\ge 0$, $\cI^{\mathrm F}_{\mathrm{loc}}(\cF^S_\circ(t))\cap \cO^c(\cF^S_\square(t))$ occurs. We denote $\cF^S(\infty)=\bigcup_{t\ge 0} \cF^S_\circ(t)$. Then for any $x\in\bbZ^2$, either $\cF^{\{x\}}(\infty)=[A]_{\mathrm F}^{x}$, in case $[A]_{\mathrm F}^{x}$ is finite, or both $[A]_{\mathrm F}^{x}$ and $\cF^{\{x\}}(\infty)$ are infinite (although not necessarily equal).
\end{observation}

\begin{corollary}
\label{cor:Markov}
Given nested rectangles $S\subset R=R(a,b)$, we have
\[\bbP_p\left(\cC^{\mathrm F}(S,R)\right)=\frac{\bbP_p(\cF^S(\infty)=R)}{e^{-2(a+b)q}},\]
\end{corollary}
\begin{proof}
Consider the event $\cE=\cO^c(R_\square)$, where we associate buffer state $4$ to $R$. Then
\begin{align*}
\bbP_p(\cE)&{}=e^{-2(a+b)q},&
\{\cF(\infty)=R\}&{}=\cE\cap\cC^{\mathrm F}(S,R)\end{align*}
and the latter two events are independent.
\end{proof}

\subsection{The entropy function}
\label{subsec:entropy}
We next introduce the function $h$ which governs the second term in \cref{eq:main}:
\begin{equation}
\label{eq:def:h}
h:(0,\infty)\to(0,\infty):z\mapsto \sqrt{\frac{2+\sqrt 2}{e^z-1}}.
\end{equation}
However, in order to make it appear more naturally, let us start with an heuristic investigation of \cref{tab:transitions}. In view of \cref{lem:optimal:path}, it is not surprising that one should focus on squares, so we consider transitions from a framed square $F(0,0;a,a;s)$. Intuitively, since we are interested in the second term, we want to isolate the part of $\cT$ giving rise to the first term in \cref{eq:main}. This corresponds to taking a factor $e^{-f(qa)}$ for increasing the width or height of the rectangle by 1. Further neglecting double and triple buffer deletions and those to frame states $4$ and $1''$, as well as the additive terms $q$ and $2q$ in \cref{tab:transitions}, we obtain that buffer creations ``cost'' $e^{-qa}$; loops ``cost'' $1$; single buffer deletions ``cost'' $pe^{f(qa)}$. This gives a positive $6\times 6$ matrix, whose Perron--Frobenius eigenvalue will turn out to be $1+h(qa)\sqrt p=1+\sqrt{2+\sqrt 2}\sqrt{e^{-qa}\cdot pe^{f(qa)}}$ (where $\sqrt{2+\sqrt{2}}$ is the Perron--Frobenius eigenvalue of the unweighted bi-directed 6-cycle with one directed edge removed as in \cref{fig:augmented:cycle}). Moreover, the remaining transitions neglected above only perturb this value by a smaller order term, as we rigorously establish in \Cref{subsec:upper:step,subsec:lower:step} below.

We need the following properties of $h$, whose elementary proofs are left to the reader. The function $h$ is analytic, decreasing and satisfies\footnote{Again, we will also use analogous asymptotics for $h'$.}
\begin{equation}\label{eq:h:asymptotics}
h(z)=\begin{cases}
\sqrt{\frac{2+\sqrt 2}{z}}\left(1-O(z)\right)&z\to0,\\
\frac{\sqrt{2+\sqrt 2}}{e^{z/2}}(1+O(e^{-z}))&z\to\infty.
\end{cases}
\end{equation}
In particular, \cref{eq:h:asymptotics} implies that $h$ is integrable, so $\lambda_2^{\mathrm F} := \int_0^\infty h\in(0,\infty)$. Furthermore, one can explicitly compute
\begin{equation}
\label{eq:lambda2}\lambda_2^{\mathrm F} :=\int_0^\infty h=\sqrt{2+\sqrt 2}\int_0^\infty\frac{\md z}{\sqrt{e^{z}-1}}=2\sqrt{2+\sqrt 2}\left[\arctan\left(\sqrt{e^z-1}\right)\right]_{0}^\infty=\pi\sqrt{2+\sqrt 2}.\end{equation}

\section{Upper bound}
\label{sec:upper}
In this section we prove the upper bound in \cref{th:main}.

\subsection{Proof of the upper bound in Theorem~\ref{th:main}}
\label{subsec:upper:reduction}

In view of \cref{cor:reduction}, the upper bound in \cref{th:main} follows immediately, if we prove
\begin{equation}
    \label{eq:upper:main}\bbP_p\left(\cI^{\mathrm F}_{\mathrm{loc}}\left(R\left(a^{(M)},a^{(M)}\right)\right)\right)\ge \exp\left(-\frac2q\int_{0}^\infty f+\frac{2}{\sqrt q}\int_0^\infty h-\frac{2C_2\log^{3/2}(1/q)}{\sqrt[3]q}\right),
\end{equation}
with $a^{(M)}=\Lambda$ (recall \cref{eq:def:Lambda}) and $f,h,q$ from \cref{eq:def:q,eq:def:f,eq:def:h}. To do this, we recursively prove a lower bound on the probability of locally internally filling larger and larger squares $R^{(i)}=R(a^{(i)},a^{(i)})$. We start with $a^{(0)}=1$, so that
\begin{equation}
\label{eq:R0:upper}\bbP_p\left(\cI^{\mathrm F}_{\mathrm{loc}}\left(R^{(0)}\right)\right) = p \ge p\exp\left(-\frac2q\int_0^{a^{(0)}q}f\right).
\end{equation}

The sequence $a^{(i)}$ is defined recursively so that $a^{(i+1)}-a^{(i)}=\lceil(a^{(i)})^{2/3}\rceil$, until reaching the final size $a^{(M)}=\Lambda$ (for the sake of simplicity, we assume that this is possible for an integer $M$). The exponent $2/3$ is not of fundamental importance, but it is chosen in order to optimise the third term in \cref{eq:upper:main}. In order to bound $M$, observe that $a^{(i+1)}-a^{(i)}\ge (a^{(m)})^{2/3}$ for any $i \geq m$. Therefore, $a^{(m+\lceil(a^{(m)})^{1/3}\rceil)}\ge 2a^{(m)}$, so
\begin{equation}
\label{eq:M:bound}M\le \sum_{m=0}^{\log_2 a^{(M)}}1+\left(\frac{\Lambda}{2^{m}}\right)^{1/3}\le \frac{\log^{1/2}(1/q)}{\sqrt[3]q}.\end{equation}

Next comes the main technical result of this section, whose proof is left to \cref{subsec:upper:step}. Intuitively speaking, the proposition below provides a lower bound on the transition probabilities during a step from $a^{(i)}$ to $a^{(i+1)}$. These piecewise estimates will be composed to provide \cref{eq:upper:main}, through a Riemann sum.
\begin{proposition}[Coarse step lower bound]
\label{prop:upper:main}
Let $a\ge 1$ and $b=a+\lceil a^{2/3}\rceil\le \Lambda$. Then
\[\frac{\bbP_p(\cI^{\mathrm F}_{\mathrm{loc}}(R(b,b)))}{\bbP_p(\cI^{\mathrm F}_{\mathrm{loc}}(R(a,a)))}\ge \exp\left(-\frac2q\int_{aq}^{bq}f+2h(aq)(b-a)\sqrt q-C_1^2\log(1/q)\right).\]
\end{proposition}

Applying \cref{prop:upper:main} to $a=a^{(i)}$, telescoping and then recalling \cref{eq:R0:upper,eq:M:bound}, we obtain
\begin{align}
\label{eq:upper:1}&{}\bbP_p\left(\cI^{\mathrm F}_{\mathrm{loc}}\left(R^{(M)}\right)\right)\\
\nonumber&{}\ge \bbP_p\left(\cI^{\mathrm F}_{\mathrm{loc}}\left(R^{(0)}\right)\right)\exp\left(-\frac{2}{q}\int_{a^{(0)}q}^{a^{(M)}q}f + 2\sqrt q\sum_{i=0}^{M-1}h\left(a^{(i)}q\right)\left(a^{(i+1)}-a^{(i)}\right)-C_1^2M\log(1/q)\right)\\
\nonumber&{}\ge \exp\left(-\frac2q\int_{0}^{a^{(M)}q}f+2\sqrt q\sum_{i=0}^{M-1}h\left(a^{(i)}q\right)\left(a^{(i+1)}-a^{(i)}\right)-\frac{C_2\log^{3/2}(1/q)}{\sqrt[3]q}\right).\end{align}
Since $f$ is positive, the integral above is less than $\lambda_1^{\mathrm F}=\int_0^\infty f$. It therefore remains to evaluate the discrepancy between the Riemann sum in \cref{eq:upper:1} and $\lambda_2^{\mathrm F}=\int_0^\infty h$.
Yet,
\begin{multline}
\label{eq:upper:2}\left|q\sum_{i=0}^{M-1}h\left(a^{(i)}q\right)\left(a^{(i+1)}-a^{(i)}\right)-\int_{a^{(0)}q}^{a^{(M)}q}h\right|\\
\begin{aligned}
&{}\le \sum_{i=0}^{M-1}\max_{a\in[a^{(i)},a^{(i+1)}]}\left|h'(qa)\right|\left(qa^{(i+1)}-qa^{(i)}\right)^2\\
&{}\overset{\mathclap{h'(z) = O(z^{-3/2})}}\le \quad \quad O\left(q^{1/3}\right)\sum_{i=0}^{M-1}\left(qa^{(i+1)}\right)^{-3/2}\left(qa^{(i)}\right)^{2/3}\left(qa^{(i+1)}-qa^{(i)}\right)\\
&{}\le O\left(q^{1/3}\right)\int_0^{qa^{(M)}}\frac{\md x}{x^{5/6}}\le O(q^{1/3})\log^{1/6}(1/q).\end{aligned}\end{multline} Moreover, \cref{eq:h:asymptotics} gives
\begin{equation}
\label{eq:upper:3}
0\le \lambda_2^{\mathrm F}-\int_{a^{(0)}q}^{a^{(M)}q}h\le O\left(\sqrt{a^{(0)}q}+e^{-a^{(M)}q/2}\right)= O(\sqrt q).\end{equation}
Putting \cref{eq:upper:1,eq:upper:2,eq:upper:3} together, we obtain \cref{eq:upper:main}, as desired.

\subsection{Proof of Proposition~\ref{prop:upper:main}}
\label{subsec:upper:step}
We next turn to proving \cref{prop:upper:main}. Fix $a\ge a^{(0)}$ and $b=a+\lceil a^{2/3}\rceil\le \Lambda$. Recalling \cref{obs:stacking}, we have
\[\frac{\bbP_p(\cI^{\mathrm F}_{\mathrm{loc}}(R(b,b)))}{\bbP_p(\cI^{\mathrm F}_{\mathrm{loc}}(R(a,a)))}\ge \max_{x,y}\bbP_p\left(\cC^{\mathrm F}(R(x,y;x+a,y+a),R(b,b)\right),\]
where the max ranges over all possible positions $(x,y)\in R(b-a+1,b-a+1)$ of a translate of $R(a,a)$ inside $R(b,b)$. Further recalling \cref{cor:Markov}, we get
\begin{align}\nonumber\frac{\bbP_p(\cI^{\mathrm F}_{\mathrm{loc}}(R(b,b)))}{\bbP_p(\cI^{\mathrm F}_{\mathrm{loc}}(R(a,a)))}&{}=\max_{x,y}\frac{\bbP_p(\cF^{R(x,y;x+a,y+a)}(\infty)=R(b,b))}{e^{-4bq}}\\
\label{eq:IFloc:ratio}&{}\ge q^3\sum_{x,y}\bbP_p\left(\cF^{R(a,a)}(\infty)=R(-x,-y;b-x,b-y)\right),\end{align}
where the term $q^3$ accommodates the number of choices for $x, y$ in the sum (recalling $a, b \leq \Lambda$).

For the sake of an upper bound, we may simply discard the transitions of the framed rectangle Markov chain that would yield a negligible contribution. To that end, we call a trajectory $(\cR(t))_{t=0}^N$ of this chain \emph{good}, if the following all hold:
\begin{itemize}
\item it only features \emph{good transitions}: buffer creations, loops and single buffer deletions;
\item $\cF(0)=F(0,a;0,a;0)$;
\item $\cF(N-1)\neq \cF(N)=F(-x,-y;b-x,b-y;4)$ for some $(x,y)\in R(b-a+1,b-a+1)$.
\end{itemize}

In order to lower bound \cref{eq:IFloc:ratio} we evaluate the probability of each good trajectory individually. We start by examining a single transition.
\begin{lemma}[Probability of a single transition]
\label{lem:single:upper}
Let $(F,F')$ be a transition in a good trajectory with $F=F(i,j;k,l;s)$ and $F'=F(i-\alpha,j-\beta;k+\gamma,l+\delta;s')$. Then
\[
\bbP_p(\cT(F,F')) =
  \begin{cases}
  1 \times e^{-qa} \times e^{O(qa^{2/3})} & \text{buffer creation,}\\
  e^{-(\alpha+\gamma)f(q(l-j))-(\beta+\delta)f(q(k-i))}\times 1\times e^{O(q)}&\text{loop,}\\
  e^{-f(q(l-j))-f(q(k-i+1))}\times pe^{f(qa)}\times e^{O(a^{-1/3})}&\text{$1$-buffer deletion.}
  \end{cases}
\]
Note that $\alpha, \beta, \gamma, \delta \in \{0, 1\}$ and for good trajectories their sum is either $0$, $1$ or $2$, corresponding to the three cases above respectively.
\end{lemma}

\begin{remark}
\label{rem:single:upper}
Before digging into the proof, a few observations are in order.
\begin{itemize}
  \item Each term above is spelled as a product $A \times B \times C$, where the first factor corresponds to growth costs, while $B$ stands for buffer manipulations and $C$ for error terms.
  \item Note also that the term $A$ in the third line (corresponding to \emph{buffer deletions}) has been artificially factored in a way that makes the diagonal growth of the box dimensions mimic the cost of one horizontal and one vertical steps (see Figure~\cref{fig:good_trajectory}).
  The cost of this modification is absorbed in the corresponding error term $C$.
\end{itemize}
\end{remark}

\begin{proof}[Proof of \cref{lem:single:upper}]
For concreteness, we assume that $s=2$, the other cases being treated identically. If $s'=3$ (buffer creation), then \cref{tab:transitions} gives
\[\bbP_p(\cT(F,F'))=e^{-q(l-j)}=e^{-qa}e^{qO(b-a)}.\] 
If $s'=2$ (loop), so that $\alpha=1$ and $\beta=\gamma=\delta=0$, then \cref{tab:transitions} gives directly the result stated. If $s'=1$ (buffer deletion), then \cref{tab:transitions} gives
\begin{align}
\nonumber
\bbP_p(\cT(F,F'))&{}=pe^{-f(q(l-j))-q}\\
\nonumber
&{}=e^{-f(q(l-j))-f(q(k-i+1))}pe^{f(qa)}\exp\left(O(q)+qO(b-a)\max_{z\in[aq,bq]}\left|f'(z)\right|\right)\\
\label{eq:transition:21}
&{}=e^{-f(q(l-j))-f(q(k-i+1))}pe^{f(qa)}\exp(O(q)+O(b-a)/a),
\end{align}
where above we have expanded $f$ linearly around $qa$ and used that $|f'(z)| \leq O(z^{-1})$.
\end{proof}

In the language of \cref{rem:single:upper}, note that in \cref{lem:single:upper}, the terms $A$, $B$, $C$ will contribute to the first, second and third order terms in \cref{eq:upper:main}. Moreover, the $A$ terms are made so that along a trajectory they give the differential form $W_p^{\mathrm F}$ (recall \cref{subsec:variations}) evaluated along a certain path from $(a,a)$ to $(b,b)$. In order to control this arbitrary path, we need the following analytic bound.
\begin{lemma}[Diagonal deviation cost]
\label{lem:W:to:g}Let $1\le a\le b$ and $\gamma$ be a piecewise linear coordinatewise non-decreasing path from $(a,a)$ to $(b,b)$. Then
\[\left|W^{\mathrm F}_p(\gamma)-\frac{2}{q}\int_{aq}^{bq}f\right|\le O\left((b-a)^3/a^2\right)
.\]
\end{lemma}
\begin{proof}
Let $\gamma_p=q\cdot \gamma$. Then
\begin{align*}
qW^{\mathrm F}_p(\gamma)&{}=W^{\mathrm F}(\gamma_p)=\int_{\gamma_p} (f(y)\md x+f(x)\md y)\\
&{}=\int_{\gamma_p} (f(x)\md x+f(y)\md y) + \int_{\gamma_p} (f(x)-f(y))(\md y-\md x)\\
&{}=2\int_{aq}^{bq} f+ f'(aq)\int_{\gamma_p} (x-y)(\md y-\md x)+O(qb-qa)^3\sup_{x \in [aq,bq]}f''(x)\\
&{}= 2\int_{aq}^{bq}f + O(q(b-a))^3/(aq)^{2},\end{align*}
where we used \cref{eq:f:asymptotics} and the fact that the differential form $(x-y)(\md y-\md x)=-\frac{1}{2}\md (x-y)^2$ is exact, so closed, to get
\[\int_{\gamma_p}(x-y)(\md y-\md x)=\int_{((aq,aq),(bq,bq))}(x-y)(\md y-\md x)=0.\qedhere\]
\end{proof}

\begin{figure}
  \centering
  \begin{tikzpicture}[scale=.6,
  square/.style={regular polygon,regular polygon sides=4},
    triangle/.style={regular polygon,regular polygon sides=3}]
    \draw[->] (0, 0) -- (10, 0);
    \draw[->] (0, 0) -- (0, 8);
    \draw[thick, color=blue] (2, 2) -- (3, 2);
    \draw[thick, color=blue] (3, 2) -- (4, 2);
    \draw[thick,arrows=-{>[scale=1.3, length=3, width=3.5,bend]},green] (4.5, 2) arc [radius=.5, start angle=90, delta angle=-270];
    \draw[thick, color=blue] (4, 2) -- (4, 3);
    \draw[thick,arrows=-{>[scale=1.3, length=3, width=3.5,bend]},green] (3.5, 3) arc [radius=.5, start angle=270, delta angle=-270];
    \draw[thick, color=blue] (4, 3) -- (5, 3);
    \draw[thick, color=orange] (5, 3) -- (6, 4);
    \draw[dashed,color=gray] (5, 3) -- (6, 3) -- (6, 4);
    \draw[thick, color=blue] (6, 4) -- (6, 5);
    \draw[thick, color=blue] (6, 5) -- (6, 6);
    \draw[thick, color=blue] (6, 6) -- (6, 7);
    \draw[thick, color=orange] (6, 7) -- (7, 8);
    \draw[dashed, color=gray] (6, 7) -- (7, 7) -- (7, 8);
    \draw[thick, color=blue] (7, 8) -- (8, 8);
    \node[anchor=north east] at (2.2, 2) {\small$(a, a)$};
    \node[anchor=west] at (8, 8) {\small$(b, b)$};
    \node at (2,2) [circle,fill,inner sep=1pt] (.1) {};
    \node at (3,2) [circle,fill,inner sep=1pt] (.1) {};
    \node at (4,2) [circle,fill,inner sep=1pt] (.1) {};
    \node at (7,8) [circle,fill,inner sep=1pt] (.1) {};
    \node at (8,8) [circle,fill,inner sep=1pt] (.1) {};
    \node at (4,2) [square,draw,inner sep=1.5pt] {};
    \node at (4,3) [square,draw,inner sep=1.5pt] {};
    \node at (6,4) [square,draw,inner sep=1.5pt] {};
    \node at (6,5) [square,draw,inner sep=1.5pt] {};
    \node at (6,6) [square,draw,inner sep=1.5pt] {};
    \node at (6,7) [square,draw,inner sep=1.5pt] {};
    \node at (4,3) [triangle,fill,inner sep=.8pt] {};
    \node at (5,3) [triangle,fill,inner sep=.8pt] {};
    \node at (12,5) [circle,fill,inner sep=1pt] (.1) {};
    \node at (12.5,5.1) [anchor=west] {state $0$};
    \node at (12,3.5) [square,draw,inner sep=1.5pt] {};
    \node at (12.5,3.6) [anchor=west] {state $1$};
    \node at (12,2) [triangle,fill,inner sep=.8pt] {};
    \node at (12.5,2.1) [anchor=west] {state $2$};
  \end{tikzpicture}
  \caption{A good trajectory, with frame states marked with the symbols in the legend.
  Blue edges maintain a state, green arrows correspond to buffer creation and orange edges stand for (single) buffer deletions, as in \Cref{fig:augmented:cycle}
  (in particular, loops in this picture do not correspond to loops in \Cref{fig:augmented:cycle}). The dashed lines represent the transformation indicated in the buffer deletion case of \cref{lem:single:upper}.}
  \label{fig:good_trajectory}
\end{figure}
We next put together the contributions of all the transitions in a good trajectory $(\cF(t))_{t=0}^N$. Denote by $K$ the number of (single) buffer deletions performed by $(\cF(t))_{t=0}^N$.
As noted below \cref{lem:single:upper}, each of these buffer deletions correspond to a diagonal jump in the size of the box and it can be transformed into one horizontal and one vertical steps, giving rise to a new trajectory $\gamma$, depicted with dashed lines in \cref{fig:good_trajectory}.

Next we will follow the trajectory $\gamma$, multiplying the cost of each of its steps, according to \cref{lem:single:upper}.
Note that these costs are written as a the product of three factors and we group those in the calculation  below. 
Let $\gamma$ be the path formed by the first terms in \cref{lem:single:upper}. Further note that the number of buffer creations is $K+4$. Then,
\begin{align}
\nonumber\prod_{t=1}^N\bbP_p(\cT(\cF(t-1),\cF(t)))&{}=e^{-W^{\mathrm F}_p(\gamma)}\times e^{-4qa}\left(e^{-qa}pe^{f(qa)}\right)^K\times e^{O(q(b-a))+O(K+1)(qa^{2/3}+a^{-1/3})},\\
&{}=\exp\left(-\frac2q\int_{aq}^{bq}f\right)\times \left(\frac{p}{e^{aq}-1}\right)^K\times e^{O(1+qa)(1+Ka^{-1/3})}.\label{eq:trajectory:upper}
\end{align}
In view of \cref{eq:trajectory:upper}, it remains to add up all good trajectories. Moreover, since we are only proving an upper bound, we are free to choose the value of $K$, so fix
\begin{equation}
\label{eq:def:K}
K=\left\lfloor h(aq)(b-a)\sqrt p\right\rfloor=O\left(a^{1/6}\right)\end{equation}
(recall \cref{eq:def:h}). Since the right-hand side of \cref{eq:trajectory:upper} only depends on the trajectory through $K$, it remains to count the number of good trajectories with $K$ buffer deletions. Note that a good trajectory that goes from state zero to four, while performing $K$ buffer deletions, undergoes exactly $2K + 3$ non-loop transitions. Moreover, given the sequence of non-loop transitions performed, one still has to distribute the loop ones respecting the total change in height and width. Thus, the total number of such trajectories is
\begin{equation}\label{eq:number:trajectories}\binom{b-a+1}{K+1}^2\cM^{2K+3}(0,3)\ge \binom{b-a}{K}^2\cM^{2K+3}(0,3),\end{equation}
where $\cM^{2K+3}(0,3)$ is the $(0,3)$-entry of the $(2K+3)$-rd power of the matrix
\begin{equation}
\label{eq:def:M:F:upper}
\cM=\begin{pmatrix}
0 & 1 & 0 & 0 & 0 & 0\\
1 & 0 & 1 & 0 & 0 & 0\\
0 & 1 & 0 & 1 & 0 & 0\\
0 & 0 & 1 & 0 & 1 & 0\\
0 & 0 & 0 & 1 & 0 & 1\\
1 & 0 & 0 & 0 & 1 & 0
\end{pmatrix},
\end{equation}
corresponding to good non-loop transitions between frame states $0, 1, 2, 3, 2'$ and $1'$.
But $\cM$ is bipartite, so
\begin{multline}
\cM^{2K+3}(0,3)\\
\begin{aligned}[b]
&{}=\begin{pmatrix}
1 & 0 & 0
\end{pmatrix}
\begin{pmatrix}
1 & 1 & 0\\
1 & 2 & 1\\
1 & 1 & 2
\end{pmatrix}^{K+1}
\begin{pmatrix}
0\\
1\\
1
\end{pmatrix}\\
&{}=
\begin{pmatrix}
1 & 0 & 0
\end{pmatrix}
\begin{pmatrix}
-1 & \frac{2+\sqrt2}{4} & \frac{2-\sqrt2}{4}\\
0 & -\frac{1}{2\sqrt 2} & \frac{1}{2\sqrt 2}\\
1 & -\frac{1}{2\sqrt 2} & \frac{1}{2\sqrt 2}\\
\end{pmatrix}
\begin{pmatrix}
1 & 0 & 0\\
0 & 2-\sqrt 2 & 0\\
0 & 0 & 2+\sqrt 2
\end{pmatrix}^{K+1}
\begin{pmatrix}
0 & -1 & 1\\
1 & -\sqrt 2 & 1\\
1 & \sqrt 2 & 1
\end{pmatrix}
\begin{pmatrix}
0\\
1\\
1
\end{pmatrix}\\
&{}=\frac{(1-\sqrt2)(2-\sqrt 2)^{K}+(1+\sqrt2)(2+\sqrt 2)^{K}}{2}\ge \left(2+\sqrt 2\right)^{K}.\end{aligned}
\label{eq:Matrix:bound}\end{multline}

Recall that for any $\sqrt n\ge m$ it holds\footnote{\url{https://rjlipton.wpcomstaging.com/2014/01/15/bounds-on-binomial-coefficents/}} that $\binom nm \ge n^m/(4m!)\ge (en/m)^{m}/(4em)$. Plugging \cref{eq:number:trajectories,eq:Matrix:bound,eq:trajectory:upper} into \cref{eq:IFloc:ratio} and using this fact, we obtain
\begin{align}
\nonumber\frac{\bbP_p(\cI^{\mathrm F}_{\mathrm{loc}}(R(b,b)))}{\bbP_p(\cI^{\mathrm F}_{\mathrm{loc}}(R(a,a)))}&{}\ge \frac{q^3 e^{-2/q\int_{aq}^{bq}f}}{(4eK)^{2}}\left(\sqrt{\frac{p(2+\sqrt2)}{e^{aq}-1}}\cdot\frac{e(b-a)}{K}\right)^{2K}e^{O(1+qa)(1+Ka^{-1/3})}\\
\nonumber&{}\overset{\mathclap{\cref{eq:def:K}, \cref{eq:def:h}}}\ge \quad \frac{q^3}{(4eK)^2}\exp\left(-\frac2q\int_{aq}^{bq}f+2K+O(1 + qa)\right)
\\
\nonumber&{}\overset{\mathclap{\cref{eq:def:K}}}\ge \quad q^4\exp\left(-\frac2q\int_{aq}^{bq}f+2K+O(qa)\right)\\
\label{eq:upper:conclusion}&{}\ge \exp\left(-\frac2q\int_{aq}^{bq}f+2h(aq)(b-a)\sqrt q -O(C_1)\log(1/q)\right)\end{align}
where in the last inequality we used \cref{eq:def:q}, together with the fact that $b\le \Lambda$, completing the proof of \cref{prop:upper:main}.

\section{Lower bound}
\label{sec:lower}

\subsection{Deducing the lower bound of Theorem~\ref{th:main} from Proposition~\ref{prop:main:lower}}
\label{subsec:lower:reduction}
By \cref{cor:reduction}, in order to prove the lower bound in \cref{th:main}, it suffices to prove
\begin{equation}
\label{eq:main:lower}\bbP_p\left(\cI^{\mathrm F}_{\mathrm{loc}}(R)\right)\le \exp\left(-\frac{2\lambda_1^{\mathrm F}}{q}+\frac{2\lambda_2^{\mathrm F}}{\sqrt q}+\frac{\log^{2}(1/p)}{\sqrt[3]q}\right)\end{equation}
with $R=R(\Lambda,\Lambda)$. In order to do this, we will use the following lower bound analogue of \cref{prop:upper:main}, whose proof is left to \cref{subsec:lower:step}.
\begin{proposition}[Coarse step upper bound]
\label{prop:main:lower}
Let $S\subset T$ be rectangles with $\lng(T)\le \min \{ 2C_2\sh(T), \Lambda \}$ and the semi-perimeters such that $\phi(T)=\phi(S)+\lceil (\phi(S))^{2/3}\rceil$. Then
\[\bbP_p\left(\cC^{\mathrm F}(S,T)\right)\le \exp\left(-W^{\mathrm F}_p(\gamma_{S,T})+h(q\sh(T))(\phi(T)-\phi(S))\sqrt q+C_3\log(1/q)\right).\]
\end{proposition}
In order to apply \cref{prop:main:lower}, we need to extract a sequence of rectangles analogous to $R^{(i)}$ in \cref{subsec:upper:reduction}. By iterating \cref{lem:AL:loc}, we see that $\cI^{\mathrm F}_{\mathrm{loc}}(R)$ implies the existence of a sequence of nested rectangles $(R_i)_{i=2}^{2\Lambda}$ such that $R_{2\Lambda}=R$, $\phi(R_{i})=i$ for all $i\in\{2,\dots,2\Lambda\}$ and the event
\begin{equation}
\label{eq:Ri:decomposition}
\{R_0\subset A\}\cap\bigcap_{i=1}^{2\Lambda}\cC^{\mathrm F}(R_{i-1},R_{i})\end{equation}
occurs.

In view of the hypotheses of \cref{prop:main:lower}, we need to ensure that typically there are no rectangles with too large aspect ratio. Such uneven rectangles can appear before, during or after the critical scale.

This motivates us to introduce $N_0$ and $N_1$, which intuitively correspond to the last violation of the aspect ratio before the critical scale and to the first violation after the critical scale respectively. More precisely,
\begin{align}
  \label{eq:def:N0}
  N_0&{} = \max \left\{ i \in \{2, \dots, 2\Lambda\}: \lng(R_i) \in \left[ C_2\sh(R_i), 1/(C_1q) \right] \right\},\\
  \label{eq:def:N1}
  N_1&{} = \min \left\{ i \in \{2, \dots, 2\Lambda\}: \sh(R_i) \in \left[ C_1/q,  \lng(R_i)/C_2 \right] \right\},
\end{align}
with the conventions that for $a > b$, $[a, b] = \varnothing$, $\max\varnothing = 2$ and $\min\varnothing = 2\Lambda$.

Let us now give an overview of the proof of the lower bound in \cref{th:main} that follows.
We will first consider what happens between $N_0$ and $N_1$, which clearly encapsulates the critical region $[1/(C_1q), C_1/q]$.
This critical region comprises the bulk of the contributions in our estimates and the existence of uneven rectangles within it will be ruled out in \cref{eq:strange}.
In \cref{eq:truncation:lower}, we show that what happens after $N_1$ can be ignored without affecting our bounds. Finally, we prove that any potential gain one could have obtained by reaching a rectangle with large aspect ratio before $1/(C_1 q)$ (this event corresponds to $N_0 > 2$) would be offset by the cost of creating such an anisotropic rectangle to start with, see \cref{eq:truncation:lower}.

We start by coarse-graining the sequence $(R_i)_{i=N_0}^{N_1}$. To do so, we set $\phi_0=N_0$, recursively define $\phi_{k+1}=\phi_k+\lceil\phi_k^{2/3}\rceil$ and choose $M$ so that $N_1\in(\phi_{M-1},\phi_M]$. For the sake of simplicity we assume that $N_1=\phi_M$ and recall that \cref{eq:M:bound} applies.

We call a sequence of rectangles $(R_{\phi_j})_{j=0}^M$, which can be obtained as above a \emph{coarse sequence} and we say that a coarse sequence $(R_{\phi_j})_{j=0}^M$ is \emph{good}, if for every $j\in\{0,\dots,M\}$ we have $\lng(R_{\phi_j})\le 2C_2\sh(R_{\phi_j})$.
Of course we will need to control non-good sequences as well and for this we introduce the following definition.
We say that a rectangle $R'\subset R$ is \emph{strange} if $\lng(R')\ge C_2\sh(R')$, $\lng(R')>1/(C_1q)$ and $\sh(R')<C_1/q$, or in words: $R'$ is critical and has a large aspect ratio.

We now use \cref{eq:Ri:decomposition,obs:stacking} to write our main decomposition
\begin{equation}
  \label{eq:bound_Iloc}
  \bbP_p\left(\cI^{\mathrm F}_{\mathrm{loc}}(R)\right) \le \sum_{R'\text{ strange}}\bbP_p\left(\cI^{\mathrm F}_{\mathrm{loc}}(R')\right)+\sum_{M}\sum_{(R_{\phi_j})_{j=0}^M}\bbP_p\left(\cI^{\mathrm F}_{\mathrm{loc}}(R_{\phi_0})\cap\bigcap_{j=1}^M\cC^{\mathrm F}\left(R_{\phi_{j-1}},R_{\phi_j}\right)\right),
\end{equation}
where the above sum ranges over all possible good coarse sequences.

We start by controlling the first term above.
By \crefdefpart{obs:traversability}{obs:travers:infect} and \cref{eq:gaps}, 
\begin{equation*}
  \sum_{R'\mathrm{ strange}}\bbP_p\left(\cI^{\mathrm F}_{\mathrm{loc}}(R')\right)
  \le \sum_{R'\mathrm{ strange}}\exp \left\{ -\lng(R')f \left( q\sh(R') \right) \right\}.
\end{equation*}
However, there are at most $|R|^2\le p^{-5}$ possible $R'$ and
\[\lng(R')f(q\sh(R'))\ge \begin{cases}
\frac{C_2f(C_1)}{C_1q}&1/(C_1q)\le \sh(R')\le C_1/q\\
\frac{f(C_1/C_2)}{C_1q}&1/(C_1q)\le\lng(R')\le C_1/q\\
\frac{C_1f(1/C_1))}{q}&\sh(R')<1/(C_1q),\lng(R')\ge C_1/q,
\end{cases}\]
so that 
\begin{equation}
    \label{eq:strange}
    \sum_{R'\mathrm{ strange}}\bbP_p\left(\cI^{\mathrm F}_{\mathrm{loc}}(R')\right)\le \exp\left(-C_1/q\right),
\end{equation}
which is much smaller than our objective in \cref{eq:main:lower}.

In order to control the second term in \cref{eq:bound_Iloc}, we start by observing that the number of coarse sequences is at most the number of ways to choose the number $M$ and then the $M$ rectangles contained in $R$. Thus, by \cref{eq:M:bound,eq:def:Lambda}, for small enough $q$ there are at most
\begin{equation}
  \label{eq:coarse:seq:number}
  q^{-1/2}\left(q^{-5}\right)^{\log^{1/2}(1/q)/\sqrt[3]q}\le \exp\left(\frac{\log^{5/3}(1/q)}{\sqrt[3]q}\right)
\end{equation}
coarse sequences.

By definition, for any good coarse sequence $(R_{\phi_j})_{j=0}^M$, we can apply \cref{prop:main:lower,eq:M:bound} to get 
\begin{multline}
\label{eq:main:lower:applied}\bbP_p\left(\bigcap_{j=1}^M\cC^{\mathrm F}\left(R_{\phi_{j-1}},R_{\phi_j}\right)\right)\\
\le
\exp\left(\sum_{j=1}^{M}\left(-W^{\mathrm F}_p\left(\gamma_{R_{\phi_{j-1}},R_{\phi_{j}}}\right)+h(a_jq)(\phi_{j}-\phi_{j-1})\sqrt q\right)+C_3\log^{3/2}(1/q)/\sqrt[3]q\right),
\end{multline}
where $a_j=\sh(R_{\phi_j})$. Since both the first and second terms above depend on the choice of the good coarse sequence, we need a refinement of \cref{lem:optimal:path} optimising both terms simultaneously. However, the second order term being of smaller order, the optimal path will only be a perturbation of $\gamma_{R_{\phi_0},R_{\phi_M}}$. This is established in the following lemma.

\begin{lemma}
\label{lem:Hp}Fix rectangles $S\subset T$ such that $\lng(T) \le \Lambda$. Let
\[H_p(\gamma)=\int_\gamma f(yq)\md x+f(xq)\md y-\sqrt qh(q\min(x,y))(\md x+\md y).\]
Then for any path $\gamma$ from $S$ to $T$ contained in $\{(x,y)\in(0,\infty)^2:y/(3C_2)\le x\le 3C_2y\}$ we have
\[H_p(\gamma)\ge H_p(\gamma_{S,T})-C_2\log(1/q),\]
recalling that $\gamma_{S, T}$ from \cref{eq:gammaR,fig:optimal_path}.
\end{lemma}

\begin{proof}
By symmetry we may assume that $S=R(a,b)$ and $T=R(c,d)$ with $a \leq b$, $c \leq d$, \begin{equation}
\label{eq:def:Delta}
\gamma\subset \Delta=\left\{(x,y)\in[1,\infty)^2:x\le y\le \min(3C_2x,\Lambda)\right\}.\end{equation}
By Green's theorem, defining $\Gamma$ as the region enclosed by $q\gamma$ and $q\gamma_{S,T}$,
\begin{align}
qH_p(\gamma_{S,T})-qH_p(\gamma)={}&\iint_{\Gamma}(f'(x)-\sqrt q h'(x)-f'(y))\md x\md y\nonumber\\
\le{}&\iint_{\bar\Gamma}(f'(x)-\sqrt q h'(x)-f'(y))\md x\md y.\label{eq:Green}\end{align}
with $\bar \Gamma=\{(x,y)\in q\Delta:f'(y)\le f'(x)-\sqrt q h'(x)\}$. Recalling that $f$ is convex and writing $y_x=\max\{y:(x,y)\in\bar\Gamma\}$, we get that \cref{eq:Green} is at most
\begin{equation}-\sqrt q\int_{q}^{\Lambda q}\md x\int_{x}^{y_x}h'(x)\md y \label{eq:HgST}=-\sqrt q\int_{q}^{\Lambda q}(y_x-x)h'(x)\md x,\end{equation}
which will be bounded by splitting the integral in three separate intervals.

Observe that by \cref{eq:f:asymptotics,eq:h:asymptotics}, for any $x\in[C_1q,\log(1/q)-C_1]$, we have $-\sqrt q h' (x)\le -f'(x)/2$, so that $f'(y_x)\in[f'(x),f'(x)/2]$. But then \cref{eq:f:asymptotics} implies that $f''(y)=\Omega(f''(x))$ for any $x\in[C_1q,\log(1/q)-C_1]$ and $y\in[x,y_x]$, so that
\[y_x\le x + O(\sqrt q) \frac{|h'(x)|}{f''(x)}\le x+\frac{O(\sqrt q)}{\min(1,x)|h'(x)|},\]
where in the last inequality we used \cref{eq:f:asymptotics,eq:h:asymptotics}. Therefore,
\begin{equation}
\label{eq:Green:part1}
-\sqrt q\int_{C_1q}^{\log(1/q)-C_1}(y_x-x)h'(x)\md x\le O(q)\int_{C_1q}^{\log(1/q)-C_1}\frac{\md x}{\min(1,x)}\le O(q\log(1/q)).
\end{equation}

Turning to $x\in[q,C_1q]$, we can use \cref{eq:h:asymptotics,eq:def:Delta} to get
\begin{equation}
\label{eq:Green:part2}-\sqrt q\int_{q}^{C_1q}(y_x-x)h'(x)\md x\le -3C_2C_1q^{3/2}\int_{q}^{C_1q}h'(x)\md x\le C_2^2q.\end{equation}
Finally, for $x\in[\log(1/q)-C_1,\Lambda]$, by \cref{eq:def:Delta,eq:h:asymptotics,eq:def:Lambda}, we have
\begin{equation}
\label{eq:Green:part3}-\sqrt q\int_{\log(1/q)-C_1}^{\Lambda q}(y_x-x)h'(x)\md x\le -q^{3/2}\Lambda \int_{\log(1/q)-C_1}^{\infty} h'(x)\md x\le C_2q\log(1/q)/2.
\end{equation}
Plugging \cref{eq:Green:part1,eq:Green:part2,eq:Green:part3} into \cref{eq:HgST}, we obtain the desired
\[qH_p(\gamma_{S,T})-qH_p(\gamma)\le C_2 q\log(1/q).\qedhere\]
\end{proof}

Returning to \cref{eq:main:lower:applied}, the fact that $h$ is decreasing and \cref{lem:Hp} give
\begin{align}
\nonumber\sum_{j=1}^MW_p^{\mathrm F}\left(\gamma_{R_{\phi_{j-1}},R_{\phi_j}}\right)-h(a_jq)\left(\phi_j-\phi_{j-1}\right)&{}\ge \sum_{j=1}^MH_p\left(\gamma_{R_{\phi_{j-1}},R_{\phi_j}}\right)\\
\label{eq:Hp}&{}\ge H_p\left(\gamma_{R_{\phi_{0}},R_{\phi_M}}\right)-C_2\log\frac1q.\end{align}
We would like to relate the last quantity to \begin{align}
\label{eq:lambda}\frac{2\lambda_1^{\mathrm F}}{q}&{}=\lim_{\substack{x\to0\\y\to\infty}}W^{\mathrm F}_p((x,x),(y,y)),&
\frac{2\lambda_2^{\mathrm F}}{\sqrt q}&{}=\lim_{\substack{x\to0\\y\to\infty}} \left(W_p^{\mathrm F}-H_p\right)((x,x),(y,y))\end{align}
(recall \cref{eq:integral,eq:lambda2}), so we need to deal with the two segments in $\gamma_{R_{\phi_0},R_{\phi_M}}$ off the diagonal (recall \cref{eq:optimal:path}). By \cref{eq:f:asymptotics,eq:h:asymptotics}, if $N_1\neq 2\Lambda$, then \cref{eq:def:N1} implies that $\sh(R_{N_1}) \leq \lng(R_{N_1})/C_2 \leq \Lambda/C_2 < \log(1/q)/(C_1q)$ and
\begin{align}
\nonumber H_p((\sh(R_{N_1}),\sh(R_{N_1})),R_{N_1})&{}=(\lng(R_{N_1})-\sh(R_{N_1}))(f-\sqrt q h)(q\sh(R_{N_1}))\\
&{}\ge \frac{2}{q}\int_{q\sh(R_{N_1})}^{q\Lambda}f.
\label{eq:truncation:upper}\end{align}
Similarly, using \cref{lem:anisotropic:bound:frobose,eq:def:N0}, recall that $f(z)=-\log z+O(z)$ for $z\to 0$ to get 
\[\bbP_{p}\left(\cI^{\mathrm F}_{\mathrm{loc}}(R_{N_0})\right)\le (2p\sh(R_{N_0}))^{N_0-1}\le \exp(-N_0f(q\lng(R_{N_0})/C_1)).\]
Hence, if $N_0\neq 2$, then \cref{eq:f:asymptotics,eq:h:asymptotics} give
\begin{align}
\nonumber
&\bbP_p\left(\cI^{\mathrm F}_{\mathrm{loc}}(R_{N_0})\right)\exp(-H_p(R_{N_0},(\lng(R_{N_0}),\lng(R_{N_0}))))\\
\nonumber &{}=\bbP_p\left(\cI^{\mathrm F}_{\mathrm{loc}}\left(R_{N_0}\right)\right)\exp\left(-(\lng(R_{N_0})-\sh(R_{N_0}))f(q\lng(R_{N_0}))+\frac1{\sqrt q}\int_{q\sh(R_{N_0})}^{q\lng(R_{N_0})}h\right)\\
\nonumber&{}\le \exp\left(-2\lng(R_{N_0})(\log(q\lng(R_{N_0})) + C_0)\right)\\
\label{eq:truncation:lower}&{}\le \exp\left(-\frac2q\int_{q}^{q\lng(R_{N_0})}f\right).\end{align}
Assembling \cref{eq:truncation:lower,eq:truncation:upper} and using that $\sh(R_{\phi_M})\ge C_1/q\ge 1/(C_1q)\ge \lng(R_{\phi_0})$ by definition, we get (regardless whether $N_0=2$ and/or $N_1=2\Lambda$)
\begin{align*}
\bbP_p\left(\cI^{\mathrm F}_{\mathrm{loc}}\left(R_{\phi_0}\right)\right)\exp\left(-H_p\left(\gamma_{R_{\phi_0},R_{\phi_M}}\right)\right)
&{}\le \exp\left(-\frac{2}q\int_{q}^{q\Lambda}(f-\sqrt qh)\right)\\
&{}\le \exp\left(-\frac{2\lambda^{\mathrm F}_1}{q}+\frac{2\lambda^{\mathrm F}_2}{\sqrt q}+3\log\frac1q\right),\end{align*}
using \cref{eq:f:asymptotics,eq:integral,eq:h:asymptotics,eq:lambda2}. Further recalling \cref{eq:main:lower:applied,eq:Hp}, we deduce that for any good coarse sequence $(R_{\phi_j})_{j=0}^M$
\begin{multline}
\label{eq:good:coarse}
\bbP_p\left(\cI^{\mathrm F}_{\mathrm{loc}}(R_{\phi_0})\cap \bigcap_{j=1}^M\cC^{\mathrm F}\left(R_{\phi_{j-1}},R_{\phi_j}\right)\right)\\
\begin{aligned}[t]&{}\le \bbP_p\left(\cI^{\mathrm F}_{\mathrm{loc}}(R_{\phi_0})\right)\exp\left(-H_p\left(\gamma_{R_{\phi_0},R_{\phi_M}}\right)+C_2\log\frac1q+C_3\frac{\log^{3/2}(1/q)}{\sqrt[3]q}\right)\\
&{}\le \exp\left(-\frac{2\lambda^{\mathrm F}_1}{q}+\frac{2\lambda^{\mathrm F}_2}{\sqrt q}+2C_3\frac{\log^{3/2}(1/q)}{\sqrt[3]q}\right).\end{aligned}
\end{multline}

We finally turn back to \cref{eq:bound_Iloc}.
Together with \cref{eq:strange,eq:coarse:seq:number,eq:good:coarse}, it leads to
\begin{equation}
\bbP_p\left(\cI^{\mathrm F}_{\mathrm{loc}}(R)\right) \le e^{-C_1/q}+\exp\left(-\frac{2\lambda_1^{\mathrm F}}{q}+\frac{2\lambda^{\mathrm F}_2}{\sqrt q}+\frac{2\log^{5/3}(1/q)}{\sqrt[3]q}\right),
\end{equation}
This concludes the proof of \cref{eq:main:lower} and, therefore, of \cref{th:main}, modulo proving \cref{prop:main:lower}.

\subsection{One coarse step: proof of Proposition~\ref{prop:main:lower}}
\label{subsec:lower:step}

The proof of \cref{prop:main:lower} is also based on the decomposition into growth sequences introduced in \cref{sec:sequences}. Fix rectangles $S$ and $T$ as in the statement. We may further assume that $\sh(T)\ge C_2$, because otherwise $\bbP_p(\cC^{\mathrm{F}}(S,T))\le 1\le e^{-W^{\mathrm{F}}_p(\gamma_{S,T})+C_3\log(1/q)}$.
We begin by using \cref{cor:Markov} to write
\begin{equation}
\label{eq:lower:initial}    \bbP_p \left( \cC^F(S, T) \right) = \frac{\bbP_p ( \cF^S(\infty) = T )}{e^{-2 \phi(T) q}}
    \overset{\phi(T) \leq 2\Lambda}\leq e^{5C_1 \log(1/p)} \; \bbP_p \left( \cF^S(\infty) = T \right),
\end{equation}
so that one can focus on the last probability above. In turn, it can be decomposed as
\begin{equation}
  \label{eq:lower:sum:paths}
    \bbP_p \left( \cF^S(\infty) = T \right) = \sum_{F_1, \dots, F_J}
    \prod_{i = 1}^J \bbP_p \left( \cT(F_{i - 1}, F_i) \right),
\end{equation}
where $J = \phi(T) - \phi(S)$ and $F_1, \dots, F_J$ ranges over the possible trajectories of $\cF^S$ reaching~$T$.

The next result is analogous to \cref{lem:single:upper}, but it is not restricted to good transitions.

\begin{lemma}
\label{lem:single:lower}
Let $(F, F')$ be a transition in a trajectory from $S$ to $T$ such that $F=F(i,j;k,l;s)$ and $F'=F(i-\alpha,j-\beta;k+\gamma,l+\delta;s')$. Setting $a = \sh(T)$, we have
\begin{multline*}
\bbP_p(\cT(F,F')) \\
\leq
  \begin{cases}
  1 \times e^{-qa} \times e^{O(C_2qa^{2/3})
  } & \text{buffer creation,}\\
  e^{-(\alpha+\gamma)f(q(l-j))-(\beta+\delta)f(q(k-i))}\times 1\times 1  &\text{loop,}\\
  e^{-f(q(l-j))-f(q(k-i+1))}\times pe^{f(qa)}\times e^{O(C_2a^{-1/3})
  }&\text{$1$-buffer deletion,}\\
  e^{-(\alpha + \gamma)f(q(l-j)) - (\beta + \delta)f(q(k-i+\alpha+\gamma))}\times p^2e^{2 f(qa)}\times e^{O(C_2a^{-1/3})
  }&\text{$2$-buffer deletion,}\\
  e^{-2f(q(l-j))-2f(q(k-i+2))}\times p^3e^{3f(qa)}\times e^{O(C_2a^{-1/3})
  }&\text{$3$-buffer deletion.}
  \end{cases}
\end{multline*}
Note that $\alpha, \beta, \gamma, \delta \in \{0, 1\}$ and their sum can assume any value between $0$ and $4$, corresponding to the five cases above respectively.
\end{lemma}

\begin{remark}
It is worth noting that:
\begin{itemize}
\item The lemma only states upper bounds, in contrast with \cref{lem:single:upper}.
\item The same observations made in \cref{rem:single:upper} are in place here, with the manipulation of the $A$ term being pertinent to all types of buffer deletions.
\item In the case of buffer deletions, $B$ includes a factor of the form $e^{K f(qa)}$, where $K$ always corresponds to the number of buffers being deleted.
\item The first three cases of \cref{lem:single:lower} were considered in \cref{lem:single:upper} as well. However, the error estimates differ because in this section we deal with more skewed aspect ratios and only provide upper bounds.
\end{itemize}
\end{remark}
\begin{proof}
For concreteness, consider the case $s=3$, $s'=1''$,  $\alpha=\beta=\gamma=1$ and $\delta=0$, which is a 2-buffer deletion. In this case, by \cref{tab:transitions}, we have
\begin{align*}\cT(F,F')&{}=p^2e^{-f(q(k-i))-2q}\le e^{-2f(q(l-j))-f(q(k-i+2))} p^2 e^{2f(qa)}\\
&{}=e^{-(\alpha+\gamma)f(q(l-j))-(\beta+\delta)f(q(k-i+\alpha+\gamma))}p^2e^{2f(qa)}e^{O(q(a-\sh(S))|f'(q\sh(S))|)}\\
&{}=e^{-(\alpha+\gamma)f(q(l-j))-(\beta+\delta)f(q(k-i+\alpha+\gamma))}p^2e^{2f(qa)}e^{O(C_2a^{1/3})},\end{align*}
where in the first inequality we used that $f$ is non-increasing and in the second one we recalled that $|f'(z)|\le O(1/z)$. The other transitions are treated analogously.
\end{proof}

We are now in position to prove \cref{prop:main:lower}.

\begin{proof}[Proof of \cref{prop:main:lower}]
With \cref{lem:single:lower}, we can return to \cref{eq:lower:sum:paths} and fix any trajectory $F_1, \dots, F_J$ from $S$ to $T$.
This trajectory naturally induces a path $\gamma$ on $\mathbb{Z}^2$ going from $(a_S, b_S)$ to $(a_T, b_T)$ (the dimensions of $S$ and $T$ respectively).
This path only makes unitary jumps to the north of east directions and it is obtained as in \cref{fig:good_trajectory} by replacing any transitions involving a buffer deletion by multiple unitary jumps.
Further let $K_0+4,K_1,K_2,K_3,K_4$ be the number of buffer creations, loops, single, double and triple buffer deletions in the trajectory $F_1,\dots,F_J$ respectively. Observe that 
\begin{align}
\label{eq:K:relations}
K_0&{}=K_2+2K_3+3K_4,&\sum_{i=0}^4iK_i&{}=\phi(T)-\phi(S)
\end{align}
(recall \cref{tab:transitions}). Recalling \cref{eq:def:Wp,lem:single:lower}, we obtain
\begin{align}
  \nonumber
  \prod_{i = 1}^J \bbP_p \left( \cT(F_{i - 1}, F_i)\right) & =
  e^{-W^F_p(\gamma)} \times p^{K_0} e^{K_0 (- q a + f(qa))} \times e^{-4 q a  + O(K_0 +1)C_2(q a^{2/3} + a^{-1/3})},\\
  \label{eq:lower:prod:path}
  &\leq e^{-W^F_p(\gamma_{S,T})} \times p^{K_0} e^{K_0 (- q a + f(qa))} \times e^{ O(K_0+1)C_2(qa^{2/3}+a^{-1/3})},
\end{align}
using \cref{eq:K:relations} and \cref{lem:optimal:path}. Moreover, the total number of transitions is $J=4+\sum_{i=0}^4K_i=4+\phi(T)-\phi(S)$. Hence, summing \cref{eq:lower:prod:path} over all trajectories gives
\begin{equation}
\label{eq:lower:sum:trajectories}
\sum_{F_1, \dots, F_J} \prod_{i=1}^{J} \bbP_p \left( \cT(F_{i - 1}, F_i )\right)
  \le e^{-W_p^{\mathrm F}(\gamma_{S,T})}p^{-3}(\cI+\cM(P'))^{J-1}(0,3),\end{equation}
where $\cI$ is the identity matrix and
\begin{align*}
\cM(x)&{}=\begin{pmatrix}
0 & x & 0 & 0 & 0 & 0\\
1 & 0 & x & 0 & 0 & 0\\
1 & 1 & 0 & x & 0 & 0\\
1 & 2 & 1 & 0 & 1 & 1\\
1 & 0 & 0 & x & 0 & 1\\
1 & 0 & 0 & 0 & x & 0
\end{pmatrix}&P'&{}=pe^{-aq+f(aq)+O(C_2)(qa^{2/3}+a^{-1/3})},
\end{align*}
recall \cref{fig:augmented:cycle}.

Note that the matrix $\mathcal{M}(x)$ does not include the $1''$ state, like in \cref{eq:def:M:F:upper}.
However, since we are now looking for an upper bound on the probability, we had to compensate for this removal by doubling the entry corresponding to the $3 \to 1$ transition.
In order to see why this is enough, note first that the weights going out of $1$ and $1''$ are the same.
Therefore, since we are evaluating $(\mathcal{I} + \mathcal{M})^{J - 1}$ at $(0, 3)$, we can join together the transitions $3 \to 1 \to s$ with those of type $3 \to 1'' \to s$.

Clearly, setting $P=pe^{-aq+f(aq)}$, we have $\cM(P')\le\cM(P)e^{O(C_2)(qa^{2/3}+a^{-1/3})}$ entry-wise. Note that since $a=\sh(T)\ge C_2$, we have $P\le pe^{f(aq)}\le C^{-2}_1$.
We next seek to bound the spectral radius $\rho(\cM(P))$, viewing it as a perturbation of the same matrix without the double and triple buffer deletions, which amounts to the matrix we already considered in \cref{eq:def:M:F:upper}.

The characteristic polynomial of $\cM(P)/\sqrt P$ reads
\[(X-1)(X+1)\left(X^4-4X^2+2-4X\sqrt P - P\right).\]
Setting $(\lambda_i)_{i=1}^6=(\sqrt{2+\sqrt 2},1,\sqrt{2-\sqrt2},-\sqrt{2-\sqrt 2},-1,-\sqrt{2+\sqrt 2})$, we have
\[X^4-4X^2+2=(X-\lambda_1)(X-\lambda_3)(X-\lambda_4)(X-\lambda_6).\]
Yet, for some constant $C>0$ and any $y\in\{\lambda_i+\xi C\sqrt{P}:i\in\{1,\dots,6\},\xi\in\{-1,1\}\}$, we have
\[\left|4y\sqrt P+P\right|<\left|y^4-4y^2+2\right|.\]
Hence, the sign of the characteristic polynomial at $\lambda_i\pm C\sqrt P$ remains unchanged by the perturbation. Since $\sqrt P<1/C_1$, the intermediate value theorem allows us to conclude that $\cM(P)/\sqrt P$ has eigenvalues $(\lambda'_i)_{i=1}^6$ with $|\lambda_{i}-\lambda'_i|\le C\sqrt P$ for $i\in\{1,\dots,6\}$.\footnote{Note that if the non-perturbed matrix had complex eigenvalues, we would have used Rouch\'e's theorem instead.} In particular,
\[\rho(\cM(P))\le \sqrt{2+\sqrt2}\sqrt Pe^{O(\sqrt P)}.\]

However, in \cref{eq:lower:sum:trajectories}, we need a particular coefficient of a power of $\cI+\cM(P)$, so we also need some control on the norm of $\cM(P)$ and not only its spectral radius (because the matrix is not symmetric). The following linear algebra result can be deduced from the surprisingly recent bound of Jiang and Lam \cite{Jiang97}. For the reader's convenience we produce a simpler proof.
\begin{lemma}
\label{lem:lin:alg}
Fix a dimension $d\ge 1$ and a submultiplicative norm $\vvvert\cdot\vvvert$ on the space $\bbM_{d}(\bbC)$ of $d\times d$ complex matrices.\footnote{For instance, the operator norm, but recall that in finite dimension all norms are equivalent.} Let $(\lambda_i)_{i=1}^d$ be the eigenvalues of $\cM\in\bbM_d(\bbC)$. Assume that $\varepsilon=\min_{i\neq j}|\lambda_i-\lambda_j|>0$. Then for every $n\ge 0$ we have 
\[\vvvert\cM^n\vvvert\le d\left(\frac{(1+\vvvert\cI\vvvert)\vvvert\cM\vvvert}{\varepsilon}\right)^{d-1}\rho(\cM)^n.\]
\end{lemma}
\begin{proof}
This proof is due to Quentin Moulard. 

Consider the Lagrange interpolation polynomials
\[L_i(X)=\frac{\prod_{j\neq i}(X-\lambda_j)}{\prod_{j\neq i}(\lambda_j-\lambda_i)}.\]
Note that $\cM^n=\sum_{i=1}^d\lambda_i^n L_i(\cM)$, so 
\[\vvvert\cM^n\vvvert\le d\rho(\cM)^n\max_{i}\vvvert L_i(\cM)\vvvert.\]
Yet, by submultiplicativity, for any $i\in\{1,\dots,d\}$, we get
\[
    \vvvert L_i(\cM)\vvvert\le \frac{\prod_{j\neq i}\vvvert\cM-\lambda_j\cI\vvvert}{\prod_{j\neq i}|\lambda_j-\lambda_i|}\le \left(\frac{\vvvert\cM\vvvert + \rho(\cM)\vvvert\cI\vvvert}{\varepsilon}\right)^{d-1},
\]
which concludes the proof since $\rho(\cM)=\lim_{m\to\infty}\vvvert \cM^m\vvvert^{1/m}\le \vvvert\cM\vvvert$.
\end{proof}
We are now ready to conclude the proof of \cref{prop:main:lower}. Assembling \cref{lem:lin:alg,eq:lower:sum:trajectories,eq:lower:sum:paths,eq:lower:initial}, we obtain
\begin{multline*}
\bbP_p\left(\cC^{\mathrm F}(S,T)\right)\\\le e^{-W_p^{\mathrm F}(\gamma_{S,T})}p^{-6C_1}\left(1+\sqrt{2+\sqrt2}\sqrt P\exp\left(O\left(\sqrt{P}+C_2\left(qa^{2/3}+a^{-1/3}\right)\right)\right)\right)^{\phi(T)-\phi(S)}.\end{multline*}
Recalling that $h(aq)=\sqrt{2+\sqrt2}\sqrt{P/p}$ (see \cref{eq:def:h,eq:def:f}) and \cref{eq:def:q}, we get
\begin{align*}
\bbP_p\left(\cC^{\mathrm F}(S,T)\right)\le{}&e^{-W_p^{\mathrm F}(\gamma_{S,T})}p^{-6C_1}e^{(\phi(T)-\phi(S))\sqrt ph(aq)(1+O(\sqrt P+C_2(qa^{2/3}+a^{-1/3})))}\\
\le{}& e^{-W_p^{\mathrm F}(\gamma_{S,T})+(\phi(T)-\phi(S))\sqrt qh(aq)}p^{-6C_1}e^{C_2^3a^{2/3}\sqrt P (\sqrt P+qa^{2/3}+a^{-1/3})}\\
\le&{}e^{-W_p^{\mathrm F}(\gamma_{S,T})+(\phi(T)-\phi(S))\sqrt qh(aq)}p^{-6C_1}e^{C_2^4},\end{align*}
where in the last inequality we observed that $\sqrt P=pe^{-aq+f(aq)}= O(1/a)$ by \cref{eq:def:q,eq:f:asymptotics} and $a^{-1/6}+qa^{5/6}\le 1$, since $1\le a\le\Lambda$. This concludes the proof of \cref{prop:main:lower}.
\end{proof}

\section{The bootstrap percolation paradox}
\label{sec:paradox}
Having completed the proof of \cref{th:main}, we may return to the bootstrap percolation paradox. We focus on the Frob\"ose model, but one can proceed analogously for the two-neighbour and modified two-neighbour models (see \cref{fig:modified}). 

Recalling \cref{cor:reduction}, estimating $\tau$ is equivalent to computing 
\[\widetilde\Pi(p)=\left(\bbP_p\left(\cI_{\mathrm{loc}}^{\mathrm{F}}(R(\Lambda,\Lambda))\right)\right)^{-1/2}.\]
Also recalling \cref{eq:Ri:decomposition,obs:stacking,cor:Markov}, we obtain that
\[\widetilde\Pi(p)=p^{O(C_1)}\max_{x\in R(\Lambda,\Lambda)}\left(\bbP_p\left(\exists t\ge 0:\cF^{\{x\}}(t)=R(\Lambda,\Lambda)\right)\right)^{-1/2}.\]
Proceeding in a very similar way, it is possible to show the more convenient result (recall that $\phi(R)$ denotes the semi-perimeter of the rectangle $R$)
\begin{equation}
\label{eq:def:Pip}
\Pi(p):=\left(\bbP_p\left(\exists t\ge 0:\phi\left(\cF^{\{0\}}(t)\right)=\left\lceil\frac{2\log(1/p)}{p}\right\rceil\right)\right)^{-1/2}=\widetilde\Pi(p)p^{O(C_1)}.\end{equation}

In order to compute $\Pi(p)$ more efficiently, we project the Markov chain $\cF^{\{0\}}$ to a process on $\bbN^2\times\{0,1,1',1'',2,2',3,4\}$ via the mapping $F(a,b;c,d;s)\mapsto (c-a,d-b,s)$. We denote this process $(\cR(t))_{t\ge 0}$ with $\cR(0)=(1,1,0)$. Using translation invariance, it is not hard to check that this process is also a Markov chain, whose transition probabilities can be obtained by summing the appropriate transition probabilities of the original chain listed in \cref{tab:transitions} (the corresponding expressions can be found in the supplementary material Rust code\footnote{\label{footnote:supplementary}Implementations in C, Rust and Python-Taichi can be found in \url{https://github.com/augustoteixeira/bootstrap} or as part of the arxiv submission source.}). Since $\phi(\cF^{\{0\}}(t))$ increases by at most 4 at each step and the chain $\cR$ has no cycles except the absorbing states, we can compute the probability of reaching each possible state of $\cR$ recursively via dynamic programming. This way, for any fixed $p$, we can compute $\Pi(p)$ exactly. 

\begin{table}
    \centering
    \begin{tabular}{c|c||c|c||c|c||c|c}
    \scriptsize{$\log_2\frac1p$}&\scriptsize{$\log\Pi(p)$}&\scriptsize{$\log_2\frac1p$}&\scriptsize{$\log\Pi(p)$}&\scriptsize{$\log_2\frac1p$}&\scriptsize{$\log\Pi(p)$}&\scriptsize{$\log_2\frac1p$}&\scriptsize{$\log\Pi(p)$}\\\hline
    \scriptsize{2}&\tiny{1.8231469544522168}&\scriptsize{6}&\tiny{71.22104704459092}&\scriptsize{10}&\tiny{1519.9177238798902}&\scriptsize{14}&\tiny{26243.443273103207}\\
    \scriptsize{3}&\tiny{4.742671577932995}&\scriptsize{7}&\tiny{159.10494055779233}&\scriptsize{11}&\tiny{3130.360634994818}&\scriptsize{15}&\tiny{52891.342141028406}\\
    \scriptsize{4}&\tiny{12.392931032600497}&\scriptsize{8}&\tiny{344.5259389380065}&\scriptsize{12}&\tiny{6393.748024566681}&\scriptsize{16}&\tiny{106363.14234086743}\\
    \scriptsize{5}&\tiny{30.54732365348029}&\scriptsize{9}&\tiny{729.489374480061}&\scriptsize{13}&\tiny{12981.361913134877}&\scriptsize{17}&\tiny{213556.78508818566}
    \end{tabular}
    \caption{The numerical results for $\Pi(p)$ of \cref{eq:def:Pip} for (local) Frob\"ose bootstrap percolation.}
    \label{tab:frobose}
\end{table}

The main practical issue with the above algorithm is that the probabilities we compute are of order $\exp(-\lambda_1^{\mathrm F}/p)$, which quickly becomes too small for conventional data types (IEEE 754).
In order to substantially increase the precision, we store the logarithms of each probability $a = \log(r)$, using the expression $\log(r + s) = \log(e^a + e^b) = \max\{a, b\} + \log(1 + e^{-|a - b|})$ in order to obtain the log of the sum or two probabilities without losing precision.
Other optimisations have also been implemented, such as restricting calculations to only regions where the probabilities are non-zero.
But most importantly, we have made use of a Graphical Processing Unit (GPU) in order to take advantage of the parallelism built into our algorithm, which was responsible for large speed benefits.

Our implementation of the algorithm can be found in the supplementary material,$^{\ref{footnote:supplementary}}$ but we emphasise that it still has a lot of room for further optimisations (hand-written CUDA code, optimal scheduling, extra and more powerful GPU's...).
We ran this code on a Intel Xeon E3-1240 v3, with a single NVIDIA GeForce RTX 4060 8GB during 24 hours. The resulting data is plotted in \cref{fig:frobose} and given in \cref{tab:frobose}. In \cref{fig:frobose:2} we use it to successively estimate $\alpha$, $\lambda_1^{\mathrm F}$, $\beta<\alpha$ and $\lambda_2^{\mathrm F}$, given the exact values of the previous ones, so that
\begin{equation}
\label{eq:estimation}
\log \Pi(p)\approx \frac{\lambda_1^{\mathrm F}}{p^\alpha}-\frac{\lambda_2^{\mathrm{F}}}{p^\beta}.\end{equation}
Since there the numerical results are exact up to computation precision and we are interested in asymptotics, we use a simple linear regression based on the last three data points. This somewhat arbitrary choice was made before the data was generated, in order to avoid overfitting. The results are in excellent agreement with the rigorous result \cref{eq:main}. Moreover, directly fitting all four parameters in \cref{eq:estimation} based on the last four data points in \cref{tab:frobose} yields 
\begin{align}
\label{eq:4fit}
\alpha&{}\approx0.99978
,&\lambda_1^{\mathrm{F}}&{}\approx1.6507
,&\beta&{}\approx0.53239
,&
\lambda_2^{\mathrm{F}}&{}\approx 4.2518
,\end{align}
which is also in good agreement with \cref{th:main}, see code for fitting and plotting in the tagged repository. Note that the data points in \cref{fig:frobose:2,fig:frobose} are simply reparametrisations of the ones given in \cref{tab:frobose}).

\begin{figure}
    \centering
\subcaptionbox{Simultaneously estimating $\alpha\approx1.007$ and $\lambda_1^{\mathrm F}\approx 1.505$.}[.48\textwidth]{
\centering    
\begin{tikzpicture}[x=0.55cm,y=0.35cm]
\draw[->] (1,0) -- (12.5,0) node[above left]{$\log\frac1p$};
\foreach \x in {1,2,3,4,5,6,7,8,9,10,11,12}
\draw[shift={(\x,0)}] (0pt,2pt) -- (0pt,-2pt) node[below] {\footnotesize $\x$};
\draw[->] (1,0) -- (1,13) node[right]{$\log\log\Pi(p)$};
\foreach \y in {2,4,6,8,10,12}
\draw[shift={(1,\y)}] (2pt,0pt) -- (-2pt,0pt) node[left] {\footnotesize $\y$};
\fill[color=red] (1.3862943611198906, 0.60056410377630554) circle (2pt);
\fill[color=red] (2.0794415416798357, 1.5566006009671121) circle (2pt);
\fill[color=red] (2.7725887222397811, 2.5171262320386472) circle (2pt);
\fill[color=red] (3.4657359027997265, 3.4192770763292959) circle (2pt);
\fill[color=red] (4.1588830833596715, 4.2657883792943716) circle (2pt);
\fill[color=red] (4.8520302639196169, 5.0695639880218302) circle (2pt);
\fill[color=red] (5.5451774444795623, 5.8421693820691747) circle (2pt);
\fill[color=red] (6.2383246250395077, 6.5923448023736384) circle (2pt);
\fill[color=red] (6.9314718055994531, 7.3264114833489051) circle (2pt);
\fill[color=red] (7.6246189861593985, 8.0489034957452024) circle (2pt);
\fill[color=red] (8.317766166719343, 8.7630759207367106) circle (2pt);
\fill[color=red] (9.0109133472792884, 9.4712699087242758) circle (2pt);
\fill[color=red] (9.7040605278392338, 10.175171456747272) circle (2pt);
\fill[color=blue] (10.397207708399179, 10.87599493982103) circle (2pt);
\fill[color=blue] (11.090354888959125, 11.574614389306269) circle (2pt);
\fill[color=blue] (11.78350206951907, 12.271658048596455) circle (2pt);
\draw[color=blue] (10.397207708399179, 10.876257571520204) node[left,color=black]{$1.00676\log\frac{1}{p}+\log1.50499$}--(11.78350206951907, 12.271920680295629);
\end{tikzpicture}}
\quad
\subcaptionbox{Estimating $\lambda_1^{\mathrm F}\approx 1.635$, given $\alpha=1$.}[.48\textwidth]{
\centering    
\begin{tikzpicture}[x=0.47cm,y=0.2cm]
\draw[->] (0,0) -- (13.5,0) node[above left]{$\frac{1}{10^{4}p}$};
\foreach \x in {0,1,2,3,4,5,6,7,8,9,10,11,12,13}
\draw[shift={(\x,0)}] (0pt,2pt) -- (0pt,-2pt) node[below] {\footnotesize $\x$};
\draw[->] (0,0) -- (0,22) node[right]{$10^{-4}\log\Pi(p)$};
\foreach \y in {0,3,6,9,12,15,18,21}
\draw[shift={(0,\y)}] (2pt,0pt) -- (-2pt,0pt) node[left] {\footnotesize $\y$};
\fill[color=red] (0.0004, 0.00018231469544522166) circle (2pt);
\fill[color=red] (0.0008, 0.00047426715779329954) circle (2pt);
\fill[color=red] (0.0016, 0.0012392931032600496) circle (2pt);
\fill[color=red] (0.0032, 0.0030547323653480293) circle (2pt);
\fill[color=red] (0.064, 0.0071221047044590915) circle (2pt);
\fill[color=red] (0.0128, 0.015910494055779233) circle (2pt);
\fill[color=red] (0.0256, 0.03445259389380065) circle (2pt);
\fill[color=red] (0.0512, 0.0729489374480061) circle (2pt);
\fill[color=red] (0.1024, 0.15199177238798903) circle (2pt);
\fill[color=red] (0.2048, 0.3130360634994818) circle (2pt);
\fill[color=red] (0.4096, 0.6393748024566681) circle (2pt);
\fill[color=red] (0.8192, 1.2981361913134877) circle (2pt);
\fill[color=red] (1.6384, 2.624344327310321) circle (2pt);
\fill[color=blue] (3.2768, 5.289134214102841) circle (2pt);
\fill[color=blue] (6.5536, 10.636314234086743) circle (2pt);
\fill[color=blue] (13.1072, 21.355678508818567) circle (2pt);
\draw[color=blue] (3.2768, 5.2855621805651207) --(13.1072, 21.35389249204971) node[midway,left,color=black]{$\frac{1.634555}{10^4p}+C$};
\end{tikzpicture}}\\
\subcaptionbox{Simultaneously estimating $\beta\approx0.5100$ and $\lambda_2^{\mathrm F}\approx 5.027$, given $\alpha=1$ and $\lambda_1^{\mathrm{F}}=\pi^2/6$.}[.48\textwidth]{
\centering    
\begin{tikzpicture}[x=0.55cm,y=0.7cm]
\draw[->] (1,1.5) -- (12.5,1.5) node[above left]{$\log\frac1p$};
\foreach \x in {1,2,3,4,5,6,7,8,9,10,11,12}
\draw[shift={(\x,1.5)}] (0pt,2pt) -- (0pt,-2pt) node[below] {\footnotesize $\x$};
\draw[->] (1,1.5) -- (1,8) node[right]{$\log\left(\frac{\pi^2}{6p}-\log\Pi(p)\right)$};
\foreach \y in {1.5,2.5,3.5,4.5,5.5,6.5,7.5}
\draw[shift={(1,\y)}] (2pt,0pt) -- (-2pt,0pt) node[left] {\footnotesize $\y$};
\fill[color=red] (1.3862943611198906, 1.5595308805175614) circle (2pt);
\fill[color=red] (2.0794415416798357, 2.1302298221956164) circle (2pt);
\fill[color=red] (2.7725887222397811, 2.6337586044884755) circle (2pt);
\fill[color=red] (3.4657359027997265, 3.0951506615424593) circle (2pt);
\fill[color=red] (4.1588830833596715, 3.5279690312664931) circle (2pt);
\fill[color=red] (4.8520302639196169, 3.9405447652489878) circle (2pt);
\fill[color=red] (5.5451774444795623, 4.3382991495395649) circle (2pt);
\fill[color=red] (6.2383246250395077, 4.72487907928942) circle (2pt);
\fill[color=red] (6.9314718055994531, 5.1028787190762612) circle (2pt);
\fill[color=red] (7.6246189861593985, 5.4742197557928414) circle (2pt);
\fill[color=red] (8.317766166719343, 5.8403564807948118) circle (2pt);
\fill[color=red] (9.0109133472792884, 6.2024099272892776) circle (2pt);
\fill[color=red] (9.7040605278392338, 6.5612519683354353) circle (2pt);
\fill[color=blue] (10.397207708399179, 6.9175643735790402) circle (2pt);
\fill[color=blue] (11.090354888959125, 7.2718820538025799) circle (2pt);
\fill[color=blue] (11.78350206951907, 7.6246252955853304) circle (2pt);
\draw[color=blue] (10.397207708399179, 6.9178267799858366) node[left,color=black]{$0.510037\log\frac{1}{p}+\log5.027234$}--(11.78350206951907, 7.6248877019921277);
\end{tikzpicture}}
\quad
\subcaptionbox{Estimating $\lambda_2^{\mathrm F}\approx 5.735$, given $\alpha=1$, $\lambda_1^{\mathrm F}=\pi^2/6$ and $\beta=1/2$.}[.48\textwidth]{
\centering    
\begin{tikzpicture}[x=1.6cm,y=0.2cm]
\draw[->] (0,0) -- (4,0) node[above left]{$\frac{1}{10^2\sqrt p}$};
\foreach \x in {0,0.5,1,1.5,2,2.5,3,3.5}
\draw[shift={(\x,0)}] (0pt,2pt) -- (0pt,-2pt) node[below] {\footnotesize $\x$};
\draw[->] (0,0) -- (0,22) node[right]{$\frac{1}{10^{2}}\left(\frac{\pi^2}{6}-\log\Pi(p)\right)$};
\foreach \y in {0,3,6,9,12,15,18,21}
\draw[shift={(0,\y)}] (2pt,0pt) -- (-2pt,0pt) node[left] {\footnotesize $\y$};
\fill[color=red] (0.02, 0.047565893129406886) circle (2pt);
\fill[color=red] (0.028284271247461901, 0.084168009568528143) circle (2pt);
\fill[color=red] (0.04, 0.13926014036971121) circle (2pt);
\fill[color=red] (0.056568542494923803, 0.22090566485662955) circle (2pt);
\fill[color=red] (0.08, 0.34054733233695567) circle (2pt);
\fill[color=red] (0.11313708498984761, 0.51446619998780641) circle (2pt);
\fill[color=red] (0.16, 0.76577182175139469) circle (2pt);
\fill[color=red] (0.22627416997969521, 1.1271686774623095) circle (2pt);
\fill[color=red] (0.32, 1.6449476057269359) circle (2pt);
\fill[color=red] (0.45254833995939042, 2.3846433391034947) circle (2pt);
\fill[color=red] (0.64, 3.4390191324365378) circle (2pt);
\fill[color=red] (0.90509667991878084, 4.9393796248579358) circle (2pt);
\fill[color=red] (1.28, 7.0715647813813405) circle (2pt);
\fill[color=blue] (1.8101933598375617, 10.098573614542767) circle (2pt);
\fill[color=blue] (2.56, 14.392566640979348) circle (2pt);
\fill[color=blue] (3.6203867196751234, 20.480129217450731) circle (2pt);
\draw[color=blue] (1.8101933598375617, 10.096180492385326) --(3.6203867196751234, 20.478437024544995) node[midway,left,color=black]{$\frac{5.735441}{10^2\sqrt p}+C$};
\end{tikzpicture}}
\caption{Estimating the asymptotic expansion \cref{eq:estimation}, using the numerical estimates of $\Pi(p)$ from \cref{eq:def:Pip} for Frob\"ose bootstrap percolation. In \cref{th:main} we showed that in fact $\alpha=1$, $\lambda_1^{\mathrm{F}}=\pi^2/6\approx1.6449$, $\beta=1/2$, $\lambda_2^{\mathrm F}=\pi\sqrt{2+\sqrt{2}}\approx5.8049$.}
\label{fig:frobose:2}
\end{figure}

\section{Outlook}
\label{sec:outlook}
\subsection{Frob\"ose and two-neighbour bootstrap percolation}
Naturally, one may want to improve the bounds in \cref{th:main,th:2n}. In view of \cref{th:locality}, this problem only concerns the corresponding local models. Such sharper results can be obtained along lines of our proof, but, as the precision increases, additional technical issues arise.

The main problem is the fact that we artificially cut trajectories into coarse steps, imposing their endpoints to be squares in \cref{subsec:upper:reduction}. Similarly, in \cref{subsec:lower:reduction}, we take a union bound on coarse sequences. In order to go beyond the error term $q^{-1/3}$, one would need to cut trajectories into larger pieces (or, preferably, not at all). Yet, this entails problems, when applying \cref{lem:W:to:g}. Therefore, one would like to restrict attention to paths $\gamma$ not deviating too much form the diagonal. Taking this into account and despite \cref{fig:frobose:3}, we do \emph{not} believe that the third order exponent equals $1/5$, but heuristic computations rather suggest it to be $1/4$ (recall that in \cref{th:main} we showed that it is at most $1/3$, which is certainly compatible with \cref{fig:frobose:3}).

For the two-neighbour model additional problems arise due to parity, somewhat like in Case 4.3 of the proof of \cref{prop:reduction}. More precisely, the error term on the right-hand side of \cref{lem:traversability:exact} eventually becomes large. This translates the fact that, at small scales, the amount of consecutive loop transitions we typically want to use is insufficient for the two-term recurrence \cref{eq:g:recurrrence} to stabilise. Consequently, the factor $\alpha(aq)$ in \cref{lem:loops} becomes essentially equal to 2 for even $\ell$ and $0$ for odd $\ell$. This needs to be taken into account when counting trajectories.

\subsection{Modified two-neighbour bootstrap percolation}
\label{subsec:modified}
A much more challenging model to tackle is modified two-neighbour bootstrap percolation. Much of the treatment carried out here can also be done for its local version. This leads to the function 
\begin{equation}
\label{eq:h:modified}
h_{2'}(z)=\sqrt{(2+\sqrt2)e^{-z}e^{2f(z)}}=\frac{\sqrt{2+\sqrt2}}{2\sinh(z/2)},
\end{equation}
which is not integrable at 0. The logarithmic divergence of the primitive of $h_{2'}$ led to the discovery of \cite{Hartarsky23mod-2n}. Thus, in order to obtain the second term, one needs to integrate $h_{2'}$ down to scale $1/\sqrt p$, at which point it stops being valid. Indeed, below that scale one would use essentially no loop transitions.

However, a more serious problem arises for the modified model. Namely, locality in the strong form of \cref{th:locality} does \emph{not} hold. To see this, observe that at small scales, configurations with minimal number of infections dominate. But a minimal configuration of infections internally filling a square is an object known under the name \emph{separable permutation} (this was observed in some form already in \cite{Shapiro91}; also see \cite{Dukes23} for a recent account of extremal problems in bootstrap percolation). The latter are counted by Schr\"oder numbers, which grow exponentially faster than the number of local internally filling configurations. Indeed, typical separable permutations admit a scaling limit called the \emph{Browninan separable permuton} \cite{Bassino18}, which is far from corresponding to a local decomposition. Moreover, non-local decompositions arise naturally on scale $1/\sqrt p$ from successive single buffer deletion, buffer creation, single buffer deletion, which can be replaced by finding a $1\times 1$ internally filled square at the corner of the current rectangle. Thus, we expect that the second term of the asymptotic expansion of $\log\tau$ will coincide with its its local analogue, but will be followed by a subsequent correction only $\log(1/p)$ away reflecting non-locality.

In view of the above discussion, we ask the following.
\begin{question}
\label{conj:mod:2n}
For modified two-neighbour bootstrap percolation on $\bbZ^2$ and any $\varepsilon>0$, do we have
\begin{multline*}
\lim_{p\to 0}\bbP_p\Bigg(\exp\left(\frac{\pi^2}{6p}-\frac{(\sqrt{2+\sqrt 2}+\varepsilon)\log(1/p)}{2\sqrt p}\right)\le \tau^{\mathrm{mod}}\\
\le \exp\left(\frac{\pi^2}{6p}-\frac{(\sqrt{2+\sqrt 2}-\varepsilon)\log(1/p)}{2\sqrt p}\right)\Bigg)=1?\end{multline*}
\end{question}
We supply a numerical estimate along the lines discussed in \cref{sec:paradox} in \cref{fig:modified} supporting an affirmative answer to \cref{conj:mod:2n} in the case of the local modified model. Note that $\sqrt{2+\sqrt 2}/2\approx 0.924$ is indeed only $4\%$ away from the estimated value for the second order constant.

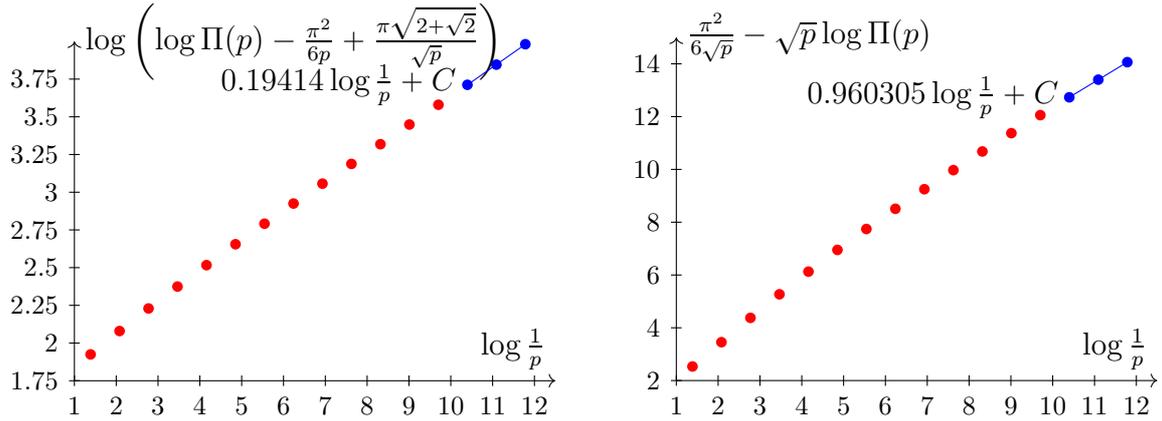
\begin{figure}
    \centering
\subcaptionbox{Estimating the third order exponent $0.194$ for Frob\"ose bootstrap percolation. Examining the discrete derivatives suggests that it is not yet close to its limiting value.\label{fig:frobose:3}}[.48\textwidth]{
\centering
\begin{tikzpicture}[x=0.55cm,y=2cm]
\draw[->] (1,1.75) -- (12.5,1.75) node[above left]{$\log\frac1p$};
\foreach \x in {1,2,3,4,5,6,7,8,9,10,11,12}
\draw[shift={(\x,1.75)}] (0pt,2pt) -- (0pt,-2pt) node[below] {\footnotesize $\x$};
\draw[->] (1,1.75) -- (1,4) node[right]{$\log\left(\log\Pi(p)-\frac{\pi^2}{6p}+\frac{\pi\sqrt{2+\sqrt 2}}{\sqrt p}\right)$};
\foreach \y in {1.75,2,2.25,2.5,2.75,3,3.25,3.5,3.75}
\draw[shift={(1,\y)}] (2pt,0pt) -- (-2pt,0pt) node[left] {\footnotesize $\y$};
\fill[color=red] (1.3862943611198906, 1.9247190957029456) circle (2pt);
\fill[color=red] (2.0794415416798357, 2.0796856982186616) circle (2pt);
\fill[color=red] (2.7725887222397811, 2.229327194210327) circle (2pt);
\fill[color=red] (3.4657359027997265, 2.3746212871028725) circle (2pt);
\fill[color=red] (4.1588830833596715, 2.516447079584414) circle (2pt);
\fill[color=red] (4.8520302639196169, 2.6552398120696648) circle (2pt);
\fill[color=red] (5.5451774444795623, 2.7912460059711899) circle (2pt);
\fill[color=red] (6.2383246250395077, 2.9249432096017456) circle (2pt);
\fill[color=red] (6.9314718055994531, 3.0569327843018503) circle (2pt);
\fill[color=red] (7.6246189861593985, 3.187828289786574) circle (2pt);
\fill[color=red] (8.317766166719343, 3.3182537286249989) circle (2pt);
\fill[color=red] (9.0109133472792884, 3.4487861844365435) circle (2pt);
\fill[color=red] (9.7040605278392338, 3.5799439114767466) circle (2pt);
\fill[color=blue] (10.397207708399179, 3.7121789795943307) circle (2pt);
\fill[color=blue] (11.090354888959125, 3.8458693096102881) circle (2pt);
\fill[color=blue] (11.78350206951907, 3.9813185853961119) circle (2pt);
\draw[color=blue] (10.397207708399179, 3.7118858219660193) node[left,color=black]{$0.19414\log\frac{1}{p}+C
$}--(11.78350206951907, 3.9810254277678006);
\end{tikzpicture}}\quad
\subcaptionbox{Estimating the second order constant $\lambda'_2\approx0.960$ for modified two-neighbour bootstrap percolation on $\bbZ^2$. In \cref{conj:mod:2n} we ask whether $\lambda'_2=\sqrt{2+\sqrt2}/2\approx0.924$.\label{fig:modified}}[.48\textwidth]{
\centering
    \begin{tikzpicture}[x=0.55cm,y=0.35cm]
\draw[->] (1,2) -- (12.5,2) node[above left]{$\log\frac1p$};
\foreach \x in {1,2,3,4,5,6,7,8,9,10,11,12}
\draw[shift={(\x,2)}] (0pt,2pt) -- (0pt,-2pt) node[below] {\footnotesize $\x$};
\draw[->] (1,2) -- (1,15) node[right]{$\frac{\pi^2}{6\sqrt p}-\sqrt p\log\Pi(p)$};
\foreach \y in {2,4,6,8,10,12,14}
\draw[shift={(1,\y)}] (2pt,0pt) -- (-2pt,0pt) node[left] {\footnotesize $\y$};
\fill[color=red] (1.3862943611198906, 2.5336173824903687) circle (2pt);
\fill[color=red] (2.0794415416798357, 3.4542750157635616) circle (2pt);
\fill[color=red] (2.7725887222397811, 4.375715428902341) circle (2pt);
\fill[color=red] (3.4657359027997265, 5.269321110664854) circle (2pt);
\fill[color=red] (4.1588830833596715, 6.127509797976364) circle (2pt);
\fill[color=red] (4.8520302639196169, 6.950806312754377) circle (2pt);
\fill[color=red] (5.5451774444795623, 7.742799297258774) circle (2pt);
\fill[color=red] (6.2383246250395077, 8.507654470088749) circle (2pt);
\fill[color=red] (6.9314718055994531, 9.24951321142202) circle (2pt);
\fill[color=red] (7.6246189861593985, 9.972215649531723) circle (2pt);
\fill[color=red] (8.317766166719343, 10.679135394361651) circle (2pt);
\fill[color=red] (9.0109133472792884, 11.373169080311992) circle (2pt);
\fill[color=red] (9.7040605278392338, 12.056750506103697) circle (2pt);
\fill[color=blue] (10.397207708399179, 12.731894171509936) circle (2pt);
\fill[color=blue] (11.090354888959125, 13.400250258565961) circle (2pt);
\fill[color=blue] (11.78350206951907, 14.063159033781114) circle (2pt);
\draw[color=blue] (10.397207708399179, 12.732802056816745) node[left,color=black]{$0.960305\log\frac1p+C
$}--(11.78350206951907, 14.064066919087924);
\end{tikzpicture}}
\caption{Estimating the asymptotics, using the numerical estimates of $\Pi(p)$ from \cref{eq:def:Pip}.}
\end{figure}

\subsection{Generalisations}
Recall that \cite{Duminil-Copin23} provides sharp thresholds analogous to \cref{eq:Holroyd} for a range of well-behaved models. It is then natural to investigate when locality holds for them. Roughly speaking, we expect this to be the case, whenever ``corner growth is not more advantageous than side growth''. If, on the contrary, ``corner growth is more advantageous than side growth'', then we expect the phenomenology of modified two-neighbour model to arise and the model to deviate from its local counterpart at some polynomial scale (smaller than the critical one).

It is also natural to consider higher dimensional models, starting with two-neighbour bootstrap percolation in 3 dimensions (or its variants), but eventually also the $r$-neighbour model (see \cite{Balogh12} for the analogue of \cref{eq:Holroyd}). It is also unclear whether locality can be useful for studying bootstrap percolation beyond $\mathbb Z^d$ (also see e.g.\ \cite{Kolesnik22}).

Independently of establishing locality, our work provides motivation for studying local models in their own right. One could define a general local model as a polygon with prescribed side directions and finite number of possible frame states which grows in each direction with a certain size-dependent probability in a Markovian way (recall \cref{subsec:Markov}). Can one quantify their probability of never being absorbed?

\subsection{Bootstrap percolation paradox}
Finally, we end with a note of optimism regarding numerical work. While simulations have previously proved highly unreliable for studying bootstrap percolation, locality has allowed us to obtain very accurate results at par with rigorous results. Can this new viewpoint enable numerical works to go beyond theoretical ones and provide reliable predictions for the behaviour of bootstrap percolation models? One interesting specific direction of research along this line is the following. In \cite{Duminil-Copin23}, the sharp threshold constant $\lambda_1$ is determined implicitly as the solution of an infinite-dimensional optimisation problem with a cost functional which is itself implicitly defined. It is therefore not clear how one could determine this constant in practice. The numerical approach discussed in \cref{sec:paradox} may allow to numerically estimate these constants.

\section*{Acknowledgements}
I.H. was supported by the Austrian Science Fund (FWF): P35428-N and ERC Starting Grant 680275 ``MALIG.''
A.T. was supported by grants ``Projeto Universal'' (406250/2016-2) and ``Produtividade em Pesquisa'' (304437/2018-2) from CNPq and
``Jovem Cientista do Nosso Estado'', (202.716/2018) from FAPERJ. We thank Quentin Moulard and Jana Reker for enlightening discussions and Instituto Superior T\'ecnico of the University of Lisbon and IMPA, where part of the work was done, for their hospitality.

\appendix
\section{Local two-neighbour bootstrap percolation}
\label{appendix}
In this appendix we discuss the additional technicalities needed for the proof of \cref{th:2n} as compared to \cref{th:main}. We will not dwell on trivial changes such as replacing $f$ by $g$, $W^{\mathrm F}$ by $W$, etc. In view of \cref{cor:reduction}, \cref{th:2n} is proved as soon as we establish 
\begin{equation}
\label{eq:2n:main}\exp\left(-\frac{2\lambda_1}{q}+\frac{2\lambda_2}{\sqrt q}-\frac{\log^{2}(1/q)}{\sqrt[3]q}\right)\le \bbP_p(\cI_{\mathrm{loc}}(R))\le \exp\left(-\frac{2\lambda_1}{q}+\frac{2\lambda_2}{\sqrt q}+\frac{\log^{2}(1/q)}{\sqrt[3]q}\right),\end{equation}
where $R=R(\Lambda,\Lambda)$ (recall \cref{eq:def:q,eq:integral,eq:def:Lambda}) and $\lambda_2=\int_0^\infty h_2$ with $h_2$ defined in \cref{eq:def:h2} below.

\subsection{Preliminaries}

\subsubsection{Refined traversability}
We first need the auxiliary function $\alpha$ for the following refinement of \cref{lem:traversability}. Recall $f,\beta,g,q$ from \cref{eq:def:q,eq:def:beta,eq:def:f,eq:def:g} and set
\begin{align}
\label{eq:def:alpha}
\alpha:(0,\infty)\to(1,2):z\mapsto\frac{2\beta(e^{-f(z)})}{\sqrt{e^{-f(z)}(4-3e^{-f(z)})}},\\
\label{eq:def:beta:bar}\bar\beta:(0,1)\to[-1/3,0):u\mapsto\frac{u-\sqrt{u(4-3u)}}{2}.\end{align}
\begin{lemma}[Refined traversability probability]
\label{lem:traversability:exact}
Let $R=R(a,b)$ be a rectangle. Set $u=1-e^{-bq}=e^{-f(bq)}$. Then
\[\left|\frac{\bbP_p(\cT_\rightarrow (R))}{\alpha(bq)\exp(-ag(bq))}-1\right|\le \left(\frac{|\bar\beta(u)|}{\beta(u)}\right)^{a+1}.\]
\end{lemma}
\begin{proof}
Let $x_n=\bbP_p(\cT_\rightarrow(R(n,b)))$ for any non-negative integer $n$. We first claim that 
\begin{equation}
\label{eq:xn}x_n=\frac{(\beta(u))^{n+1}-(\bar\beta(u))^{n+1}}{\beta(u)-\bar\beta(u)}.\end{equation}
This is verified directly for $n\in\{0,1\}$. Moreover, 
\begin{equation}
\label{eq:g:recurrrence}x_{n+2}=x_{n+1}u+x_{n}(1-u)u.\end{equation}
Solving this recurrence relation, we obtain \cref{eq:xn}. Finally, we use that
\[\alpha(bq)e^{-ag(bq)}=\frac{(\beta(u))^{a+1}}{\beta(u)-\bar\beta(u)}.\qedhere\]
\end{proof}

\subsubsection{The entropy function}
With \cref{eq:def:alpha} at hand, we are ready to define $h$. Recalling 
\cref{eq:def:f,eq:def:g,eq:def:alpha}, define the following functions from $(0,\infty)$ to itself
\begin{align}
\nonumber\cM_{2,1}(z)&{}= pe^{-f(z)+3g(z)},&\cM_{3,1}(z)&{}= pe^{-f(z)+4g(z)},\\
\nonumber\cM_{2,2}(z)&{}=2pe^{-f(z)+4g(z)-z},&\cM_{3,2}(z)&{}= pe^{-2f(z)+5g(z)-z},\\
\cM_{0,0}(z)&{}= e^{-2z},
\label{eq:def:M}
\end{align}
\begin{equation}
h_2(z)=\alpha(z)\sqrt{\left(2+\sqrt2\right)\cM_{0,0}(z)\frac{\cM_{2,1}(z)+\cM_{3,1}(z)+\cM_{2,2}(z)+\cM_{3,2}(z)}{p}}.
\label{eq:def:h2}
\end{equation}

We will need the following properties of the above functions, whose elementary proofs are left to the reader (or to a computation software). The functions $\alpha$ and $h$ are analytic,
\begin{align}
\nonumber\alpha(z)&{}=\begin{cases}
1+\frac{\sqrt z}{2}+O(z)&z\to0,\\
2-2e^{-z}+O(e^{-2z})&z\to\infty,
\end{cases}\\
\label{eq:h2:asymptotics}
h_2(z)&{}=\begin{cases}
    \sqrt{\frac{3(2+\sqrt 2)}{z}}\left(1-\frac{\sqrt z}{6}+O(z)\right)&z\to0,\\
    \frac{2\sqrt{2(2+\sqrt 2)}}{e^z}(1+O(e^{-z}))&z\to\infty.
\end{cases}
\end{align}
In particular, \cref{eq:h:asymptotics} implies that $h_2$ is integrable, so $\lambda_2=\int_0^\infty h_2\in(0,\infty)$ and one can numerically compute that $\lambda_2\approx7.054547$.

\subsubsection{Frames}
We first need to change \cref{e:buffers} into
\begin{equation}
  \label{e:buffers:2n}
  \begin{split}
    & B^r(R) = R(c,b;c+2,d), \qquad B^u(R) = R(a, d; c, d + 2)\\
    & B^l(R) = R(a - 2, b; a, d), \qquad B^d(R) = R(a, b - 2; c, b).
  \end{split}
\end{equation}
Then \cref{def:frame} is replaced by the following. In words, each buffer now has thickness 2 and whenever two adjacent ones are present, we also add the corner site between them.
\begin{definition}
\label{def:frame:2n}
A \emph{framed rectangle} is a rectangle $R(a,b;c,d)$ equipped with a label $s\in\{0,1,1',1'',2,2',2'',3,4\}$. We denote it by $F(a,b;c,d;s)$. We refer to $(c-a,d-b)$ as its \emph{dimensions} and to $s$ as its \emph{frame state}. For a framed rectangle $F=F(a,b;c,d;s)$ we set $F_\circ=R(a,b;c,d)$ and define the \emph{frame} of $F$ by
\[F_\square=\begin{cases}
\varnothing&s=0,\\
B^r(F_\circ)&s=1,\\
B^r(F_\circ)\cup \{(c,d)\}\cup B^u(F_\circ)&s=2,\\
B^r(F_\circ)\cup \{(c,d)\}\cup B^u(F_\circ)\cup\{(a-1,d)\}\cup B^l(F_\circ)&s=3,\\
\phantom{B^r(F_\circ)\cup \{(c,d)\}\cup{}}B^u(F_\circ)\cup\{(a-1,d)\}\cup B^l(F_\circ)&s=2',\\
\phantom{B^r(F_\circ)\cup \{(c,d)\}\cup B^u(F_\circ)\cup\{(a-1,d)\}\cup {}}B^l(F_\circ)&s=1',\\
\phantom{B^r(F_\circ)\cup \{(c,d)\}\cup{}}B^u(F_\circ)&s=1'',\\
B^r(F_\circ)\cup\phantom{\{(c,d)\}\cup B^u(F_\circ)\cup\{(a-1,d)\}\cup {}} B^l(F_\circ)&s=2'',\\
\begin{aligned}
&{}B^r(F_\circ)\cup \{(c,d)\}\cup B^u(F_\circ)\cup\{(a-1,d)\}\cup B^l(F_\circ)\\
&{}\cup\{(a-1,b-1)\}\cup B^d(F_\circ)\cup\{(c,b-1)\}\end{aligned}&x=4,
\end{cases}\]
see \cite{Hartarsky19}*{Fig.\ 1}. We denote $F_\blacksquare=F_\circ\cup F_\square$.
\end{definition}
Note that an additional frame state $2''$ appeared with respect to \cref{def:frame}.

\subsubsection{The framed rectangle Markov chain}
It is possible to construct a Markov chain along the lines of \cref{subsec:Markov}. Unfortunately, doing so requires an extensive amount of casework, which we have not been able to condense into a reasonable amount of space. We therefore restrict our attention to proving the lower bound of \cref{eq:2n:main}. We nevertheless mention that for the upper bound of \cref{eq:2n:main} the changes only affect \cref{subsec:lower:step} and go along the same lines as the ones for the lower bound. 

For the purposes of the lower bound of \cref{eq:2n:main}, it is sufficient to construct a sub-Markovian process by only considering the relevant transitions, which are nonetheless rather numerous. We thus replace \cref{tab:transitions} by \cref{tab:transitions:2n}.

In order to keep transition events $\cT(F,F')$ disjoint for distinct choices of $F'$ we need a finer choice of events as compared to \cref{eq:def:transitions}. Since they are completely analogous, we only provide the definitions (and associated verifications in the sequel) for transitions from framed rectangles with frame state $1$ (for other frame states the only difference, except symmetry, is that we also need to enforce a bounded number of additional sites in $F'_\square\setminus F_\square$ to be healthy). For $F=F(0,0;a,b;1)$ and $F'=F(0,0;a+\gamma,b+\delta;s')$ we define
\begin{equation}
\cT(F,F')=\begin{cases}
    \cO^c(R(0,b;a,b+2))\cap\cO^c(\{(a,b)\})&(s',\gamma,\delta)=(2,0,0)\\
    \cO^c(R(a,b;a+2,b+1))\cap\cO(R(0,b;a,b+1))&(s',\gamma,\delta)=(1,0,1)\\
    \begin{aligned}&\cO^c(R(a,b;a+2,b+2))\cap\cO^c(R(0,b;a,b+1))\\&{}\cap\cO(R(0,b+1;a,b+2))\end{aligned}&(s',\gamma,\delta)=(1,0,2)\\
    \cO(R(0,b;a+1,b+1))\cap \cO(\{(a+1,b)\})&(s',\gamma,\delta)=(0,2,1)\\
    \begin{aligned}&\cO(\{(a,b)\})\cap\cO^c(\{(a+1,b)\})\\&{}\cap \cO(R(a+1,0;a+2,b+1))\end{aligned}&(s',\gamma,\delta)=(0,3,1)\\    \begin{aligned}&\cO^c(R(0,b;a+1,b+1))\cap\cO(R(0,b+1;a,b+2))\\
    &{}\cap \cO(R(a+1,b;a+2,b+2))\end{aligned}&(s',\gamma,\delta)=(0,2,2)\\
    \begin{aligned}&\cO^c(R(0,b;a+1,b+1))\cap\cO(R(0,b+1;a,b+2))\\
    &{}\cap\cO^c(R(a+1,b;a+2,b+2))\cap \cO(\{(a,b+1)\})\\
    &{}\cap\cO(R(a+1,0;a+2,b+2))\end{aligned}&(s',\gamma,\delta)=(0,3,2)
\end{cases}
\label{eq:def:transitions:2n}
\end{equation}
We note that for $F$ with frame state $1$ the only transitions we are omitting (as they are not required for the upper bound) correspond to $(s',\gamma,\delta)\in\{(1,1,1),(1,1,2)\}$, but for frame states $2,2'$ there are tens of them and for frame state $3$ almost a hundred. 

As in \cref{obs:transitions}, we have the following.
\begin{observation}
\label{obs:transitions:2n}
\Cref{tab:transitions:2n} defines a sub-stochastic matrix $\cT$. For any framed rectangle $F$, the events $\cT(F,F')$ (see \cref{eq:def:transitions:2n} corresponding to transitions present in $\cT$ are disjoint for different choices of $F'$. For each transition $\cT(F,F')$ we have $F_\circ\subset F'_\circ$ and $F_\blacksquare\subsetneq F'_\blacksquare$, unless $F$ has buffer state 4. If $F=F(a,b;c,d;s)$ and $F'=F(a',b';c',d';s)$, then
\[\cT(F,F')\subset\left\{A:(A\setminus F_\blacksquare)\in\left(\cC^{\mathrm F}(R(a,b;c,d),R(a',b';c',d'))\cap \cO^c(F'_\square)\right)\right\}.\]
The transition events $\cT(F,F')$ define a sub-Markovian process $(\cF(t))_{t\in\bbN}$ with transition matrix $\cT$, whose state $\cF(t)$ is measurable with respect to $A\cap \cF(t)_\blacksquare$.\footnote{For the reader's convenience, we recall that such a sub-Markovian process may be viewed simply as a Markov chain with an additional absorbing cemetery state.}

Given a rectangle $S$, we denote by $(\cF^S(t))_{t\in\bbN}$ this process with initial state $S$ with frame state 0. If $\cI_{\mathrm{loc}}(S)$ occurs, then for all $t\ge 0$, $\cI_{\mathrm{loc}}(\cF^S_\circ(t))\cap \cO^c(\cF^S_\square(t))$ occurs. We denote $\cF^S(\infty)=\bigcup_{t\ge 0}\cF_\circ^S(t)$. Then for any $x\in\bbZ^2$, either $\cF^{\{x\}}(\infty)=[A]^x$, in case $[A]^x$ is finite, or both $[A]^x$ and $\cF^{\{x\}}(\infty)$ are infinite (although not necessarily equal).
\end{observation}

Similarly, \cref{cor:Markov} becomes the following.
\begin{corollary}
\label{cor:Markov:2n}
Given nested rectangles $S\subset R=R(a,b)$, we have
\[\bbP_p\left(\cC(S,R)\right)\ge \frac{\bbP_p(\cF^S(\infty)=R)}{e^{-4(a+b+1)q}}.\]
\end{corollary}

\begin{table}
    \centering
    \small
\begin{tabular}{c c c c c c | l}
$s$ & $s'$ & $\alpha$ & $\beta$ & $\gamma$ & $\delta$ & $-\log \bbP_p(\cT(F(0,0;a,b;s),F(-\alpha,-\beta;a+\gamma,b+\delta;s')))$\\\hline\hline
0&1&0&0&0&0&$2qb$\\
1&2&0&0&0&0&$q(2a+1)$\\
2&3&0&0&0&0&$q(2b+1)$\\
3&4&0&0&0&0&$q(2a+2)$\\
2'&3&0&0&0&0&$q(2b+1)$\\
1'&2'&0&0&0&0&$q(2a+1)$\\\hline
0&0&0&0&1&0&$f(qb)$\\
0&0&0&0&2&0&$f(qb)+qb$\\
1&1&0&0&0&1&$f(qa)+2q$\\
1&1&0&0&0&2&$f(qa)+4q+qa$\\
2&2&1&0&0&0&$f(qb)+2q$\\
2&2&2&0&0&0&$f(qb)+4q+qb$\\
3&3&0&1&0&0&$f(qa)+4q$\\
3&3&0&2&0&0&$f(qa)+8q+qa$\\
2'&2'&0&0&1&0&$f(qb)+2q$\\
2'&2'&0&0&2&0&$f(qb)+4q+qb$\\
1'&1'&0&0&0&1&$f(qa)+4q$\\
1'&1'&0&0&0&2&$f(qa)+8q+qa$\\
4&4&0&0&0&0&0\\\hline
1&0&0&0&2&1&$\log(1/p)+f(q(a+1))$\\
1&0&0&0&3&1&$\log(1/p)+f(q(b+1))+q$\\
1&0&0&0&2&2&$f(2q)+f(qa)+q(a+1)$\\
1&0&0&0&3&2&$\log(1/p)+f(qa)+f(q(b+2))+q(a+3)$\\
1'&0&2&0&0&1&$\log(1/p)+f(q(a+1))$\\
1'&0&3&0&0&1&$\log(1/p)+f(q(b+1))+q$\\
1'&0&2&0&0&2&$f(2q)+f(qa)+q(a+1)$\\
1'&0&3&0&0&2&$\log(1/p)+f(qa)+f(q(b+2))+q(a+3)$\\
2&1&1&0&0&2&$\log(1/p)+f(q(b+1))+3q$\\
2&1&1&0&0&3&$\log(1/p)+f(q(a+1))+6q$\\
2&1&2&0&0&2&$f(2q)+f(qb)+q(b+4)$\\
2&1&2&0&0&3&$\log(1/p)+f(qb)+f(q(a+2))+q(b+8)$\\
2'&1'&0&0&1&2&$\log(1/p)+f(q(b+1))+3q$\\
2'&1'&0&0&1&3&$\log(1/p)+f(q(a+1))+6q$\\
2'&1'&0&0&2&2&$f(2q)+f(qb)+q(b+4)$\\
2'&1'&0&0&2&3&$\log(1/p)+f(qb)+f(q(a+2))+q(b+8)$\\
3&2'&0&1&2&0&$\log(1/p)+f(q(a+1))+5q$\\
3&2'&0&1&3&0&$\log(1/p)+f(q(b+1))+8q$\\
3&2'&0&2&2&0&$f(2q)+f(qa)+q(a+8)$\\
3&2'&0&2&3&0&$\log(1/p)+f(qa)+f(q(b+2))+q(a+12)$\\
3&2&2&1&0&0&$\log(1/p)+f(q(a+1))+5q$\\
3&2&3&1&0&0&$\log(1/p)+f(q(b+1))+8q$\\
3&2&2&2&0&0&$f(2q)+f(qa)+q(a+8)$\\
3&2&3&2&0&0&$\log(1/p)+f(qa)+f(q(b+2))+q(a+12)$
\end{tabular}
\caption{Transition probabilities for the framed rectangle (sub-)Markov chain for local two-neighbour bootstrap percolation. The table only features the transitions which actually contribute to the final result---\cref{th:2n}.}
\label{tab:transitions:2n}
\end{table}

\subsection{Lower bound of (\ref{eq:2n:main})}
We are now ready to proceed to proving the lower bound of \cref{eq:2n:main}, which gives the upper bound of \cref{th:2n}. \cref{subsec:upper:reduction} requires essentially no change. The only difference is that we start our sequence with $a^{(0)}=q^{-2/3}$, while $a^{(M)}=\Lambda$ remains unchanged (once again, we assume for simplicity that these are both integers and that an integer $M$ is compatible with the recurrence relation $a^{(i+1)}-a^{(i)}=\lceil (a^{(i)})^{2/3}\rceil$). This makes the error term in \cref{eq:upper:3} of order $\sqrt[6]q$ instead, but does not impact the final result. Furthermore, \cref{eq:R0:upper} is replaced by 
\[\bbP_p\left(\cI_{\mathrm{loc}}\left(R^{(0)}\right)\right)\ge p^3\exp\left(-W_p\left(\gamma_{R^{(0)}}\right)\right),\]
as provided by \cref{prop:Holroyd:lower}. It therefore remains to prove the following analogue of \cref{prop:upper:main} by applying the method of \cref{subsec:upper:step}.
\begin{proposition}[Coarse step lower bound]
\label{prop:upper:main:2n}
Let $a\ge a^{(0)}$ and $b=a+\lceil a^{2/3}\rceil\le \Lambda$. Then
\[\frac{\bbP_p(\cI_{\mathrm{loc}}(R(b,b)))}{\bbP_p(\cI_{\mathrm{loc}}(R(a,a)))}\ge \exp\left(-\frac2q\int_{aq}^{bq}g+2h_2(aq)(b-a)\sqrt q-C_1^2\log(1/q)\right).\]
\end{proposition}

\begin{proof}
 Fix $a\ge a^{(0)}$ and $b=a+\lceil a^{2/3}\rceil\le \Lambda$ and recall \cref{eq:def:M}. Since we have already restricted $\cT$ to the relevant transitions, we may directly proceed to the analogue of \cref{lem:single:upper}.
\begin{lemma}[Probability of a single transition]
\label{lem:single:upper:2n}
Let $(F,F')$ be a transition in \cref{tab:transitions:2n} with $F=F(i,j;c,d;s)$ and $F'=F(i-\alpha,j-\beta;c+\gamma,d+\delta;s')$, such that $\{c-i,d-j\}\subset [a,b]$. If $s=1\neq s'$ and $\alpha=\beta=0$, then
\[\bbP_p\left(\cT(F,F')\right)=e^{-\gamma g(q(d-j))-\delta g(q(c-i+\gamma))}\times\cM_{\gamma,\delta}(aq)\times \begin{cases}
    e^{O(a^{-1/3})}&s'=0,\\
    e^{O(qa^{2/3})}&s'=2.
\end{cases}\]
The same expressions apply for $s\in\{0,1',2,2',3\}$ up to symmetry.
\end{lemma}
The proof of \cref{lem:single:upper:2n} is analogous to the one of \cref{lem:single:upper} and therefore omitted. However, note that loops were left out. Indeed we need to pay special attention to sequences of loops, since there are now two loops per frame state rather than a single one. This is dealt with by the following lemma. 

\begin{lemma}[Loop sequences]
\label{lem:loops}
Let $F=F(0,0;c,d;1)$ and $F'=F(0,0;c,d+\ell;1)$ be framed rectangles with $\ell\ge C_1\log(1/q)\max(1,1/\sqrt{cq})$ and $c\in[a,b]$. Then
\[\cT(F,F'):=\bbP_p\left(\exists t\ge 1,\cF(t)=F'|\cF(0)=F\right)\ge e^{-\ell g(cq)}\times \alpha(aq)\times e^{-O(q\ell+\sqrt qa^{1/6})}.\]
Analogous statements hold for other buffer states.
\end{lemma}
\begin{proof}
Recalling \cref{obs:stacking,eq:def:transitions:2n}, we obtain
\begin{align*}
\bbP_p\left(\exists t\ge 1,\cF(t)=F'|\cF(0)=F\right)&{}=\bbP_p\left(\cO^c(F'_\square\setminus F_\square)\right)\bbP_p\left(\cT_\uparrow(R(c,\ell))\right)\\
&{}\ge \alpha(cq)e^{-\ell g(cq)}\left(1-\left(\frac{|\bar\beta(e^{-f(cq)})|}{\beta(e^{-f(cq)})}\right)^{\ell+1}\right)e^{-O(q\ell)}\\
&{}\ge \alpha(cq)e^{-\ell g(cq)}e^{-O(q\ell)},
\end{align*}
using \cref{lem:traversability:exact} in the first inequality and $\ell\ge C_1\log(1/q)\max (1,1/\sqrt {cq})$ and the fact that $|\bar\beta(u)|/\beta(u)\le 1-\sqrt{u}/2$ for any $u\in[0,1]$ in the second one (recall \cref{eq:def:beta:bar,eq:def:beta}). Moreover,
\begin{align*}
\alpha(cq)&{}\ge \alpha(aq)\exp\left(-qO(b-a)\max_{z\in[aq,bq]}\left|(\log\alpha)'(z)\right|\right)\\
&{}=\alpha(aq)\exp\left(-qO(b-a)/\sqrt{aq}\right)=\alpha(aq)\exp\left(-O\left(\sqrt q a^{1/6}\right)\right).\qedhere\end{align*}
\end{proof}

Consider a possible trajectory $(F(t))_{t\ge 0}$ of the (sub-Markovian) chain $(\cF^{R(a,a)}(t))_{t\ge 0}$. Let $(F^{(t)})_{t\ge 0}$ be the sequence obtained by omitting any $F(t)$ such that the previous and next transitions are both loops. We call this sequence a \emph{reduced trajectory}. Similarly, we denote the (random) reduced trajectory associated to $(\cF^{R(a,a)}(t))_{t\ge 0}$ by $(\cF^{(t)})_{t\ge 0}$. Note that the reduced trajectory becomes finite, if $\cF^{R(a,a)}(\infty)$ is finite. 

\begin{definition}[Good reduced trajectory]
We say that a finite reduced trajectory $(F^{(t)})_{t=0}^{4K+8}$ is \emph{good} if it satisfies the following conditions. For each $t\in[0,2K+3]$, 
\begin{itemize}
\item $F^{(2t)}$ and $F^{(2t+1)}$ have the same frame state,
\item $F^{(2t+1)}$ and $F^{(2t+2)}$ have distinct frame states,
\item $\phi(F^{(2t+1)})-\phi(F^{(2t)})\ge C_1\log(1/q)\max(1,1/\sqrt{aq})$.
\end{itemize}
\end{definition}

Assembling \cref{lem:loops,lem:single:upper:2n}, we obtain that for any fixed good finite reduced trajectory $(F^{(t)})_{t=0}^{4K+8}$ with $F^{(4K+8)}$ of frame state $4$, we have\footnote{Note that we use the notation $\cT$ from \cref{tab:transitions:2n} and \cref{lem:single:upper:2n} for even and odd $t$ respectively.}
\begin{multline*}
\bbP_p\left(\forall t\le 4K+8, \cF^{(t)}=F^{(t)}\right)
=\prod_{t=1}^{4K+8}\bbP_p\left(\cT\left(F^{(t-1)},F^{(t)}\right)\right)\\
\ge e^{-W_p(\gamma)}\times \prod_{t=0}^{2K+3}(\alpha(aq)\cM_{\gamma_t,\delta_t}(aq))\times e^{-O(q(b-a)+(K+1)(qa^{2/3}+a^{-1/3}+\sqrt qa^{1/6}))}\end{multline*}
where $\gamma$ is the piecewise linear path from $(a,a)$ to $(b,b)$ corresponding to the first order terms in \cref{lem:loops,lem:single:upper:2n} and $\gamma_t,\delta_t$ are $\gamma,\delta$ in \cref{lem:single:upper:2n} for $F=F^{(2t+1)},F'=F^{(2t+2)}$. Furthermore, \cref{lem:W:to:g} applies without change to $W_p$ and $g$. Thus,
\begin{multline*}
\bbP_p\left(\forall t\le 4K+8, \cF^{(t)}=F^{(t)}\right)\\
\ge \exp\left(-\frac2q\int_{aq}^{bq}g\right)\times\prod_{t=0}^{2K+3}\left(\alpha(aq)\cM_{\gamma_t,\delta_t}(aq)\right)\times e^{-O(1+(K+1)(qa^{2/3}+a^{-1/3}))}.\end{multline*}

We now set $K$ to be given by \cref{eq:def:K} with $h_2$ instead of $h$ (recall \cref{eq:def:h2}). The last ingredient we need is an analogue of \cref{eq:number:trajectories}. Namely, we seek to bound the number $N(\Gamma)$ of good finite reduced trajectories $(F^{(t)})_{t=0}^{4K+8}$ with $F^{(0)}=F(0,0;a,a;0)$, $F^{(4K+8)}=F(-x,-y;b-x,b-y;4)$ for any $(x,y)\in\bbZ^2$ featuring a given sequence $\Gamma=(\gamma_t,\delta_t)_{t=0}^{2K+3}$. Note that $\gamma_t$ for odd $t$ and $\delta_t$ for even $t$ correspond to extending rectangles in the horizontal directions, while the remaining $\gamma_t$ and $\delta_t$ are vertical extensions. Let 
\begin{align*}
\alpha&{}=\sum_{t=0}^{K+1}\delta_{2t}+\gamma_{2t+1},&\beta&{}=\sum_{t=0}^{K+1}\gamma_{2t}+\delta_{2t+1}\end{align*}
and $m=\lceil C_1\log(1/q)\max(1,1/\sqrt{aq})\rceil\le 2C_1\log(1/q)a^{-1/4}$ (since $a\ge a^{(0)}$). Then 
\begin{align*}
N(\Gamma)&{}=\binom{b-a-\alpha-m(K+2)+(K+1)}{K+1}\binom{b-a-\beta-m(K+2)+(K+1)}{K+1}\\
&{}\ge \binom{b-a-2mK}{K}^2.\end{align*}
Proceeding as in \cref{eq:IFloc:ratio,eq:Matrix:bound,eq:upper:conclusion} we conclude that
\begin{multline*}
\frac{\bbP_p(\cI_{\mathrm{loc}}(R(a,a)))}{\bbP_p(\cI_{\mathrm{loc}}(R(b,b)))}\\
\ge \exp\left(-\frac2q\int_{aq}^{bq}g\right)\times e^{2h_2(aq)(b-a)\sqrt q}\times e^{-O(C_1)\log(1/q)}\left(1-\frac{O(mK)}{b-a}\right)^{O(K)}.\end{multline*}
This concludes the proof of \cref{prop:upper:main:2n}, since, recalling \cref{eq:def:K} gives
\[\frac{mK^2}{b-a}\le \frac{O(C_1)\log(1/q)a^{-1/4+2/6}}{a^{2/3}}\le a^{-1/13}\le q^{1/20},\]
so that 
\[\left(1-\frac{O(mK)}{b-a}\right)^{O(K)}\ge 1/2.\qedhere\]
\end{proof}

\bibliographystyle{plain}
\bibliography{Bib}
\end{document}